\documentclass[12pt]{amsart}
\usepackage{color,psfrag,epsfig,url,hhline
}
\numberwithin{figure}{section}
\numberwithin{table}{section}
\newtheorem{theorem}{Theorem}[section]

\newtheorem{lemma}[theorem]{Lemma}
\newtheorem{prop}[theorem]{Proposition}

\theoremstyle{definition}
\newtheorem{definition}[theorem]{Definition}

\newtheorem{cor}[theorem]{Corollary}

\theoremstyle{remark}
\newtheorem{remark}[theorem]{Remark}
\numberwithin{equation}{section}
\newfont{\tap}{tap scaled 650}

\def \R{{\mathbb R}}
\def \g{{\mathfrak g}}
\def \h{{\mathfrak h}}
\def \Z{{\mathbb Z}}
\def \S{{\mathcal S}}
\def \O{{\mathcal O}}

\def \Q{{\mathbb Q}}

\def \A{{\mathcal A}}
\def \[{[ }
\def \]{] }
\def \M{{\mathfrak M}}
\def \T{{\mathcal T}}
\def \P{\mathbb{P}}
\def\F{\mathcal{F}}

\def\t{\widetilde}
\def\mr{\mathrm}
\def\h{\widehat}
\def\o{\overline}
\def\p{\partial}
\def\g{\gamma}
\def\D{\mathsf{D}}
\def\x{\mathbf{x}}
\def\pp{\mathbf{p}}
\def\cc{\mathbf{c}}

\def\PP{{\mathbb P}}
\def\ZZ{{\mathbb Z}}

\def \B{{\mathcal B}}

\begin{document}

\title[Cluster algebras and triangulated orbifolds]{Cluster algebras and triangulated orbifolds}
\author{Anna Felikson}
\address{Independent University of Moscow, B. Vlassievskii 11, 119002 Moscow, Russia}
\curraddr{School of Engineering and Science, Jacobs University Bremen, Campus Ring 1, D-28759, Germany}
\email{a.felikson@jacobs-university.de}
\thanks{Research was supported in part by grants DFG FE-1241/2 (A.F.), DNS 0800671 (M.S.) and RFBR 11-01-00289-a (P.T.)}

\author{Michael Shapiro}
\address{Department of Mathematics, Michigan State University, East Lansing, MI 48824, USA}
\email{mshapiro@math.msu.edu}

\author{Pavel Tumarkin} 
\address{Department of Mathematical Sciences, Durham University, Science Laboratories, South Road, Durham, DH1 3LE, UK}
\email{pavel.tumarkin@durham.ac.uk}

\begin{abstract}
We construct geometric realizations for non-exceptional mutation-finite cluster algebras by extending the theory of Fomin and Thurston~\cite{FT} to skew-symmetrizable case. Cluster variables for these algebras are renormalized lambda lengths on certain hyperbolic orbifolds. We also compute the growth rate of these cluster algebras, provide the positivity of Laurent expansions of cluster variables, and prove the sign-coherence of $\cc$-vectors.

\end{abstract}

\maketitle
\setcounter{tocdepth}{1}
\tableofcontents

\section{Introduction}
\label{intro}
We continue investigation of cluster algebras of finite mutation type started in~\cite{FST1} and~\cite{FST2}.

In~\cite{FST1}, we classified all the \emph{skew-symmetric} exchange matrices with finite mutation class. It occurs that all but eleven exceptional mutation classes of skew-symmetric exchange matrices of rank at least $3$ can be obtained from triangulated marked bordered surfaces via construction provided by Fomin, Shapiro and Thurston~\cite{FST}.

In~\cite{FST2}, we completed classification of finite mutation classes of exchange matrices by extending the combinatorial technique of~\cite{FST} to general (i.e., {\it skew-symmetrizable}) case. All but several exceptional finite mutation classes consist of so called {\it s-decomposable} exchange matrices (the precise definitions will be given below).

In this paper, we relate non-exceptional mutation-finite cluster algebras to triangulated orbifolds. Extending the technique of Fomin and Thurston~\cite{FT} to skew-symmetrizable case, we construct geometric realizations for algebras with s-decomposable exchange matrices. In these realizations,
 (tagged) triangulations of certain orbifolds form clusters with { (modified)} lambda lengths of arcs serving as cluster variables.
The geometric realization provides various structural results, for example, we prove that the exchange graph in a cluster algebra with s-decomposable exchange matrices does not depend on coefficients. 

One of the tools of~\cite{FST2} was a notion of {\it unfolding} introduced by Zelevinsky (it can be understood as a counterpart of the {\it unfolding procedure} introduced by Lusztig in~\cite{L} for generalized Cartan matrices). In particular, we construct unfoldings for a class of mutation-finite matrices. In the current paper we provide a geometric version of unfolding, and construct unfoldings for almost all mutation-finite matrices. We then use unfoldings to compute the growth rate of all cluster algebras originating from orbifolds, and for generalization of positivity results by Musiker, Schiffler and Williams~\cite{MSW} to Laurent expansions of corresponding cluster variables.

Another application of the construction is a proof of the sign-coherence for $\cc$-vectors. In~\cite{FZ4}, Fomin and Zelevinsky conjectured that all the entries of $\cc$-vectors are either nonnegative or nonpositive. This conjecture was proved for skew-symmetric cluster algebras by Derksen, Weyman and Zelevinsky~\cite{DWZ}, and for a large class of skew-symmetrizable algebras by Demonet~\cite{D}. We extend the list of algebras for which the conjecture holds by proving the sign-coherence for $\cc$-vectors for all algebras originating from orbifolds. 

\medskip

The paper is organized as follows.

In Section~\ref{cluster}, we recall necessary definitions and basic facts on cluster algebras, exchange matrices, and their diagrams.

Section~\ref{blockdecomp} is devoted to the technique of s-decomposable diagrams. We recall the basic facts and results from~\cite{FST} and~\cite{FST2}, and introduce block decompositions of matrices.

In Section~\ref{orbifolds-s}, we construct a triangulated orbifold for any s-decomposable diagram. The simplicial complex of triangulations of this orbifold coincides with exchange graph of corresponding cluster algebra. The construction is close to the similar construction of Chekhov and Mazzocco~\cite{ChM}.

{ 
In Section~\ref{sec-geom}, a geometric realization of cluster algebras  with s-decomposable exchange matrices is constructed. To do this, we proceed in a way similar to~\cite{FT}, where cluster variables were represented by modified lambda lengths of arcs of triangulations of marked bordered surfaces. However, unlike~\cite{FT}, we need to consider arcs of triangulations not of the given orbifold but of some its modification. We call this modified orbifold an {\it associated orbifold}. This associated orbifold can be constructed  for any specific s-decomposable matrix. In some special cases an associated orbifold occurs to be a regular surface.
}

In Section~\ref{sec-lam}, we generalize the notion of laminations and shear coordinates to the orbifold case. Sections~\ref{sec-opened} and~\ref{sec-teichm} are devoted to a construction of a geometric realization of cluster algebras with arbitrary coefficients. Main results are contained in Section~\ref{main}.

In Section~\ref{sec-growth}, we investigate the growth of cluster algebras with s-decomposable exchange matrices. We use orbifolds to show that the exchange graph of a cluster algebra with an s-decomposable skew-symmetrizable exchange matrix is quasi-isometric to the exchange graph of a cluster algebra with some block-decomposable skew-symmetric exchange matrix. In this way we classify the growth rate of all cluster algebras with s-decomposable exchange matrices. This gives rise to a classification of the growth of all cluster algebras~\cite{FSTT}.

In Section~\ref{unfolding-s}, we recall the definition of an unfolding of skew-symmetrizable matrices introduced by A.~Zelevinsky (personal communication), extend it to a notion of unfolding of a diagram, and recall the results of~\cite{FST2}. Section~\ref{unfolding-fin} is devoted to a construction of unfoldings of almost all mutation-finite skew-symmetrizable matrices.

In Section~\ref{sec-pos}, we prove the positivity conjecture for (almost all) cluster algebras originating from orbifolds, namely, for ones with s-decomposable exchange matrices admitting unfolding. This is done by extending the results of~\cite{MSW} to the orbifold case.

Finally, in Section~\ref{signs} we prove the sign-coherence of $\cc$-vectors for all cluster algebras originating from orbifolds.

\medskip

We would like to thank  L.~Chekhov and S.~Fomin for fruitful discussions, and the anonymous referee for valuable comments.

\section{Basics on cluster algebras}
\label{cluster}

\noindent
We briefly remind the definition of a cluster algebra.

An integer $n\times n$ matrix $B$ is called \emph{skew-symmetrizable} if there exists an
integer diagonal $n\times n$ matrix $D=diag(d_1,\dots,d_n)$,
such that the product $BD$ is a skew-symmetric matrix, i.e., $b_{ij}d_j=-b_{ji}d_i$.

Let $\PP$ be \emph{a tropical semifield } equipped with commutative multiplication $\cdot$ and addition $\oplus$. The multiplicative group of $\PP$ is \emph{a coefficient group} of cluster algebra, i.e, it is a free abelian group. $\ZZ\PP$  is the integer group ring, $\F$ is a field of rational functions in $n$ independent
variables with coefficients in the field of fractions of $\ZZ\PP$.
$\F$ is called {\it an ambient field}.

\begin{definition}
\emph{A seed} is a triple $(\x,\pp,B)$, where
\begin{itemize}
\item[-]
$\pp=(p_{x}^\pm)_{x\in\x}$, a $2n$-tuple of elements of $\PP$ is a \emph{coefficient tuple} of cluster $\x$;

\item[-]
$\x=\{x_1,\dots,x_n\}$ is a collection of algebraically independent rational functions of $n$ variables which generates $\F$ over the field of fractions of $\ZZ\PP$;

\item[-]
$B$ is a skew-symmetrizable integer matrix (\emph{exchange matrix}).
\end{itemize}
The part $\x$ of seed $(\x,\pp,B)$ is called \emph{cluster}, elements $x_i$ are called \emph{cluster variables}.

\end{definition}

\begin{definition}[seed mutation]
For any $k$, $1\le k\le n$ we define \emph{the mutation} of seed $(\x,\pp,B)$ in direction $k$
as a new seed $(\x',\pp',B')$ in the following way:
\begin{equation}\label{eq:MatrixMutation}
b'_{ij}=\left\{
           \begin{array}{ll}
             -b_{ij}, & \hbox{ if } i=k \hbox{ or } j=k; \\
             b_{ij}+\frac{|b_{ik}|b_{kj}+b_{ik}|b_{kj}|}{2}, & \hbox{ otherwise.}
           \end{array}
         \right.
\end{equation}

\begin{equation}\label{eq:ClusterMutation}
x'_i=\left\{
           \begin{array}{ll}
             x_i, & \hbox{ if } i\ne k; \\
             \frac{p^+_k\prod_{b_{jk}>0} x_j^{b_{jk}}+p^-_k\prod_{b_{ji}<0} x_j^{-b_{ji}}}{x_k}, & \hbox{ otherwise.}
           \end{array}
         \right.
\end{equation}

\begin{eqnarray}\label{eq:CoeffMutation}
  p'^\pm_k &=& p^\mp_k \\
  \hbox{ for } i\ne k\qquad p'^+_i/p'^-_i &=& \left\{
           \begin{array}{ll}
             (p^+_k)^{b_{ki}}p^+_i/p^-_i, & \hbox{ if } b_{ki}\ge 0; \\
             (p^-_k)^{b_{ki}}p^+_i/p^-_i, & \hbox{ if } b_{ki}\le 0; \\
           \end{array}
         \right.
\end{eqnarray}

\end{definition}

\noindent
We write $(\x',\pp',B')=\mu_k\left((\x,\pp,B)\right)$.
Notice that $\mu_k(\mu_k((\x,\pp,B)))=(\x,\pp,B)$.
We say that two seeds are \emph{mutation-equivalent}
if one is obtained from the other by a sequence of seed mutations.
Similarly we say that two clusters or two exchange matrices are \emph{mutation-equivalent}.

Notice that exchange matrix mutation~(\ref{eq:MatrixMutation}) depends only on the exchange matrix itself.
The collection of all matrices mutation-equivalent to a given matrix $B$ is called the \emph{mutation class} of $B$.

For any skew-symmetrizable matrix $B$ we define \emph{initial seed} $(\x,\pp,B)$ as a collection
$(\{x_1,\dots,x_n\},\{p_1^\pm,\ldots,p_n^\pm\},B)$, where $B$ is the \emph{initial exchange matrix}, $\x=\{x_1,\dots,x_n\}$ is the \emph{initial cluster}, $\pp=\{p_1^\pm,\ldots,p_n^\pm\}$ is the \emph{initial coefficient tuple}.

{\it Cluster algebra} $\A(B)$ associated with the skew-sym\-met\-ri\-zab\-le $n\times n$ matrix $B$ is a subalgebra of $\Q(x_1,\dots,x_n)$ generated by all cluster variables of the clusters mutation-equivalent
to the initial seed $(x,B)$.

Cluster algebra $\A(B)$ is called \emph{of finite type} if it contains only finitely many
cluster variables. In other words, all clusters mutation-equivalent to initial cluster contain
only finitely many distinct cluster variables in total.

\begin{definition}\label{def:FinMutType} A cluster algebra with only finitely many exchange matrices is called \emph{of finite mutation type}.
\end{definition}

\begin{remark} Since the orbit of an exchange matrix depends on the exchange matrix only, we may speak about skew-symmetrizable matrices of finite mutation type.
\end{remark}



Following~\cite{FZ2}, we encode an $n\times n$ skew-symmetrizable integer matrix $B$ by a finite simplicial $1$-complex $S$ with oriented weighted edges called {\it diagram}. The weights of a diagram are positive integers.

Vertices of $S$ are labeled by $[1,\dots,n]$. If $b_{ij}>0$, we join vertices $i$ and $j$ by an  edge directed from $i$ to $j$ and assign to this edge weight $-b_{ij}b_{ji}$. Not every diagram corresponds to a skew-symmetrizable integer matrix: given a diagram $S$ of a skew-symmetrizable integer matrix $B$, a product of weights along any chordless cycle of $S$ is a perfect square (cf.~\cite[Exercise~2.1]{Kac}).

Distinct matrices may have the same diagram. At the same time, it is easy to see that only finitely many matrices may correspond to the same diagram.
All weights of a diagram of a skew-symmetric matrix are perfect squares. Conversely, if all weights of a diagram $S$ are perfect squares, then there exists a skew-symmetric matrix $B$ with diagram $S$.

As it is shown in~\cite{FZ2}, mutations of exchange matrices induce {\it mutations of diagrams}. If $S$ is the diagram corresponding to matrix $B$, and $B'$ is a mutation of $B$ in direction $k$, then we call the diagram $S'$ associated to $B'$ a {\it mutation of $S$ in direction $k$} and denote it by $\mu_k(S)$. A mutation in direction $k$ changes weights of diagram in the way described in Figure~\ref{quivermut} (see~\cite{FZ2}).

\begin{figure}[!h]
\begin{center}
\psfrag{a}{\small $a$}
\psfrag{b}{\small $b$}
\psfrag{c}{\small $c$}
\psfrag{d}{\small $d$}
\psfrag{k}{\small $k$}
\psfrag{mu}{\small $\mu_k$}
\epsfig{file=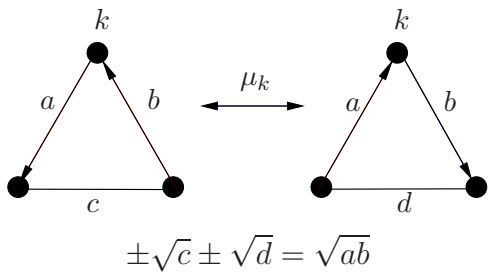,width=0.4\linewidth}\\
\medskip
$\pm\sqrt{c}\pm\sqrt{d}=\sqrt{ab}$
\caption{Mutations of diagrams. The sign before $\sqrt{c}$ (resp., $\sqrt{d}$) is positive if the three vertices form an oriented cycle, and negative otherwise. Either $c$ or $d$ may vanish. If $ab$ is equal to zero then neither value of $c$ nor orientation of the corresponding edge does change.}
\label{quivermut}

\end{center}
\end{figure}

Hence, for a given diagram, the notion of its {\it mutation class} is well-defined. We call a diagram (resp., matrix) {\it mutation-finite} if its mutation class is finite.

\medskip

\section{Block decompositions of diagrams and matrices}
\label{blockdecomp}

First, we remind the definitions from~\cite{FST} and~\cite{FST2}.

\begin{definition}
\label{def-block}
In~\cite{FST}, a {\it block} is a diagram isomorphic to one of the diagrams with black/white colored vertices shown in Fig.~\ref{bloki}, or to a single vertex. Vertices marked in white are called {\it outlets}, we call the black ones {\it dead ends}. A connected skew-symmetric diagram $S$ is called {\it block-decomposable} if it can be obtained from a collection of blocks by identifying outlets of different blocks along some partial matching (matching of outlets of the same block is not allowed), where two single edges with same endpoints and opposite directions cancel out, and two single edges with same endpoints and same directions form an edge of weight $4$. A non-connected diagram $S$ is called  block-decomposable either if $S$ satisfies the definition above, or if $S$ is a disjoint union of several diagrams (without any edge joining one to another) satisfying the definition above. If a skew-symmetric diagram $S$ is not block-decomposable then we call $S$ {\it non-decomposable}. Depending on a block, we call it {\it a block of type} $\rm{I}$, $\rm{II}$, $\rm{III}$, $\rm{IV}$, $\rm{V}$, or simply {\it a block of $n$-th type}.

\begin{figure}[!h]
\begin{center}
\psfrag{1}{${\rm{I}}$}
\psfrag{2}{${\rm{II}}$}
\psfrag{3a}{${\rm{IIIa}}$}
\psfrag{3b}{${\rm{IIIb}}$}
\psfrag{4}{${\rm{IV}}$}
\psfrag{5}{${\rm{V}}$}
\epsfig{file=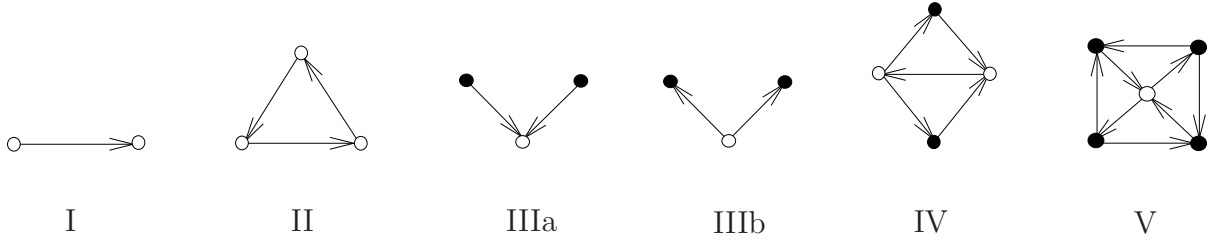,width=0.999\linewidth}
\caption{Skew-symmetric blocks. Outlets are colored in white, dead ends are black.}
\label{bloki}
\end{center}
\end{figure}

\end{definition}

Block-decomposable diagrams are in one-to-one correspondence with adjacency matrices of arcs of ideal (tagged) triangulations of bordered two-dimensional surfaces with marked points (see~\cite[Section~13]{FST} for the detailed explanations). Mutations of block-decomposable diagrams correspond to flips of (tagged) triangulations. In particular, this implies that mutation class of any block-decomposable diagram is finite, and any subdiagram of a block-decomposable one is block-decomposable too.

It is proved in~\cite{FST1} that block-decomposable diagrams almost exhaust mutation-finite ones. Namely, any mutation-finite non-decomposable skew-symmetric diagram of order at least $3$ is mutation-equivalent to one of $11$ exceptional diagrams, see~\cite[Theorem~6.1]{FST1}.

\begin{definition}
\label{def-s-block}
To adopt the technique of blocks to general (skew-symmetrizable) case, we introduce new blocks called {\it s-blocks} of types $\mr{\t{III}a}$, $\mr{\t{III}b}$, $\t{\mr{IV}}$, $\t{\mr V}_1$, $\t{\mr V}_2$, and $\t{\mr V}_{12}$ shown in Table~\ref{newblocks}, and exceptional blocks shown in Table~\ref{newblocks-e}.
\end{definition}

\begin{table}[!h]
\caption{s-blocks and their local unfoldings (see Sections~\ref{unfolding-s},~\ref{unfolding-fin}). Vertex $v_i$ and the set $E_i$ are marked in the same way. Outlets are colored white.}
\label{newblocks}
\begin{tabular}{|l|c|c|c|c|c|c|}
\hline
&&&&&&\\
\begin{tabular}{c}
\raisebox{-0.4cm}{s-blocks}
\end{tabular}
&
\psfrag{2-}{\tiny $2$}
\psfrag{3at}{\scriptsize ${\mr{\t{III}a}}$}
\parbox[c]{1.4cm}{\epsfig{file=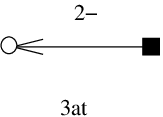,width=0.99\linewidth}}
&
\psfrag{2-}{\tiny $2$}
\psfrag{3bt}{\scriptsize ${\mr{\t{III}b}}$}
\parbox[c]{1.4cm}{\epsfig{file=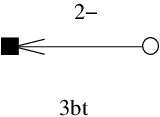,width=0.99\linewidth}}
&
\psfrag{2-}{\tiny $2$}
\psfrag{4t}{\scriptsize $\t{\mr{IV}}$}
\parbox[c]{1.4cm}{\epsfig{file=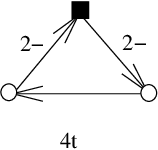,width=0.99\linewidth}}
&
\psfrag{2-}{\tiny $2$}
\psfrag{51t}{\scriptsize $\t{\mr{V}}_1$}
\parbox[c]{1.4cm}{\epsfig{file=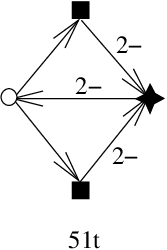,width=0.99\linewidth}}
&
\psfrag{2-}{\tiny $2$}
\psfrag{52t}{\scriptsize $\t{\mr{V}}_2$}
\parbox[c]{1.4cm}{\epsfig{file=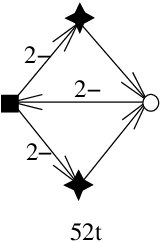,width=0.99\linewidth}}
&
\psfrag{2-}{\tiny $2$}
\psfrag{4}{\tiny $4$}
\psfrag{512t}{\scriptsize $\t{\mr{V}}_{12}$}
\parbox[c]{1.4cm}{\epsfig{file=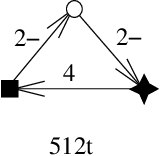,width=0.99\linewidth}}
\\
&&&&&&
\\
\hline
&&&&\multicolumn{3}{c|}{}
\\
\begin{tabular}{c}
Unfoldings
\end{tabular}
&
\parbox[c]{1.4cm}{\epsfig{file=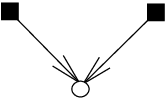,width=0.99\linewidth}}
&
\parbox[c]{1.4cm}{\epsfig{file=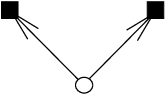,width=0.99\linewidth}}
&
\parbox[c]{1.4cm}{\epsfig{file=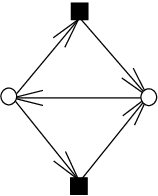,width=0.99\linewidth}}
&
\multicolumn{3}{c|}{
\parbox[c]{1.4cm}{\epsfig{file=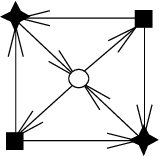,width=0.99\linewidth}}
}
\\
&&&&\multicolumn{3}{c|}{}
\\
\hline

\end{tabular}
\end{table}

\begin{table}
\caption{Exceptional s-blocks and their unfoldings. s-blocks shown in the table have no outlets, so they cannot be used to construct other s-decomposable diagrams} 
\label{newblocks-e}
\begin{tabular}{|l|c|c|c|}
\hline
&&&\\
\begin{tabular}{c}
\raisebox{-0.4cm}{s-blocks}
\end{tabular}
&
\psfrag{2-}{\tiny $2$}
\parbox[c]{1.9cm}{\epsfig{file=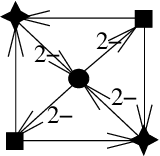,width=0.99\linewidth}}
&
\psfrag{2-}{\tiny $2$}
\psfrag{4-}{\tiny $4$}
\parbox[c]{1.9cm}{\epsfig{file=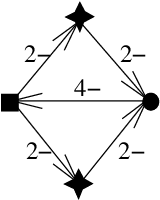,width=0.99\linewidth}}
&
\psfrag{4-}{\tiny $4$}
\psfrag{4}{\tiny $4$}
\parbox[c]{1.9cm}{\epsfig{file=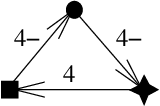,width=0.99\linewidth}}
\\
&&&
\\
\hline
&\multicolumn{3}{c|}{}
\\
\begin{tabular}{c}
Unfoldings
\end{tabular}
&
\multicolumn{3}{c|}{
\parbox[c]{2.6cm}{\epsfig{file=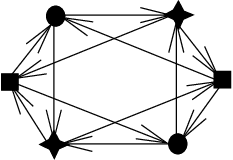,width=0.99\linewidth}}
}
\\
&\multicolumn{3}{c|}{}
\\
\hline
\end{tabular}
\end{table}

\begin{definition}
\label{s-dec diagr}
A diagram is called {\it s-decomposable} if it can be obtained from a collection of blocks and s-blocks according to the same rules as block-decomposable diagram (the way of identification remains well-defined since any edge with two white ends has weight one).
We keep the term ``block-decomposable'' for s-decomposable diagrams corresponding to skew-symmetric matrices. A diagram called {\it non-decomposable} if it is not s-decomposable.

\end{definition}

It is proved in~\cite{FST2} that any mutation-finite non-decomposable diagram of order at least $3$ is either skew-symmetric or mutation-equivalent to one of $7$ exceptional diagrams, see~\cite[Theorem~5.13]{FST2}.

\begin{remark}
The exceptional s-blocks shown in Table~\ref{newblocks-e} have no outlets, so they cannot be used in constructing other s-decomposable diagrams. However, they are mutation-finite, cannot be decomposed into other blocks and s-blocks, and can be constructed as diagrams of triangulations of some orbifolds (see Table~\ref{all-matrices-e}), so we call them {\em s-blocks} for completeness of the theory. 

\end{remark}

Now we can define s-decomposable matrices.

\begin{definition}
\label{s-dec matr}
A skew-symmetrizable matrix is {\it s-decomposable} (respectively, {\it block-de\-com\-posable}) if its diagram is {s-decomposable} (respectively, {block-decomposable})

\end{definition}

Block-decomposable (or s-decomposable) matrices can be indeed decomposed into blocks in the following way. Let $B$ be an s-decomposable $n\times n$ matrix with diagram $\D$. For every block $\B_j$ in $\D$ spanned by vertices $v_{i_1},\dots,v_{i_k}$ consider the following $n\times n$ matrix $B_j$: the matrix corresponding to the block (see Tables~\ref{blocks} and~\ref{all-matrices}) is located on $(i_1,\dots,i_k)$-places of $B_j$, and the other entries are zeros. Then $B$ is the sum of all matrices $B_j$ for all blocks $\B_j$.

\begin{table}
\caption{Skew-symmetric blocks and the corresponding surfaces}
\label{blocks}
\begin{tabular}{|c|c|c|c|}
\hline
Block
&Diagram&Matrix&Surface\\
\hline
\raisebox{5mm}{$\mr{{I}}$}&
\raisebox{5mm}{\epsfig{file=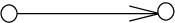,width=0.09\linewidth}}&
\raisebox{5mm}{\small $\left(\begin{smallmatrix}
0&1\\
-1&0\\
\end{smallmatrix}\right)$}
&\raisebox{-2mm}[18mm][4mm]{\psfrag{boundary}{\tiny\tap boundary}\epsfig{file=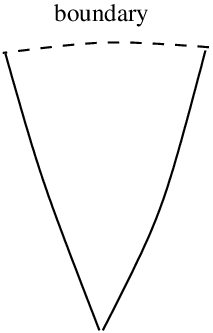,width=0.073\linewidth}}
\\
\hline
\raisebox{4mm}{$\mr{{II}}$}&
\raisebox{1mm}{\epsfig{file=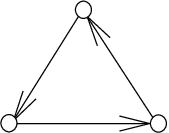,width=0.09\linewidth}}&
\raisebox{4mm}{\small $\left(\begin{smallmatrix}
0&1&-1\\
-1&0&1\\
1&-1&0\\
\end{smallmatrix}\right)$}&
\raisebox{-1mm}[16mm][3mm]{\epsfig{file=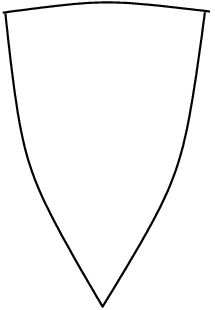,width=0.068\linewidth}}
\\
\hline
\raisebox{5mm}{$\mr{{III}a}$}&
\psfrag{3a}{}\raisebox{5mm}{\epsfig{file=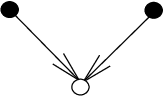,width=0.09\linewidth}}&
\raisebox{8mm}{\small $\left(\begin{smallmatrix}
0&-1&-1\\
1&0&0\\
1&0&0\\
\end{smallmatrix}\right)$}&
\psfrag{u}{\tiny }
\psfrag{v1}{\tiny }
\psfrag{v2}{\tiny }
\raisebox{-1mm}[20mm][3mm]{\epsfig{file=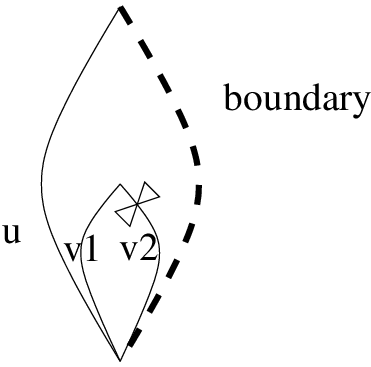,width=0.12\linewidth}}
\\
\hline
\raisebox{5mm}{$\mr{IIIb}$}&
\psfrag{3b}{}\raisebox{5mm}{\epsfig{file=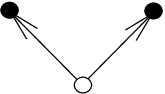,width=0.09\linewidth}}&
\raisebox{8mm}{\small $\left(\begin{smallmatrix}
0&0&-1\\
0&0&-1\\
1&1&0\\
\end{smallmatrix}\right)$}&
\psfrag{u}{\tiny }
\psfrag{v1}{\tiny }
\psfrag{v2}{\tiny }
\raisebox{-1mm}[20mm][3mm]{\epsfig{file=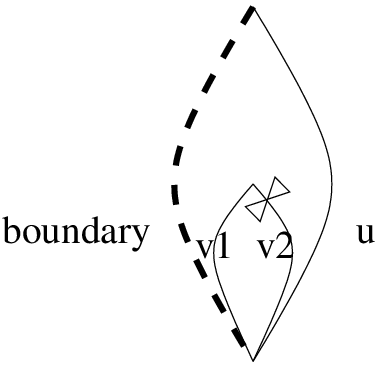,width=0.121\linewidth}}
\\
\hline
\raisebox{6mm}{$\mr{{IV}}$}&
\psfrag{4}{}\raisebox{2mm}{\epsfig{file=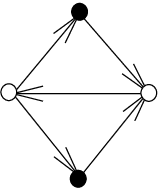,width=0.09\linewidth}}&
\raisebox{9mm}{\small $\left(\begin{smallmatrix}
0&1&-1&-1\\
-1&0&1&1\\
1&-1&0&0\\
1&-1&0&0\\
\end{smallmatrix}\right)$}&
\psfrag{w}{}
\psfrag{u}{}
\psfrag{v1}{}
\psfrag{v2}{}
\psfrag{p1}{}
\psfrag{p2}{}
\raisebox{-0.5mm}[23mm][3mm]{\epsfig{file=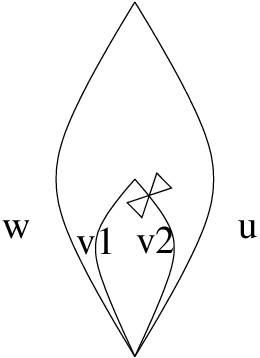,width=0.095\linewidth}}
\\
\hline
\raisebox{6mm}{$\mr{{V}}$}&
\psfrag{5}{}\raisebox{1.5mm}{\epsfig{file=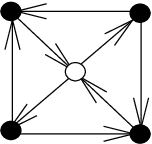,width=0.09\linewidth}}&
\raisebox{7mm}{\small $\left(\begin{smallmatrix}
0&1&-1&-1&1\\
-1&0&1&1&0\\
1&-1&0&0&-1\\
1&-1&0&0&-1\\
-1&0&1&1&0
\end{smallmatrix}\right)$}&
\psfrag{u}{}
\psfrag{w1}{}
\psfrag{w2}{}
\psfrag{p1}{}
\psfrag{p2}{}
\raisebox{-1mm}{\epsfig{file=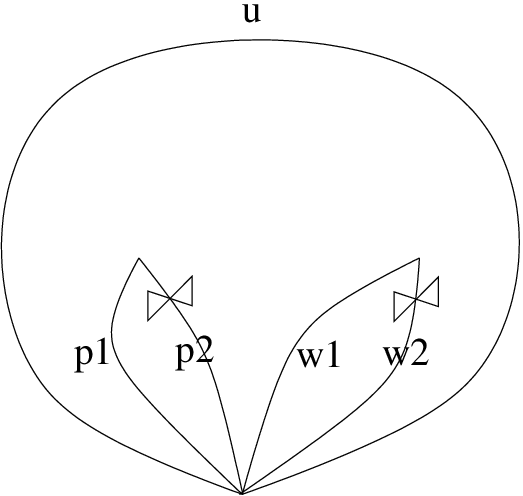,width=0.12\linewidth}}
\\
\hline
\end{tabular}
\end{table}

\begin{table}
\caption{s-blocks, their matrices, orbifolds, and associated orbifolds, see Sections~\ref{sec lambda 2} and~\ref{sec_gen}}
\label{all-matrices}
\begin{tabular}{|c|c|c|c|c|}
\hline
s-Block
&Diagram&Orbifold&Matrix
&Associated
orbifold
\\
\hline
\raisebox{-4mm}{$\mr{\t{III}a}$}&
\psfrag{2-}{\tiny $2$}\psfrag{3at}{}\raisebox{-10mm}[0mm][0mm]{\epsfig{file=diagrams_pic/block3at.eps,width=0.09\linewidth}}&
\raisebox{-10mm}[0mm][0mm]{\epsfig{file=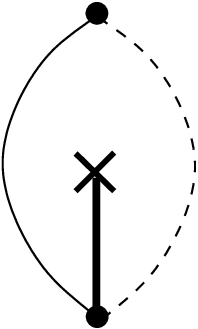,width=0.05\linewidth}}&
\raisebox{2mm}[9mm][5mm]{\small $\left(\begin{smallmatrix}
0&-1\\
2&0\\
\end{smallmatrix}\right)$}&
\raisebox{-4mm}[0mm][0mm]{\epsfig{file=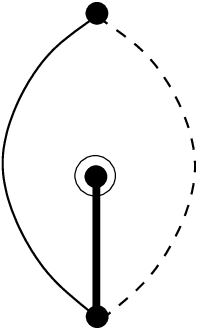,width=0.05\linewidth}}\\
\hhline{|~|~|~|-|-|}
&&&
\raisebox{5mm}{\small $\left(\begin{smallmatrix}
0&-2\\
1&0\\
\end{smallmatrix}\right)$}&
\raisebox{0mm}{\epsfig{file=diagrams_pic/1.eps,width=0.05\linewidth}}
\\
\hline
\raisebox{-4mm}{$\mr{\t{III}b}$}&
\psfrag{2-}{\tiny $2$}\psfrag{3bt}{}\raisebox{-10mm}[0mm][0mm]{\epsfig{file=diagrams_pic/block3bt.eps,width=0.09\linewidth}}&
\raisebox{-10mm}[0mm][0mm]{\epsfig{file=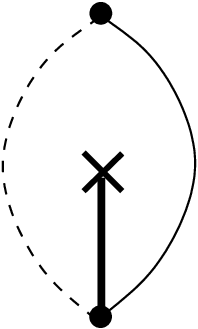,width=0.05\linewidth}}&
\raisebox{2mm}[9mm][5mm]{\small $\left(\begin{smallmatrix}
0&-2\\
1&0\\
\end{smallmatrix}\right)$}&
\raisebox{-4mm}[0mm][0mm]{\epsfig{file=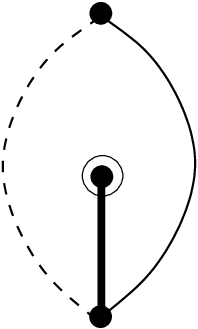,width=0.05\linewidth}}\\
\hhline{|~|~|~|-|-|}
&&&
\raisebox{5mm}{\small $\left(\begin{smallmatrix}
0&-1\\
2&0\\
\end{smallmatrix}\right)$}&
\raisebox{0mm}{\epsfig{file=diagrams_pic/2.eps,width=0.05\linewidth}}
\\
\hline
\raisebox{-4mm}{$\mr{\t{IV}}$}&
\psfrag{2-}{\tiny $2$}\psfrag{4t}{}\raisebox{-10mm}[0mm][0mm]{\epsfig{file=diagrams_pic/block4t.eps,width=0.09\linewidth}}&
\raisebox{-10mm}[0mm][0mm]{\epsfig{file=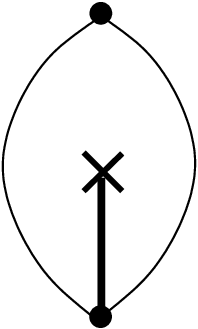,width=0.05\linewidth}}&
\raisebox{2mm}[9mm][5mm]{\small $\left(\begin{smallmatrix}
0&1&-1\\
-1&0&1\\
2&-2&0\\
\end{smallmatrix}\right)$}&
\raisebox{-4mm}[0mm][0mm]{\epsfig{file=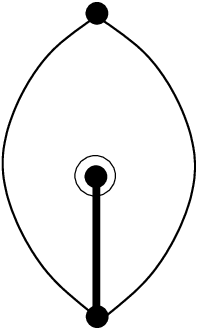,width=0.05\linewidth}}\\
\hhline{|~|~|~|-|-|}
&&&
\raisebox{5mm}{\small $\left(\begin{smallmatrix}
0&1&-2\\
-1&0&2\\
1&-1&0\\\end{smallmatrix}\right)$}&
\raisebox{0mm}{\epsfig{file=diagrams_pic/3.eps,width=0.05\linewidth}}\\
\hline
\raisebox{-4mm}{$\mr{\t{V}_1}$}&
\psfrag{2-}{\tiny $2$}\psfrag{51t}{}\raisebox{-17mm}[0mm][0mm]{\epsfig{file=diagrams_pic/block51t.eps,width=0.09\linewidth}}&
\raisebox{-10mm}[0mm][0mm]{\epsfig{file=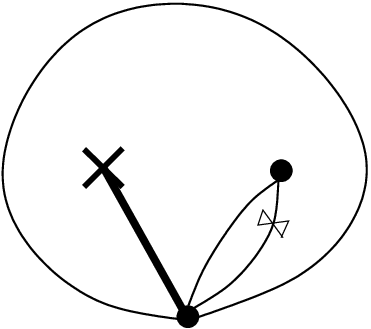,width=0.07\linewidth}}&
\raisebox{0mm}[8mm][5mm]{\small $\left(\begin{smallmatrix}
0&1&-1&1\\
-1&0&1&0\\
2&-2&0&-2\\
-1&0&1&0\\
\end{smallmatrix}\right)$}&
\raisebox{-4mm}[0mm][0mm]{\epsfig{file=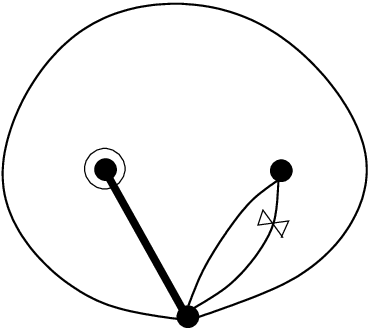,width=0.07\linewidth}}\\
\hhline{|~|~|~|-|-|}
&&&
\raisebox{5mm}{\small $\left(\begin{smallmatrix}
0&1&-2&1\\
-1&0&2&0\\
1&-1&0&-1\\
-1&0&2&0\\
\end{smallmatrix}\right)$}&
\raisebox{0mm}{\epsfig{file=diagrams_pic/4.eps,width=0.07\linewidth}}
\\
\hline
\raisebox{-4mm}{$\mr{\t{V}_2}$}&
\psfrag{2-}{\tiny $2$}\psfrag{52t}{}\raisebox{-17mm}[0mm][0mm]{\epsfig{file=diagrams_pic/block52t.eps,width=0.09\linewidth}}&
\raisebox{-10mm}[0mm][0mm]{\epsfig{file=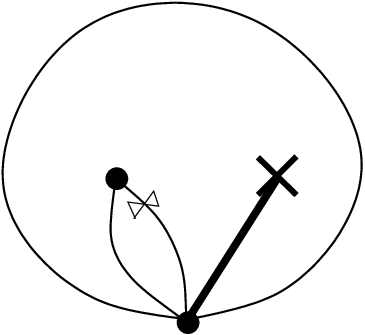,width=0.07\linewidth}}&
\raisebox{0mm}[8mm][5mm]{\small $\left(\begin{smallmatrix}
0&2&-2&2\\
-1&0&1&0\\
1&-1&0&-1\\
-1&0&1&0\\
\end{smallmatrix}\right)$}&
\raisebox{-4mm}[0mm][0mm]{\epsfig{file=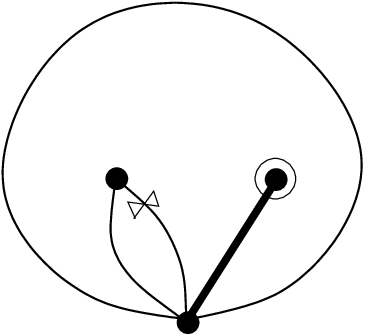,width=0.07\linewidth}}\\
\hhline{|~|~|~|-|-|}
&&&
\raisebox{4mm}{\small $\left(\begin{smallmatrix}
0&1&-1&1\\
-2&0&1&0\\
2&-1&0&-1\\
-2&0&1&0\\
\end{smallmatrix}\right)$}&
\raisebox{0mm}{\epsfig{file=diagrams_pic/5.eps,width=0.07\linewidth}}
\\
\hline
\raisebox{-10mm}[10mm][0mm]{$\mr{\t{V}_{12}}$}&
\psfrag{2-}{\tiny $2$}\psfrag{4}{\tiny $4$}\psfrag{512t}{}\raisebox{-20mm}[0mm][0mm]{\epsfig{file=diagrams_pic/block512t.eps,width=0.09\linewidth}}&
\raisebox{-15mm}[0mm][0mm]{\epsfig{file=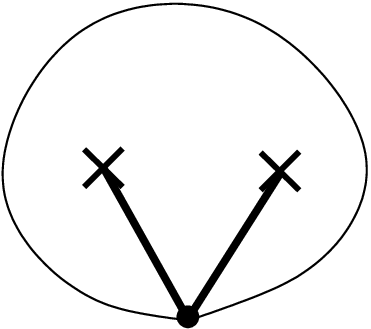,width=0.07\linewidth}}&
\raisebox{2mm}[7mm][1mm]{\small $\left(\begin{smallmatrix}
0&2&-2\\
-1&0&1\\
2&-2&0\\
\end{smallmatrix}\right)$}&
\raisebox{0mm}{\epsfig{file=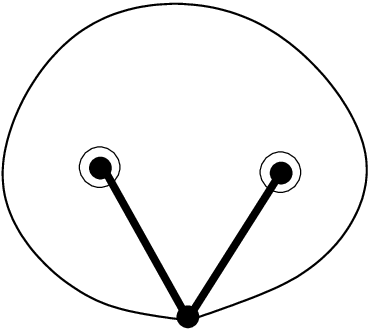,width=0.07\linewidth}}\\
\hhline{|~|~|~|-|-|}
&&&
\raisebox{4mm}{\small $\left(\begin{smallmatrix}
0&1&-1\\
-2&0&1\\
4&-2&0\\
\end{smallmatrix}\right)$}&
\raisebox{0mm}{\epsfig{file=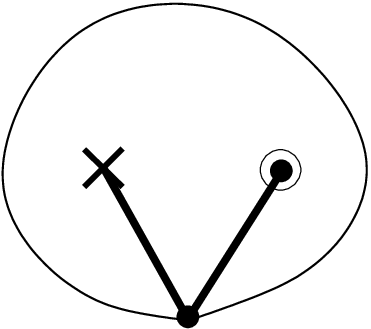,width=0.07\linewidth}}\\
\hhline{|~|~|~|-|-|}
&&&
\raisebox{4mm}{\small $\left(\begin{smallmatrix}
0&2&-4\\
-1&0&2\\
1&-1&0\\
\end{smallmatrix}\right)$}&
\raisebox{0mm}{\epsfig{file=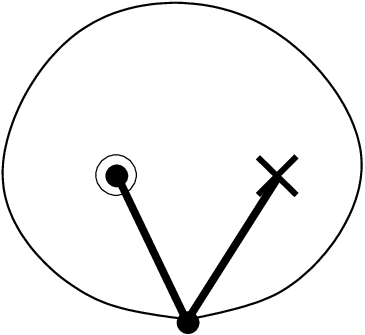,width=0.07\linewidth}}\\
\hhline{|~|~|~|-|-|}
&&&
\raisebox{4mm}{\small $\left(\begin{smallmatrix}
0&1&-2\\
-2&0&2\\
2&-1&0\\
\end{smallmatrix}\right)$}&
\raisebox{0mm}{\epsfig{file=diagrams_pic/6.eps,width=0.07\linewidth}}\\
\hline
\end{tabular}

\end{table}

\begin{table}
\caption{Exceptional s-blocks, their matrices, orbifolds, and associated orbifolds, see Sections~\ref{sec lambda 2} and~\ref{sec_gen}}
\label{all-matrices-e}
\begin{tabular}{|c|c|c|c|}
\hline
Diagram&Orbifold&Matrix
&Associated
orbifold
\\
\hline
\psfrag{2-}{\tiny $2$}\psfrag{6t}{}\raisebox{-10mm}[0mm][0mm]{\epsfig{file=diagrams_pic/block6t1.eps,width=0.08\linewidth}}&
\raisebox{-10mm}[0mm][0mm]{\epsfig{file=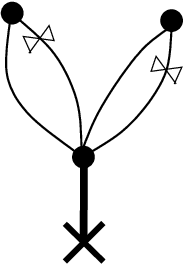,width=0.08\linewidth}}&
\raisebox{4mm}[5mm][0mm]{\small $\left(\begin{smallmatrix}
0&0&1&1&-1\\
0&0&1&1&-1\\
-1&-1&0&0&1\\
-1&-1&0&0&1\\
2&2&-2&-2&0\\
\end{smallmatrix}\right)$}&
\raisebox{-2mm}{\epsfig{file=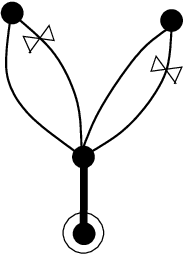,width=0.08\linewidth}}\\
\hhline{|~|~|-|-|}
&&
\raisebox{7mm}{\small $\left(\begin{smallmatrix}
0&0&1&1&-2\\
0&0&1&1&-2\\
-1&-1&0&0&2\\
-1&-1&0&0&2\\
1&1&-1&-1&0\\\end{smallmatrix}\right)$}&
\raisebox{0mm}{\epsfig{file=diagrams_pic/7.eps,width=0.08\linewidth}}
\\
\hline
\psfrag{2-}{\tiny $2$}\psfrag{4-}{\tiny $4$}\raisebox{-30mm}[0mm][0mm]{\epsfig{file=diagrams_pic/block6t2.eps,width=0.08\linewidth}}&
\raisebox{-30mm}[0mm][0mm]{\epsfig{file=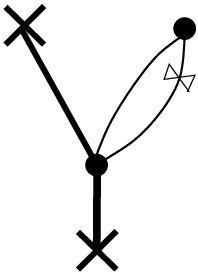,width=0.08\linewidth}}&
\raisebox{4mm}[5mm][0mm]{\small $\left(\begin{smallmatrix}
0&0&1&-1\\
0&0&1&-1\\
-2&-2&0&2\\
2&2&-2&0\\
\end{smallmatrix}\right)$}&
\raisebox{-2mm}{\epsfig{file=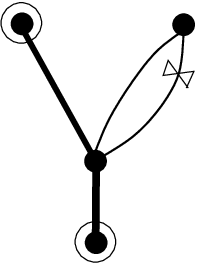,width=0.08\linewidth}}\\
\hhline{|~|~|-|-|}
&&
\raisebox{7mm}{\small $\left(\begin{smallmatrix}
0&0&1&-2\\
0&0&1&-2\\
-2&-2&0&4\\
1&1&-1&0\\
\end{smallmatrix}\right)$}&
\raisebox{0mm}{\epsfig{file=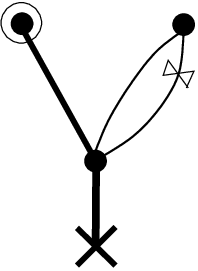,width=0.08\linewidth}}
\\
\hhline{|~|~|-|-|}
&&
\raisebox{7mm}{\small $\left(\begin{smallmatrix}
0&0&2&-1\\
0&0&2&-1\\
-1&-1&0&1\\
2&2&-4&0\\
\end{smallmatrix}\right)$}&
\raisebox{0mm}{\epsfig{file=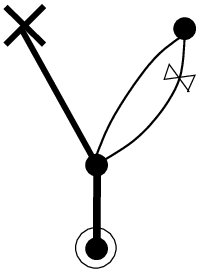,width=0.08\linewidth}}
\\
\hhline{|~|~|-|-|}
&&
\raisebox{7mm}{\small $\left(\begin{smallmatrix}
0&0&2&-2\\
0&0&2&-2\\
-1&-1&0&2\\
1&1&-2&0\\
\end{smallmatrix}\right)$}&
\raisebox{0mm}{\epsfig{file=diagrams_pic/17.eps,width=0.08\linewidth}}
\\
\hline
\psfrag{4-}{\tiny $4$}\psfrag{4}{\tiny $4$}\raisebox{-25mm}[0mm][0mm]{\epsfig{file=diagrams_pic/block6t3.eps,width=0.08\linewidth}}&
\raisebox{-30mm}[0mm][0mm]{\epsfig{file=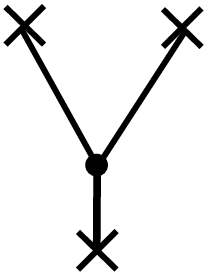,width=0.08\linewidth}}&
\raisebox{4mm}[5mm][0mm]{\small $\left(\begin{smallmatrix}
0&2&-2\\
-2&0&2\\
2&-2&0\\
\end{smallmatrix}\right)$}&
\raisebox{-2mm}{\epsfig{file=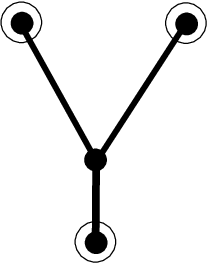,width=0.08\linewidth}}\\
\hhline{|~|~|-|-|}
&&
\raisebox{7mm}{\small $\left(\begin{smallmatrix}
0&1&-1\\
-4&0&2\\
4&-2&0\\
\end{smallmatrix}\right)$}&
\raisebox{0mm}{\epsfig{file=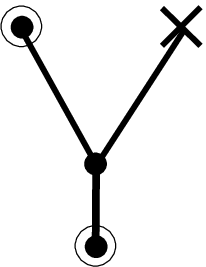,width=0.08\linewidth}}
\\
\hhline{|~|~|-|-|}
&&
\raisebox{7mm}{\small $\left(\begin{smallmatrix}
0&2&-1\\
-2&0&1\\
4&-4&0\\
\end{smallmatrix}\right)$}&
\raisebox{0mm}{\epsfig{file=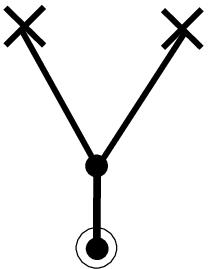,width=0.08\linewidth}}
\\
\hhline{|~|~|-|-|}
&&
\raisebox{7mm}{\small $\left(\begin{smallmatrix}
0&2&-2\\
-2&0&2\\
2&-2&0\\
\end{smallmatrix}\right)$}&
\raisebox{0mm}{\epsfig{file=diagrams_pic/18.eps,width=0.08\linewidth}}
\\
\hline
\end{tabular}

\end{table}

\section{Triangulations of orbifolds}
\label{orbifolds-s}
Let $S$ be a connected oriented 2-dimensional surface with (possibly empty) boundary $\p S$.

{ By an {\it orbifold} $\O$ we mean a triple $\O=(S,M,Q)$, where $S$ is a bordered surface with a finite set of marked points $M$, and $Q$
is a finite (non-empty) set of special points called {\it orbifold points}, $M\cap Q=\emptyset$. Some marked points may belong to $\p S$ (moreover, every boundary component must contain at least one marked point; the interior marked points are also called {\it punctures}), while
all orbifold points are interior points of $S$ (later on, as we will supply the orbifold with a metric, the orbifold points will have angle $\pi$). By {\it boundary} $\p \O$ we mean $\p S$.
}

An {\it arc} $\gamma$ in $\O$ is a curve in $\O$ considered up to relative isotopy (of $\O\setminus \{M\cup Q\}$) modulo endpoints such that
\begin{itemize}
\item one of the following holds:
\begin{itemize}
\item either both endpoints of $\gamma$ belong to $M$ (and then $\gamma$ is an {\it ordinary arc})
\item or one endpoint belongs to $M$ and another belongs to $Q$ (then $\gamma$ is called {\it pending arc});
\end{itemize}
\item $\gamma$ has no self-intersections, except that its endpoints may coincide;
\item except for the endpoints, $\gamma$ and $M\cup Q\cup \p \O$ are disjoint;
\item if $\gamma$ cuts out a monogon then this monogon contains either a point of $M$ or at least two points of $Q$;
\item $\gamma$ is not homotopic to a boundary segment.

\end{itemize}

Note that we do not allow both endpoints of $\gamma$ to be in $Q$.

Two arcs $\gamma$ and $\gamma'$ are {\it compatible} if the following two conditions hold:
\begin{itemize}
\item they do not intersect in the interior of $\O$;
\item if both $\gamma$ and $\gamma'$ are pending arcs, then the ends of $\gamma$ and $\gamma'$ that are orbifold points do not coincide { (i.e., two pending arcs may share a marked point, but neither an ordinary point nor a orbifold point)}.

\end{itemize}

A {\it triangulation} of $\O$ is a maximal collection  of distinct pairwise compatible arcs.
The arcs of a triangulation cut $\O$ into {\it triangles}. We allow self-folded triangles as well as triangles one or two
of whose edges are pending arcs. See Fig.~\ref{triangles} for the list of possible triangles.

The following lemma is evident.

\begin{lemma}
\label{max}
Any set of compatible arcs on an orbifold is contained in some triangulation.

\end{lemma}

\begin{figure}[!h]
\begin{center}
\epsfig{file=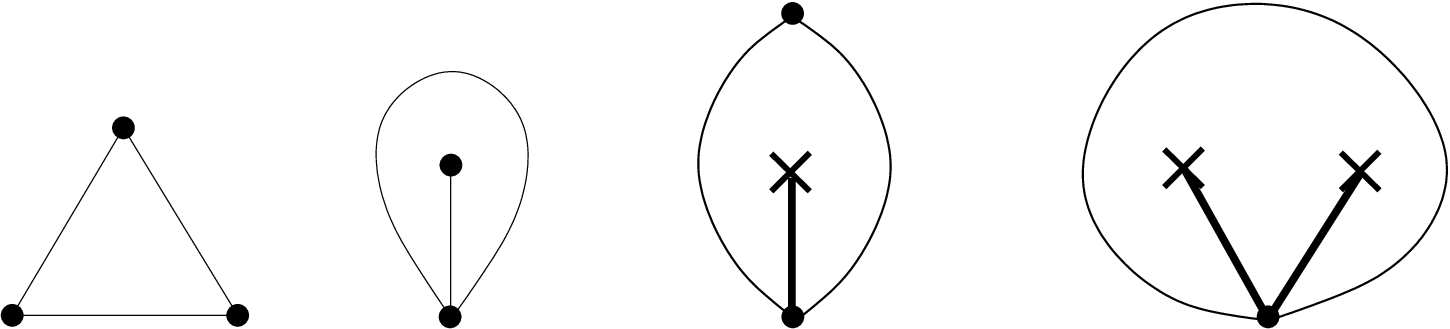,width=0.6\linewidth}
\caption{Types of triangles admissible for triangulations of orbifolds
(vertices marked by a cross denote orbifold points, bold edges denote pending arcs) }
\label{triangles}
\end{center}
\end{figure}

A {\it flip} of an arc $\gamma$ of a triangulation $T$  replaces $\gamma$ by a unique arc $\gamma'\ne \gamma$ such that
$\gamma' \cup (T\setminus \gamma)$ forms a new triangulation of $S$.
In Fig.~\ref{flip-pending} we show flips involving pending arcs.

\begin{figure}[!h]
\begin{center}
\epsfig{file=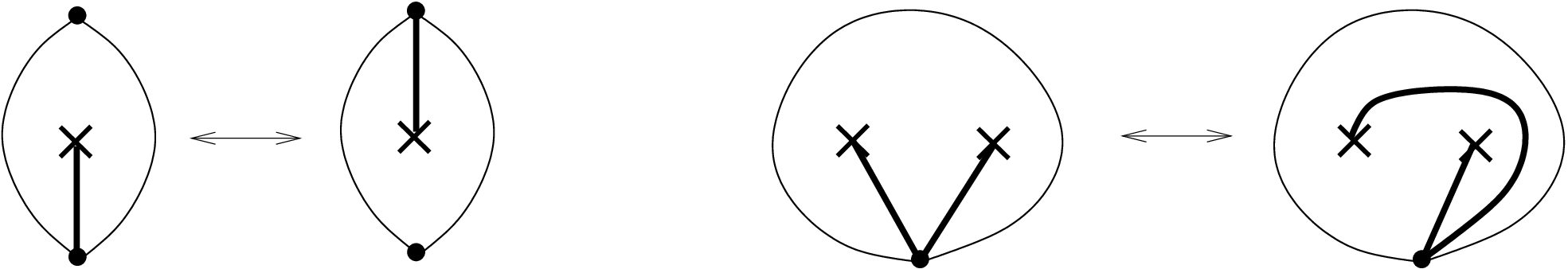,width=0.9\linewidth}
\caption{Flips of pending arcs}
\label{flip-pending}
\end{center}
\end{figure}

\subsection{Transitivity of flips on triangulations of orbifolds}

In this section we prove the following theorem.

\begin{theorem}
\label{transitivity}
For any orbifold $\O$ flips act transitively on triangulations of $\O$.

\end{theorem}

By a {\it system of pending arcs} of a triangulation $T$ we mean the union of all pending arcs of $T$.

\begin{lemma}
\label{same pending arcs}
Flips act transitively on triangulations with the same system of pending arcs.

\end{lemma}

\begin{proof}
Choose a system of pending arcs on $\O$.
To prove the lemma we cut the orbifold $\O$ along all pending arcs (i.e., we replace any pending arc by a hole with one marked point on the boundary) and denote by $S$ the obtained surface. The marked points of $S$ are the same as of $\O$, the orbifold points disappear. Every pending arc of $\O$ produces a boundary component of $S$ with exactly one marked point.

The triangulations of $\O$ containing the chosen system of pending arcs are in one-to-one correspondence with the triangulations of $S$
(and if two triangulations of $S$ are related by a flip in some arc, then the corresponding triangulations of $\O$ are related by
a flip in the corresponding arc). So, the lemma follows from transitivity of flips on triangulations of $S$ (see~\cite{H} and~\cite{FST}).

\end{proof}

A set $\{\g_1,\dots,\g_k\}$ of pending arcs on $\O$ is {\it compatible} if $\g_i$ and $\g_j$ are compatible for every $i\ne j$.

Every maximal compatible set of pending arcs is a system of pending arcs for some triangulation of $\O$: to see this, we cut $\O$ along all pending
arcs and triangulate the surface. In the sequel we will use the notion of system of pending arcs as a maximal compatible set of pending arcs not related to any triangulation.

An {\it elementary transformation} of a system $\Gamma=\{\g_1,\dots,\g_n\}$ of pending arcs is a substitution of a pending arc $\g_i$ by any other pending arc $\g_i'$ compatible with the set $\{\cup_j \g_j\}\setminus \g_i$ and not intersecting interior of $\g_i$.

\begin{lemma}
\label{elem}
Let $\Gamma=\{\g_1,\dots,\g_n\}$ be a system of pending arcs on $O$.
Let $\phi_i$ be an elementary transformation of a $\Gamma$ substituting $\g_i$ by $\g_i'$.
Then there exist triangulations $T$ and $T'$ containing the systems $\Gamma$ and $\Gamma'=\phi_i(\Gamma)$ respectively, such that $T'=f_i(T)$ where $f_i$ is a flip in the pending arc $\g_i$.

\end{lemma}

\begin{proof}
Let $c_i$ be the common orbifold point of $\g_i$ and $\g_i'$. Let $x_i$ and $x_i'$ be the marked (i.e., non-orbifold) ends of $\g_i$ and $\g_i'$
(possibly, $x_i=x_i'$). Consider a path $\g_i\g_i'$ from $x_i$ through $c_i$ to $x_i'$.
Denote by $p_l$ and $p_r$ the paths in $\O$ built as in Fig.~\ref{paths}: $p_l$ goes from $x_i$ to $x_i'$
following the path $\g_i\g_i'$ and shifted to the left, while $p_r$ is the similar path shifted to the right from $\g_i\g_i'$.
Since the disc bounded by $p_l$ and $p_r$ contains no singularities except $c_i$, the curves $p_l$,$p_r$ and $\g_i$ are sides of
an admissible triangle $\Delta$ on $\O$ (in fact, $p_l$ may coincide with $p_r$ if $\O=\Delta$).
Similarly,  $p_l$, $p_r$ and $\g_i'$ are sides of a triangle $\Delta'$.

Now, delete the triangle with sides $p_l$, $p_r$, $\g_i$ from the orbifold $\O$ and choose any triangulation $T_0$ on $\O\setminus \Delta$
compatible with the set of pending arcs $\cup_j \g_j\setminus \g_i$. Then $T\cup \Delta$ is a triangulation of $\O$
compatible with the  system of pending arcs $\Gamma$. Similarly, $T\cup \Delta'$  is compatible with $\Gamma'$.
It is left to note that the triangulation $T\cup \Delta'$ can be obtained from $T\cup \Delta$ by a flip in the pending arc $\g_i$.

\end{proof}

\begin{figure}[!h]
\begin{center}
\psfrag{x}{\scriptsize $x$}
\psfrag{x'}{\scriptsize $x'$}
\psfrag{pl}{\scriptsize $p_l$}
\psfrag{pr}{\scriptsize $p_r$}
\psfrag{z'}{\scriptsize $\g_i'$}
\psfrag{z}{\scriptsize $\g_i$}
\psfrag{c}{\scriptsize $c_i$}
\epsfig{file=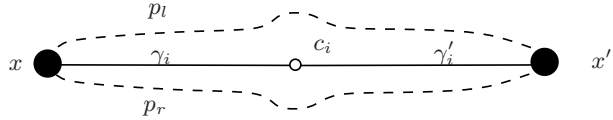,width=0.5\linewidth}
\caption{Proof of Lemma~\ref{elem}: paths $p_l$ and $p_r$}
\label{paths}
\end{center}
\end{figure}

Lemmas~\ref{same pending arcs} and~\ref{elem} show that to prove Theorem~\ref{transitivity}it is sufficient to prove the transitivity of action of elementary transformations on the set of systems of pending arcs of $\O$.

A system $\Gamma$ of pending arcs is {\it centered at a marked point $x$} if $x$ is an endpoint of every pending arc of $\Gamma$.

\begin{lemma}
\label{center}
For any system $\Gamma=\{\g_1,\dots,\g_n\}$ of pending arcs on $\O$ and any marked point $x\in \O$ one can find a sequence of at most $n$ elementary transformations which takes $\Gamma$ to a system $\Gamma'$ centered at $x$.

\end{lemma}

\begin{proof}
For each orbifold point $c_i$ there exists a unique pending arc $\g_i$ containing $c_i$, so $\O\setminus \{\cup_j \g_j\}$
is connected. This implies that we can perform an elementary transformation that replaces $\gamma_i$ by a pending arc connecting $c_i$ with a chosen fixed marked point $x$.

\end{proof}

\begin{lemma}
\label{order}
Let $\Gamma$ and $\t\Gamma$ be two systems of pending arcs, both centered at the same marked point $x$.
Then there exists a sequence of elementary transformations taking $\Gamma$ to $\t\Gamma$.

\end{lemma}

\begin{proof}
First, we will show that using elementary transformations we can create a system containing a given pending arc.

\medskip

\noindent
{\bf Claim~1.}
{\it
Let $\Gamma=\{\g_1,\dots,\g_n \}$ be a system of pending arcs centered at $x$ and let $p$ be a path from $x$ to an orbifold point $c_1$.
Then there exists a system $\Gamma'$ of pending arcs centered at $x$, $p\in \Gamma'$, and a sequence of elementary transformations taking $\Gamma$ to $\Gamma'$.
}

\smallskip

To prove the statement we perturb $p$ so that it intersects $\Gamma$ transversely.
If $p\cap (\Gamma\setminus x)=\emptyset$ then there is nothing to prove: an elementary transformation substituting $\g_1$ by $p$ turns $\Gamma$ into
required system $\Gamma'$. So, we assume that $p\cap (\Gamma\setminus x)\ne \emptyset$. Denote by $|p\cap \Gamma|$ the number of points of intersection. We will show that there exists an elementary transformation which takes $\Gamma$ to a system $\Gamma_1$ such that
$|\Gamma_1\cap p |<|\Gamma\cap p|$.

Let $t\in \g_i$ be the first intersection point of the path $p$ (from $x$ to $c_1$)
with $\Gamma\setminus x$. Consider a path $\g_i'$ composed of the segment $[x,t]$ of the path $p$ and a segment $[t,c_i]$ of the path $\g_i$
(and then shift $\g_i'$ to minimize the number of intersections with $p$ and $\Gamma$, see Fig.~\ref{claim1}).
Notice that $\g_i'\cap (\Gamma\setminus \g_i)=\emptyset$, so there exists an elementary transformation of $\Gamma$ substituting $\g_i$ by $\g_i'$.
On the other hand, the path $p$ intersects system $\Gamma'=\g_i'\cup (\Gamma\setminus \g_i)$  in all the intersection points of $p\cap \Gamma$
except $t$, so $|\Gamma_1\cap p |=|\Gamma\cap p|-1$.

\begin{figure}[!h]
\begin{center}
\psfrag{x}{\scriptsize $x$}
\psfrag{p}{\scriptsize $p$}
\psfrag{t}{\raisebox{-0.5mm}{\scriptsize $t$}}
\psfrag{z'}{\scriptsize $\g_i'$}
\psfrag{z}{\scriptsize $\g_i$}
\epsfig{file=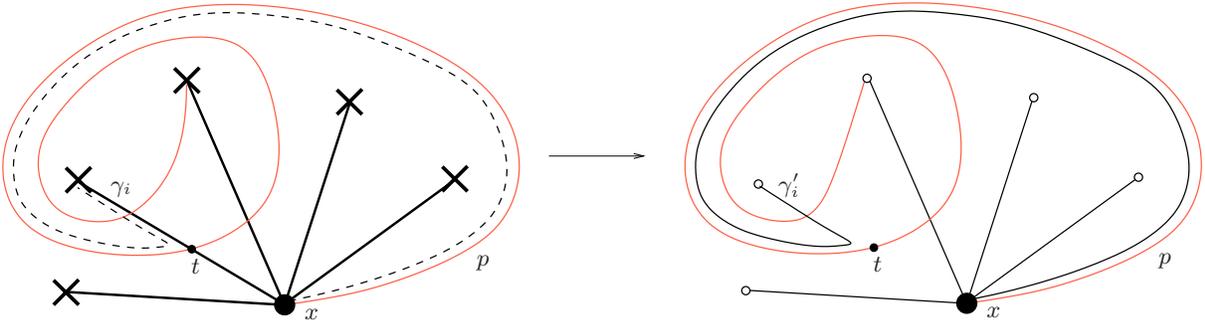,width=1.0\linewidth}
\caption{Elementary transformation decreasing the number of intersections $|\Gamma\cap p|$.}
\label{claim1}
\end{center}
\end{figure}

Thus, elementary transformations of the system of pending arcs allow us to decrease the number of intersection points of $\Gamma$ with $p$
by $1$, which implies that after several elementary transformations we come to the case  $p\cap (\Gamma\setminus x)=\emptyset$ which was
treated above. This proves Claim~1.

\medskip

Now we will use elementary transformations to include a given pending arc in a system of pending arcs preserving some subset of the system.

\noindent
{\bf Claim~2.}
{\it
Let $\Gamma=\{\g_1,\dots,\g_{n} \}$ be a system of pending arcs centered at $x$ and let $p$ be a path from $x$ to an orbifold point $c_i$.
Suppose that $p$ does not intersect the curves $\g_1,\dots,\g_{k-1}$. Then there exists a system $\Gamma_1$ of pending arcs containing   $\{\g_1,\dots,\g_{k-1},p\}$ and a sequence of elementary transformations taking $\Gamma$ to $\Gamma_1$.
}

Indeed, since $p$ does not intersect $\g_1,\dots,\g_{k-1}$, elementary transformations described in the proof of Claim~1 never affect
the curves $\g_1,\dots,\g_{k-1}$, so, these pending arcs also belong to the resulting collection $\Gamma_1$.

\smallskip

Now, to prove the lemma, it is sufficient to apply Claim~2 several times. Namely, given two systems $\Gamma=\{\g_1,\dots,\g_n\}$ and
$\t\Gamma=\{\g_1',\dots, \g_n'\}$, we choose $p=\g_1'$ and apply Claim~2. As we obtain a system $\Gamma_1$ containing $\g_1'$, choose $p=\g_2'$ and apply Claim~2 again to obtain a system $\Gamma_2$ containing both $\g_1'$ and $\g_2'$ (Claim~2 applies since $p=\g_2'$ does not intersect $\g_1'$).
Applying Claim~2 $n$ times (and choosing $p=\g_k'$ at $k$-th iteration) we obtain the required system $\t\Gamma$.
\end{proof}

\begin{proof}[Proof of Theorem~\ref{transitivity}]
We will show that every triangulation $T$ can be transformed by flips into one fixed triangulation $T_0$, the system of pending arcs of which is centered at randomly chosen marked point $x$.

By Lemma~\ref{center}, we can take system $\Gamma$ of pending arcs of $T$ to a system $\Gamma'$ centered at $x$. Applying Lemma~\ref{order}, we can obtain system $\Gamma_0$ of pending arcs of $T_0$. By Lemmas~\ref{same pending arcs}~and~\ref{elem}, all the elementary transformations above can be realized by sequences of flips. Now, applying Lemma~\ref{same pending arcs} another one time, we perform a sequence of flips to obtain triangulation $T_0$.

\end{proof}

Later we will also need the following two easy statements concerning triangulations of orbifolds.

\begin{lemma}
\label{bubbles-even}
Let $\O$ be an orbifold with even number of orbifold points.
Then there exists a triangulation $T$ of $\O$ such that every orbifold point of $\O$ is contained in a monogon with two pending arcs.

\end{lemma}

\begin{proof}
To find the required triangulation, first we connect all orbifold points with the same marked point $x$ (Lemma~\ref{center}).
Then we group the pending arcs in disjoint pairs of neighboring arcs, and
for each pair $(e,e')$ we draw a curve $p$ which starts at $x$, goes along $e$, then goes along $e'$
and returns back to $x$ (we assume that $p$ is close enough to $e$ and $e'$ so that $e,e'$ and $p$ compose a triangle).
Notice, that the curves obtained for different pairs of adjacent pending arcs are distinct (excluding the case of a sphere with
exactly one puncture and four orbifold points; in this case we just consider one of the two homotopy equivalent curves).
So, if there are $n=2k$ orbifold points in $\O$ then we build a compatible set of $2k$ pending arcs and $k$ curves enclosing the pairs
of pending arcs in disks. Any triangulation containing this set of arcs satisfies the conditions of the lemma
(such a triangulation does exist in view of Lemma~\ref{max}).

\end{proof}

\begin{lemma}
\label{bubbles-odd}
Let $\O$ be an orbifold with odd number $n=2k+1$ of orbifold points, $k\ge 1$.
Then there exists a triangulation $T$ of $\O$ such that
\begin{itemize}
\item $T$ contains $k$ triangles $\Delta_1,\dots,\Delta_k$
with two pending arcs and one triangle $\Delta_0$ with one pending arc (and several triangles without pending arcs);
\item $\Delta_0$ has a common edge with $\Delta_1$.
\end{itemize}
\end{lemma}

\begin{proof}
The proof is similar to the proof of the previous lemma. First, we find a compatible system of $n=2k+1$
pending arcs connecting all orbifold points with the same marked point $x$. Then we enclose $k$ pairs of adjacent pending arcs by curves $p_1,\dots,p_k$
and add one extra curve $p_0$ enclosing the extra pending arc together with one of the discs above (as in Fig.~\ref{bubbles}).
Similarly to the case of even number of orbifold points, there are several exclusions: namely,
if $\O$ is a sphere with only one puncture and at most $5$ orbifold points, then the curve $p_0$ coincides with one of the $p_i$.
Now, we are left to take any triangulation containing all the curves described above.

\end{proof}

\begin{figure}[!h]
\begin{center}
\psfrag{D3}{\scriptsize $\Delta_0$}
\psfrag{D1}{\scriptsize $\Delta_1$}
\psfrag{D2}{\scriptsize $\Delta_2$}
\psfrag{c0}{\scriptsize $p_0$}
\psfrag{c1}{\scriptsize $p_1$}
\psfrag{c2}{\scriptsize $p_2$}
\psfrag{x}{\scriptsize $x$}
\psfrag{c'}{\scriptsize $p_0$}
\epsfig{file=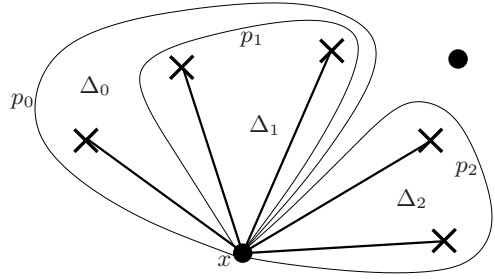,width=0.4\linewidth}
\caption{To the proof of Lemma~\ref{bubbles-odd}.}
\label{bubbles}
\end{center}
\end{figure}

\subsection{Mutations of diagrams and tagged triangulations of orbifolds}

To every triangulation of an orbifold we associate the following {\it diagram} $\D=\D(T)$:
\begin{itemize}
\item vertices of $\D$ correspond to arcs of $T$ (we denote by $v_i\in \D$ a vertex corresponding to an arc $e_i\in T$);
\item for every non-self-folded triangle $\Delta\in T$ and every pair of sides $(e_i,e_j)$ of $\Delta$ we draw an arrow in $\D$ from $v_i$ to $v_j$ if $e_j$ follows $e_i$ in $\Delta$ in clockwise order;
\item for every self-folded triangle $\Delta\in T$ with sides $(e_i,e_j,e_i)$ and for every arrow from $v_i$ to $v_k$ (from $v_k$ to $v_i$) we draw an arrow from $v_j$ to $v_k$ (respectively, from $v_k$ to $v_j$).
\item arrows between a pending arc and a non-pending arc are labeled by $2$;
\item arrows between two pending arcs are labeled by $4$;
\item if $v_i$ and $v_j$ are connected by two arrows in opposite directions, then these arrows cancel out;
\item if $v_i$ and $v_j$ are connected by two arrows in the same direction, then these two arrows are substituted by one arrow labeled by $4$.

\end{itemize}

\begin{remark}
If there are no orbifold points (i.e., the orbifold is just a bordered surface with marked points), the construction above coincides with the one from~\cite{FST} leading to a quiver associated to a triangulation of a surface.

\end{remark}

As it was shown in~\cite{FST}, block-decomposable quivers are precisely ones corresponding to triangulations of surfaces. We will now generalize this result by establishing similar correspondence between s-decomposable diagrams and triangulations of orbifolds.

\begin{lemma}
Any diagram obtained from triangulation of an orbifold is s-decomposable.

\end{lemma}

\begin{proof}
Skew-symmetric blocks together with s-blocks represent all possible triangles which may appear in the triangulation.
So, we take the blocks corresponding to the triangles and attach their outlets in accordance with the gluings in the triangulation.
This results in the s-decomposition of the diagram corresponding to given triangulation.

%

\end{proof}

\begin{lemma}
\label{construction of O}
Any s-decomposable diagram can be obtained from a triangulation of some orbifold.

\end{lemma}

\begin{proof}
For each of the blocks we take the corresponding triangulated surface (or orbifold, see Tables~\ref{blocks},~\ref{all-matrices} and~\ref{all-matrices-e})
and attach them along the boundary edges in accordance with the gluing of the blocks in the diagram.

\end{proof}

So, block decompositions of s-decomposable diagrams are in one-to-one correspondence with triangulations of orbifolds.

\begin{remark}
One can see from Table~\ref{all-matrices-e} that all the exceptional blocks arise from  triangulations of a sphere without boundary with four punctures and orbifold points in total, from which one, two, or three are punctures, and the remaining ones are orbifold points.
\end{remark}

Now we will show that, as in the case of surfaces, the construction of a diagram is consistent with action of flips on triangulations.
If $\D(T)$ is the diagram built by a triangulation $T$ and $f_i$ is a flip of an arc $e_i$ of $T$, we denote by $\mu_i$ the mutation of $\D(T)$ in vertex $v_i$.

\begin{lemma}
\label{flip-mut}
For each flip $f_i$ one has $\D(f_i(T))=\mu_i(\D(T))$.

\end{lemma}

\begin{proof}
The proof is a straightforward exhaustion of finitely many possibilities for adjacent to $e_i$ triangles.
Notice that  one need to verify only the cases when one of the adjacent triangles contains a pending arc
(in view of the similar fact known for the triangulated surfaces, see~\cite{FST}).

\end{proof}

As in the case of surfaces, one can note that not every edge of triangulation can be flipped. More precisely, there is no flip in an interior edge of a self-folded triangle. This difficulty can be resolved by the same trick as in the case of surfaces, namely, by introducing {\it tagged} triangulations (see~\cite{FST},~\cite{FT}). The construction is exactly the same as in the surface case, so we do not stop here for the details. We only mention that the ends of pending arcs being orbifold points are always tagged plain. It is easy to see that Lemma~\ref{flip-mut} holds for tagged triangulations as well.

We summarize the discussion above in the following lemma.

\begin{lemma}
\label{realization of diagram mut}
Let $\D$ be an s-decomposable diagram, and let $\O$ be an orbifold with tagged triangulation $T$ such that $\D=\D(T)$.
Then a diagram $\D'$ is mutation-equivalent to $\D$ if and only if $\D'$ can be obtained as $\D'=\D'(T')$ for some tagged triangulation $T'$ of $\O$.
Moreover, $\D'=\mu_{i_k}\circ\dots\circ\mu_{i_1}(D)$ if and only if  $T'=f_{i_k}\circ\dots\circ f_{i_1}(T)$.

\end{lemma}

\subsection{Weighted orbifolds and matrix mutations}
\label{w_orb}

In the previous section we established a correspondence between s-decomposable diagrams and triangulated orbifolds. Every s-decomposable diagram $\D$ can be considered as a diagram of some mutation-finite s-decomposable skew-symmetrizable matrix $B$. In contrast to the skew-symmetric case, such matrix is not uniquely defined: every s-decomposable diagram with at least one edge labeled by $2$ can be constructed by several matrices. In this section, we associate with every s-decomposable matrix a triangulated orbifold with additional structure.

Given an s-decomposable skew-symmetrizable matrix $B$, denote by $D$ a unique diagonal matrix with positive integer entries $(d_1,\dots,d_n)$ such that $BD$ is skew-symmetric, and the greatest common divisor of the entries of $D$ is one. Given an s-decomposable diagram $\D$ constructed by s-decomposable matrix $B$, we call $d_i$ a {\it weight} of vertex $v_i$. The matrix $D$ is the same for every matrix $B'$ mutation equivalent to $B$, hence the weights of vertices of a diagram do not change under mutations.

\begin{definition}
An s-decomposable diagram $\D$ with a collection of weights $(d_1,\dots,d_n)$ as above is called a {\it weighted diagram} and is denoted by $\D^w$.
Weighted s-decomposable diagrams carry exactly the same information as s-decomposable matrices.

\end{definition}
In other words, weighted s-decomposable diagrams are in one-to-one correspondence with s-decomposable matrices.

It was shown in~\cite[Lemma~6.3]{FST2} that the weights $d_i$ of outlets of all blocks of any weighted s-decomposable connected diagram $\D^w$ are equal. The weight of any outlet is called {\it weight of the regular part} of $\D^w$ and is denoted by $w$.

Now fix an s-decomposable matrix $B$ and corresponding weighted diagram $\D^w$. According to Lemma~\ref{construction of O}, vertices of the diagram correspond to arcs of a triangulation $T$ of some orbifold $\O$. In this way we assign weights to every arc of $T$. In particular, we assign weight to every pending arc.

\begin{definition}
Given a weighted diagram $\D^w$ and corresponding triangulation $T$ of orbifold $\O$, a {\it weighted orbifold} $\O^w$ is the orbifold $\O$ with weights assigned to all its orbifold points according to the following rule: the weight of an orbifold point $c$ is the weight of the pending arc of $T$ incident to $c$ divided by $w$.

\end{definition}
The definition of weighted orbifold is consistent: a flip in any arc of triangulation does not change weight of a pending arc incident to a given orbifold point.

All possible weights of orbifolds points of a weighted orbifold are easy to describe. It was proved in~\cite{FST2} that the only weights that can appear in weighted diagram are $1$, $2$, and $4$. In~\cite[ Lemma~6.3]{FST2} we also proved that either $w=1$ or $w=2$, and if $w=1$ then there is no vertex of weight $4$. In terms of weights of orbifold points, every point has weight either $2$ or $1/2$.

Conversely, given an orbifold $\O$ with $k$ orbifold points, we can construct $2^k$ weighted orbifolds by assigning weights $2$ and $1/2$. Every triangulation of each of these orbifolds can be constructed by some s-decomposable skew-symmetrizable matrix.

Summarizing the definitions above, for every s-decomposable skew-symmetrizable matrix $B$ we constructed a weighted orbifold $\O^w$ with a triangulation $T$ via a weighted diagram $\D^w$. Now we are able to give a definition of signed adjacency matrix of a (tagged) triangulation of weighted orbifold.

\begin{definition}
If $(\O^w,T)$ is a weighted orbifold and its tagged triangulation corresponds to an s-decomposable skew-symmetrizable matrix $B$, the matrix $B$ is called {\it signed adjacency matrix} of $T$.

\end{definition}

\begin{remark}
It is easy to see that the actual number of distinct mutation classes of signed adjacency matrices constructed by one orbifold with $k$ orbifold points is usually much less than $2^k$. Namely, every two weighted orbifolds with the same number of orbifold points of weight $2$ (and thus with the same number of orbifold points of weight $1/2$) give rise to mutation-equivalent matrices (after some permutation of rows and columns). More precisely, given two orbifold points, the transposition of
them can be realized by a transposition of corresponding rows (and columns) of the signed adjacency matrix (with a mutation applied first in case of different weights).

At the same time, two weighted orbifolds with distinct number of orbifold points of weight $1/2$ (and thus $2$ as well) give rise to signed adjacency matrices with essentially different skew-symmetrizing matrices $D$, and thus these signed adjacency matrices are not mutation-equivalent. In other words, an orbifold with $k$ orbifold points provides $k+1$ distinct mutation classes of signed adjacency matrices indexed by the number of orbifold points of weight $2$ (or $1/2$).  

\end{remark}




\begin{theorem}
\label{thm_weight}
Let $B$ be a skew-symmetrizable matrix with weighted s-decomposable diagram $\D^w$. Let $\O^w$ be a weighted orbifold with a triangulation $T$ built by an s-decomposition of $\D^w$. Then a skew-symmetrizable matrix $B'$ is mutation-equivalent to $B$ if and only if $B'$ is a signed adjacency matrix of a tagged triangulation $T'$ of $\O^w$. Moreover, $B'=\mu_{i_k}\circ\dots\circ\mu_{i_1}(B)$ if and only if  $T'=f_{i_k}\circ\dots\circ f_{i_1}(T)$.

\end{theorem}

\begin{proof}
A straightforward verification shows that the matrix analogue of Lemma~\ref{flip-mut} holds, i.e. for each flip $f_i$ one has $B(f_i(T))=\mu_i(B(T))$, where $B(T)$ is a signed adjacency matrix of the tagged triangulation $T$ of weighted orbifold $\O^w$. Now, the theorem follows from Lemma~\ref{realization of diagram mut} combined with the fact that knowing a diagram $\D$ together with the weights of the orbifold points of $\O^w$ one can recover the matrix $B$ (via weighted diagram $\D^w$).

\end{proof}

\section{Geometric realization of cluster algebras}
\label{sec-geom}

\subsection{Lambda lengths as cluster variables}

In~\cite{FT} Fomin and Thurston show that the notion of lambda length introduced by Penner~\cite{P}
works well for obtaining a geometric realization of some cluster algebras. More precisely,
for every skew-symmetric block-decomposable matrix $B$ there exists a bordered hyperbolic surface
$S(B)$ {  with marked points} such that lambda lengths of arcs of (tagged) triangulations on $S(B)$ serve as cluster variables of some cluster algebra with exchange matrix $B$ in some seed.

In this section, we adjust the basic construction to the case of cluster algebras with s-decomposable skew-symmetrizable exchange matrices.

Let $B$ be a skew-symmetrizable matrix with weighted s-decomposable diagram $\D^w$. Let $\O^w$ be a triangulated weighted orbifold built by
an s-decomposition of $\D^w$. First, we will consider the case when all orbifold points on $\O^w$ are of weight $1/2$. In this case the role of cluster variables will be played by lambda lengths of the arcs of tagged triangulations of $\O^w$.
Next, we will treat the case  of orbifolds with all orbifold points of weight $2$. In this case, we will introduce a surface $\S$
``associated'' to the orbifold $\O^w$ (constructed from $\O^w$ by a simple procedure).
The lambda lengths of arcs of tagged triangulations of $\S$ will serve as cluster variables.
Finally, we consider the general case, when both orbifold points of weight $1/2$ and $2$ may appear.
Then the cluster algebra will be modeled by lambda lengths of arcs of tagged triangulations of an ``associated'' orbifold
(an orbifold constructed from $\O^w$ by the same procedure as in the previous case).
Notice, that in case of absence of the orbifold points (i.e. in case of skew-symmetric matrix $B$) we obtain exactly
the initial construction described in~\cite{FT}.

\subsection{Orbifolds with orbifold points of weight 1/2}
\label{sec lambda 1/2}

Let $B$ be a matrix with s-decomposable weighted diagram $\D^w$, let $\O^w$ be a corresponding weighted orbifold.
Suppose that all orbifold points in $\O^w$ are of weight $1/2$ (in terms of matrix/weighted diagram this means that all the outlets have weight $2$, and there are no vertices of weight $4$).

We endow $\O^w$ with a hyperbolic structure with cusps in all marked points and with angles $\pi$ in orbifold points. To choose such a structure, one may take any ideal triangulation $T_{\O^w}$ of $\O^w$ and assume that each triangle of $T_{\O^w}$ is an ideal hyperbolic triangle (there is a $1$-parameter freedom in attaching ideal triangles along a given edge). Suppose also that for each marked point $c$ on $\O^w$ we have chosen
a horocycle centered at $c$.

Such structures on $\O^w$ form a {\it decorated Teichm\"uller space} (cf.~\cite{P}): a point of a decorated Teichm\"uller space
$\widetilde \T(\O^w)$ is a hyperbolic structure as above with a collection of horocycles, one around each marked point.
It is shown by Chekhov and Mazzocco~\cite{Ch},~\cite{ChM} that $\T(\O^w)$ is parametrized by the set of functions ({\it lambda lengths}) assigned to arcs of given triangulation of $\O^w$ (including boundary segments), defined in the following way.

\begin{definition}[{{\it Lambda length}}]
\label{lambda}
For an arc $\gamma$ with both ends in marked points we define a {\it lambda length} as usual (see~\cite{P}):
$$\lambda(\gamma)=exp(l(\gamma)/2),$$
where $l(\gamma)$ is the signed distance along $\gamma$ between the horocycles (positive, if the horoballs bounded by the horocycles do not intersect, and negative otherwise).

If $\gamma$ is a pending arc, we define
$$\lambda(\gamma)=exp(l(\gamma)/2)=exp(l'(\gamma)),$$
where $l'(\gamma)=l(\gamma)/2$ is the signed distance from the orbifold point to the horocycle (negative, if the orbifold point is contained inside
the horoball, and positive otherwise).

\end{definition}

In the definition above $l(\gamma)$ can be understood as the length of the ``round trip'' from the horocycle to the orbifold point and back.

\begin{figure}[!h]
\begin{center}
\psfrag{alpha}{\scriptsize $\alpha$}
\psfrag{beta}{\scriptsize $\beta$}
\psfrag{gamma}{\scriptsize $\gamma$}
\psfrag{delta}{\scriptsize $\delta$}
\psfrag{sigma}{\scriptsize $\sigma$}
\psfrag{theta}{\scriptsize $\theta$}
\psfrag{xi}{\scriptsize $\xi$}
\psfrag{mu}{\scriptsize $\mu$}
\psfrag{nu}{\scriptsize $\nu$}
\psfrag{zeta}{\scriptsize $\zeta$}
\psfrag{eta}{\scriptsize $\eta$}
\psfrag{phi}{\scriptsize $\phi$}
\psfrag{psi}{\scriptsize $\psi$}
\psfrag{hi}{\scriptsize $\chi$}
\psfrag{a}{\small (a)}
\psfrag{b}{\small (b)}
\psfrag{c}{\small (c)}
\psfrag{d}{\small (d)}
\epsfig{file=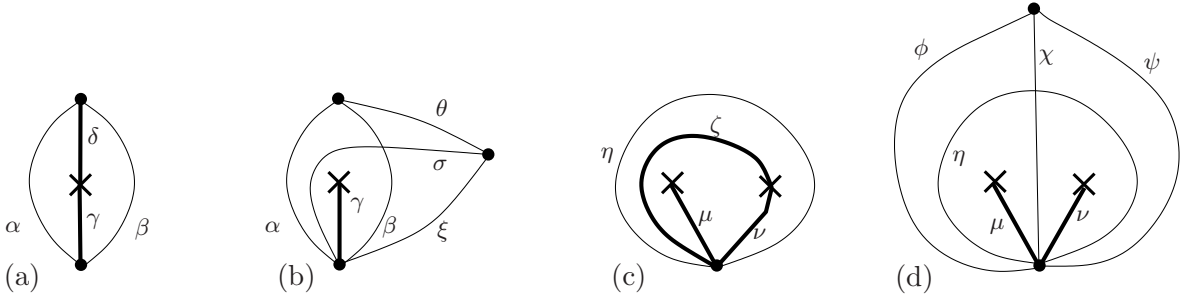,width=0.98\linewidth}
\caption{Notation for Ptolemy relations (Lemma~\ref{Ptolemy-prime})}
\label{ptolemy}
\end{center}
\end{figure}

\begin{lemma}
\label{Ptolemy-prime}
In the notation of Fig.~\ref{ptolemy} the following Ptolemy relations hold:
\begin{itemize}
\item[(a)] $\lambda(\gamma)\lambda(\delta)=\lambda(\alpha)^2+\lambda(\beta)^2$;
\item[(b)] $\lambda(\beta)\lambda(\sigma)=\lambda(\gamma)\lambda(\theta)+\lambda(\alpha)\lambda(\xi)$;
\item[(c)] $\lambda(\nu)\lambda(\zeta)=\lambda(\mu)^2+\lambda(\eta)^2$;
\item[(d)] $\lambda(\eta)\lambda(\chi)=\lambda(\mu)\lambda(\psi)+\lambda(\nu)\lambda(\phi)$.

\end{itemize}

\end{lemma}

\begin{proof}
First, we prove the relation (a). We cut the digon shown in Fig~\ref{ptolemy}.a along the pending arc $\gamma$, then glue together
two copies of the obtained triangle as in  Fig~\ref{ptolemy-proof}.a (together with the triangle we copy the chosen horocycles).
Since all orbifold points are points with angle $\pi$, we obtain a piece of hyperbolic surface. The relation (a) now follows from
the Ptolemy relation for triangulations of surfaces (together with  definitions of lambda lengths on surface and on orbifolds).

The relations (b)--(d) are proved similarly.  All of these relations describe some flips of a triangulation $T$ of $\O^w$,
so, we consider a quadrilateral $q$ (i.e.a union of two triangles) of $T$ containing the arcs included in the relation
(or a unique triangle $t$ in case of a flip in a pending arc).
We cut $\O^w$ along all pending arcs of $\O^w$ and along the boundary of the quadrilateral $q$ (respectively, along the boundary of
the triangle $t$), so that we obtain a quadrilateral or a triangle on hyperbolic plane. In case of a flip in a pending arc, we attach
two copies of the triangle $t$ along the image of the pending arc. Hence, in any case we come to a relation inside a quadrilateral on
hyperbolic plane, which follows immediately from the relations shown in~\cite{FT}.
See  Fig~\ref{ptolemy-proof}.(b)--(d) for the corresponding planar quadrilaterals.

\end{proof}

\begin{figure}[!h]
\begin{center}
\psfrag{alpha}{\scriptsize $\alpha$}
\psfrag{beta}{\scriptsize $\beta$}
\psfrag{gamma}{\scriptsize $\gamma$}
\psfrag{delta}{\scriptsize $\delta$}
\psfrag{sigma}{\scriptsize $\sigma$}
\psfrag{theta}{\scriptsize $\theta$}
\psfrag{xi}{\scriptsize $\xi$}
\psfrag{mu}{\scriptsize $\mu$}
\psfrag{nu}{\scriptsize $\nu$}
\psfrag{zeta}{\scriptsize $\zeta$}
\psfrag{eta}{\scriptsize $\eta$}
\psfrag{phi}{\scriptsize $\phi$}
\psfrag{psi}{\scriptsize $\psi$}
\psfrag{hi}{\scriptsize $\chi$}
\psfrag{a}{\small (a)}
\psfrag{b}{\small (b)}
\psfrag{c}{\small (c)}
\psfrag{d}{\small (d)}
\epsfig{file=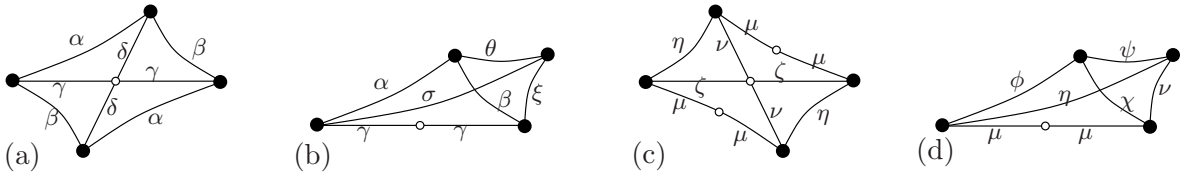,width=0.98\linewidth}
\caption{To the proof of Lemma~\ref{Ptolemy-prime}}
\label{ptolemy-proof}
\end{center}
\end{figure}

For a horocycle $h$ centered at interior marked point (puncture) of $\O^w$ denote by $L(h)$ the hyperbolic length of $h$.
Following~\cite{FT} we define a {\it conjugate horocycle} $\bar h$ around the same puncture by
the condition $L(h)L(\bar h)=1$.
As in~\cite{FT}, define a lambda length of a tagged arc using the distance to the conjugate horocycle:
for ends of the arc tagged plain one takes the distance to the initial horocycle, for ends tagged notched one takes
the distance to the conjugate horocycle.
Reasoning as in Lemma~\ref{Ptolemy-prime} one can see that the similar Ptolemy relations hold for lambda lengths of tagged
arcs.

\begin{remark}
\label{rem_ptolemy}
In Lemma~\ref{Ptolemy-prime} we discuss the basic Ptolemy relations only. More relations can be obtained in the same way by gluing some boundary edges
in Fig.~\ref{ptolemy}: some edges at the boundary can be attached to other ones, some edges can be self-identified
producing pending arcs.

\end{remark}

Now, let $T_{\O^w}$ be a triangulation of $\O^w$ with signed adjacency matrix $B=B(T_{\O^w})$. Choose an initial seed as follows
\begin{itemize}
\item $B=B(T_{\O^w})$;
\item ${\bf x}={\bf x}(T_{\O^w})=\{\lambda(\gamma): \gamma\in T_{\O^w}, \gamma\not\subset \partial \O^w\}$;
\item ${\bf p}={\bf p}(T_{\O^w})=\{\lambda(\gamma): \gamma\in T_{\O^w}, \gamma\subset \partial \O^w\}.$

\end{itemize}

Consider the cluster algebra $\A(B(T_{\O^w}),{\bf x}(T_{\O^w}),{\bf p}(T_{\O^w}))$
constructed by the initial seed $(B(T_{\O^w}),{\bf x}(T_{\O^w}),{\bf p}(T_{\O^w}))$.
Lemma~\ref{Ptolemy-prime} combined with Theorem~\ref{thm_weight} lead to the following theorem.

\begin{theorem}
\label{th_prime}
Let $B$ be an s-decomposable skew-symmetrizable matrix. Let $(\O^w,T_{\O^w})$ be a weighted orbifold and its triangulation constructed by an s-decomposition of $B$. Suppose also that weights of all orbifold points of $\O^w$ are equal to $1/2$.
Then the  cluster algebra  $\A(B(T_{\O^w}),{\bf x}(T_{\O^w}),{\bf p}(T_{\O^w}))$ satisfies the following conditions:
\begin{itemize}
\item the cluster variables of $\A$ are lambda lengths of tagged arcs of triangulations of $\O^w$;
\item the clusters consist of all lambda lengths of tagged arcs contained in the same triangulation of $\O^w$;
\item the coefficients are lambda lengths of the boundary components of $\O^w$
(the coefficient semifield $\P$ is the tropical semifield generated by the lambda lengths of boundary components);
\item the exchange graph of $\A$ coincides with the exchange graph of tagged triangulations of $\O^w$.

\end{itemize}
\end{theorem}

\subsection{Orbifolds with orbifold points of weight $2$}
\label{sec lambda 2}

As in the previous section, let $B$ be a matrix with s-decomposable weighted diagram $\D^w$, let $\O^w$ be a corresponding weighted orbifold.
Suppose that all orbifold points in $\O^w$ are of weight $2$ (this corresponds to matrices/weighted diagrams with all the outlets of weight $1$).
In this section, we introduce a new object $\S(\O^w)$ (a hyperbolic surface built from the orbifold $\O^w$) and show that the
lambda lengths of tagged arc on $\S(\O^w)$ provide a realization of a cluster algebra with exchange matrix $B$.

\begin{definition}[{{\it Associated triangulated surface}}]
\label{ass surface}
Let $T$ be a tagged triangulation of the orbifold $\O^w$. An {\it associated triangulated surface} $\S(\O^w)$ is a surface
with tagged triangulation $\hat T $ built as follows:
for each triangle $t\in T$ containing an orbifold point we cut out $t$ and attach a triangulated
disk with marked points as shown in Fig.~\ref{ass}, so that every pending arc in $t$ corresponds to a pair of conjugate arcs in
$\hat T$.

The marked points arising in the procedure will be called {\it special}.
The pairs of conjugate arcs arising in this procedure will be called {\it associated (pairs of) arcs}. Arcs in every associated pair should be flipped
simultaneously.

\end{definition}

\begin{figure}[!h]
\begin{center}
\epsfig{file=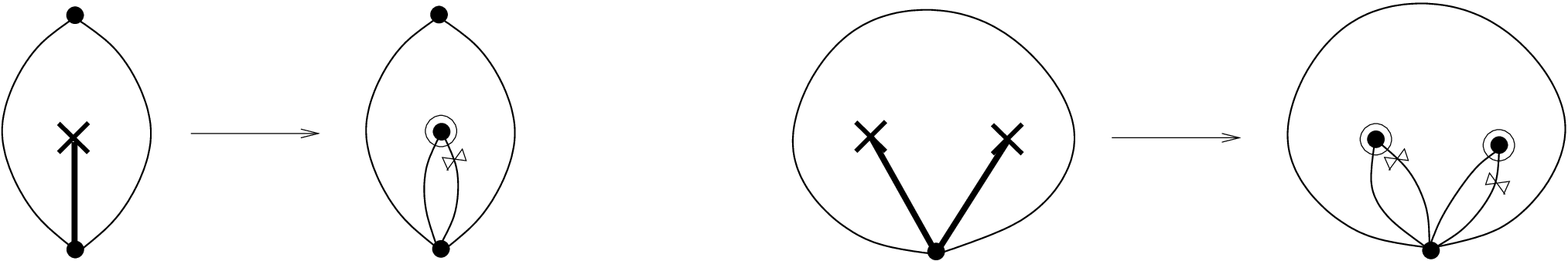,width=0.98\linewidth}
\caption{Construction of the associated surface. Special marked points are encircled.}
\label{ass}
\end{center}
\end{figure}

\begin{remark}
\label{ex gr ass surf}
The requirement that the associated arcs should be flipped simultaneously guarantees that the associated arcs always remain conjugate.
Therefore, the procedure of building the associated surface from the triangulation $T$ of $\O^w$ commutes with flips of $T$.
This implies that the exchange graph for tagged triangulations of $\O^w$ coincides with the exchange graph for tagged triangulations
of $\S(\O^w)$.

\end{remark}

{ 
Recall that horocycles $h$ and $\bar h$ are conjugate if $L(h)\cdot L(\bar h)=1$, where $L(h)$ is the hyperbolic length of $h$. Given a hyperbolic structure on $\S(\O^w)$,
a horocycle $h$ centered at interior marked point on $\S(\O^w)$ is {\it self-conjugate} if it coincides with its conjugate $\bar h$, implying $L(h)^2=1$.
}

\begin{definition}[{{\it Decorated Teichm\"uller space for associated surface}}]
A point in a decorated   Teichm\"uller space $\widetilde \T(\S)$ of the associated surface $\S=\S(\O^w)$
 is a hyperbolic structure on $\S$  with a collection of horocycles, one  around each marked point,
satisfying the condition that the horocycles centered at special marked points are self-conjugate.

\end{definition}

The lambda lengths of tagged arcs on associated surface are introduced in the usual way.
It follows directly from the definition that the associated arcs have the same lambda lengths.
Therefore, we may substitute a pair of associated tagged arcs $(\gamma',\gamma'')$ by a single tagged arc $\gamma$.
We call this single arc a {\it pending arc}, similar to the orbifold case and
define $\lambda(\gamma)=\lambda(\gamma')(=\lambda(\gamma''))$.

Given an associated surface $\S$ with $m$ special marked points, we can consider a similar surface $S$ whose marked points
are not special.
In the decorated Teichm\"uller space $\widetilde \T(S)$ of $S$ consider a codimension $m$ subspace $\Pi$ defined by any of the
following two equivalent conditions:
\begin{itemize}
\item[(1)] for marked point $x\in S$ corresponding to a special marked point of $\S$, the horocycle centered at $x$ is self-conjugate;
\item[(2)] for each pair of conjugate arcs $(\gamma',\gamma'')$ on $S$ corresponding to a pair of associated arcs of $\S$
holds $\lambda(\gamma)=\lambda(\gamma')$.

\end{itemize}
Clearly, the subspace $\Pi$ coincides with the decorated Teichm\"uller space  of  the associated surface $\S$.
The condition (2) also implies that the lambda lengths of tagged arcs of triangulations of $\S$ (together with lambda lengths of boundary segments) parametrize $\widetilde \T(\S)$.

\begin{figure}[!h]
\begin{center}
\psfrag{alpha}{\scriptsize $\alpha$}
\psfrag{beta}{\scriptsize $\beta$}
\psfrag{gamma}{\scriptsize $\gamma$}
\psfrag{delta}{\scriptsize $\delta$}
\psfrag{sigma}{\scriptsize $\sigma$}
\psfrag{theta}{\scriptsize $\theta$}
\psfrag{xi}{\scriptsize $\xi$}
\psfrag{mu}{\scriptsize $\mu$}
\psfrag{nu}{\scriptsize $\nu$}
\psfrag{zeta}{\scriptsize $\zeta$}
\psfrag{eta}{\scriptsize $\eta$}
\psfrag{phi}{\scriptsize $\phi$}
\psfrag{psi}{\scriptsize $\psi$}
\psfrag{hi}{\scriptsize $\chi$}
\psfrag{a}{\small (a)}
\psfrag{b}{\small (b)}
\psfrag{c}{\small (c)}
\psfrag{d}{\small (d)}
\epsfig{file=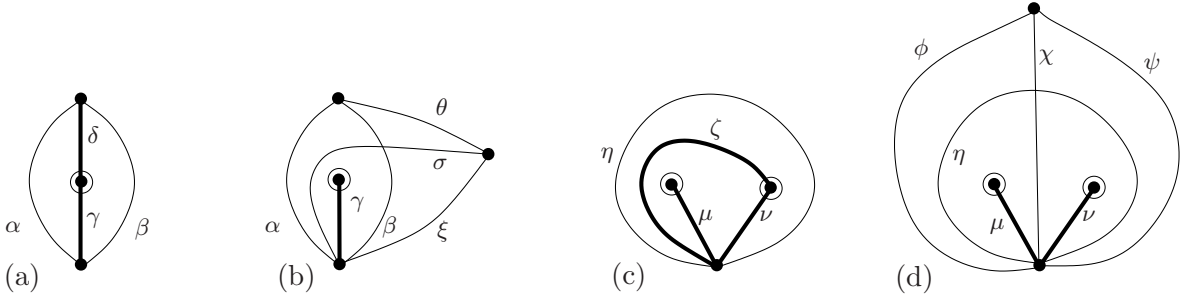,width=0.98\linewidth}
\caption{Notation for Ptolemy relations (Lemma~\ref{Ptolemy-local}). Pending arcs are drawn in bold.}
\label{ptolemy-local}
\end{center}
\end{figure}

\begin{lemma}
\label{Ptolemy-local}
In the notation of Fig.~\ref{ptolemy-local} the following Ptolemy relations hold:
\begin{itemize}
\item[(a)] $\lambda(\gamma)\lambda(\delta)=\lambda(\alpha)+\lambda(\beta)$;
\item[(b)] $\lambda(\beta)\lambda(\sigma)=\lambda(\gamma)^2\lambda(\theta)+\lambda(\alpha)\lambda(\xi)$;
\item[(c)] $\lambda(\nu)\lambda(\zeta)=\lambda(\mu)^2+\lambda(\eta)$;
\item[(d)] $\lambda(\eta)\lambda(\chi)=\lambda(\mu)^2\lambda(\psi)+\lambda(\nu)^2\lambda(\phi)$.

\end{itemize}

\end{lemma}

\begin{proof}
The relations immediately follow from the Ptolemy relations proved in~\cite{FT}.

\end{proof}

It is easy to see that the same Ptolemy relations hold also for tagged arcs.

Reasoning as in the previous section and taking into account Remark~\ref{ex gr ass surf}, we obtain a similar result for matrices providing weighted orbifolds with weights of orbifold points $2$.

\begin{theorem}
\label{th_local}
Let $B$ be an s-decomposable skew-symmetrizable matrix. Let $(\O^w,T_{\O^w})$ be a weighted orbifold and its triangulation constructed by an s-decomposition of $B$. Suppose also that weights of all orbifold points of $\O^w$ are equal to $2$.
Let $(\S=\S(\O^w),\hat T_{\S})$ be the corresponding associated surface and its triangulation.
Then the  cluster algebra  $\A= \A(B(\hat T_{\S}),{\bf x}(\hat T_{\S}),{\bf p}(\hat T_{\S}))$ satisfies the following conditions:
\begin{itemize}
\item the cluster variables of $\A$ are lambda lengths of tagged arcs contained in triangulations of $\S$;
\item the clusters consist of all lambda lengths of tagged arcs contained in the same triangulation of $\S$;
\item the coefficients are lambda length of the boundary segments of $\S$
(the coefficient semifield $\P$ is the tropical semifield generated by the lambda lengths of boundary segments);
\item the exchange graph of $\A$ coincides with the exchange graph of the tagged triangulations of $\S$.

\end{itemize}
\end{theorem}

\subsection{General case}
\label{sec_gen}
Let $B$ be an s-decomposable skew-symmetrizable matrix with wei\-ghted diagram $\D=\D(B)$, let $\O^w$ be a corresponding weighted orbifold. In this section, we present a general construction providing a geometric realization of a cluster algebra with exchange matrix $B$. In the partial cases when all orbifold points of $\O^w$ are of the same weight this construction specifies to ones described in Sections~\ref{sec lambda 2} and~\ref{sec lambda 1/2}.

\begin{definition}[{{\it Associated triangulated orbifold}}]
\label{def-ass-gen}
Given a triangulation $T$ of $\O^w$, an {\it associated orbifold} $\hat \O$ is constructed as follows:
for each triangle $t\in T$ containing orbifold points of weight 2 we substitute $t$ by a piece of surface with {\it special marked point}
as in Fig.~\ref{ass1}. Each special marked point is an end of two tagged arcs (with distinct tags at this end),
these tagged arcs are called {\it associated}. The associated arcs in the associated orbifold should be flipped simultaneously
(this ``composite'' flip is considered as one transformation and is shown by one edge in the exchange graph of tagged triangulations of
$\hat \O$). We denote obtained triangulation of $\hat \O$ by $\t T$.

\end{definition}

\begin{definition}
\label{def_tr_assoc}
By {\it triangulation} of associated orbifold $\hat\O$ we mean only tagged triangulations of $\hat\O$ with conjugate pairs in all special marked points.

\end{definition}

One can note that every tagged triangulation of $\hat \O$ (in the sense of the definition above) can be considered as $\t T$ for some triangulation $T$ of $\O^w$.

\begin{figure}[!h]
\begin{center}
\psfrag{1}{\scriptsize 1/2}
\psfrag{2}{\scriptsize 2}
\epsfig{file=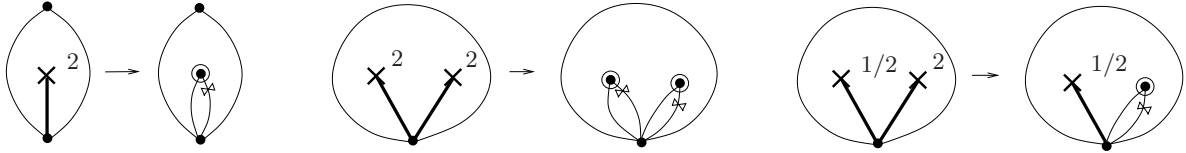,width=0.98\linewidth}
\caption{Construction of the associated orbifold. The numbers show the weights of the orbifold points.
 The special marked points are encircled.}
\label{ass1}
\end{center}
\end{figure}


Reasoning as in the case of associated surfaces in Remark~\ref{ex gr ass surf} we obtain the following lemma.

\begin{lemma}
\label{graph on O graph o}
The exchange graph of tagged triangulations of $\O^w$ coincides with exchange graph of tagged triangulations of
the associate orbifold $\hat\O$.

\end{lemma}

We define a hyperbolic structure on $\hat\O$ and pairs of conjugate horocycles in a usual way.

\begin{definition}[{{\it Decorated Teichm\"uller space for associated orbifold}}]
\label{dec teichm}
A point in a decorated   Teichm\"uller space $\widetilde \T(\hat \O)$ of the associated orbifold  $\hat \O$
 is a hyperbolic structure on $\hat \O$ (such that each marked point is turned into a cusp and an angle around each orbifold point equals
$\pi$)  together with a collection of horocycles, one  around each marked point,
satisfying the condition that the horocycles centered at special marked points are self-conjugate.

\end{definition}

As before, we use the chosen horocycles to define lambda lengths. For ordinary arcs we use the formula
$\lambda(\gamma)=exp(l(\gamma)/2),$ where $l(\gamma)$ is the signed distance along $\gamma$ between the horocycles.
For pending arcs we use the formula $\lambda(\gamma)=exp(l(\gamma)/2)=exp(l'(\gamma)),$
where $l(\gamma)=2l'(\gamma)$, and $l'(\gamma)$ is the signed distance from the horocycle to the orbifold point.

By definition, associated arcs have equal lambda lengths, so we substitute the pair of associated arcs by a single {\it double} arc (with a non-tagged end in the special marked point).

We denote by $\h T$ a triangulation $\t T$ of $\hat\O$ with  all conjugate pairs at special marked points substituted by double arcs. Flips of $\h T$ are obviously well-defined.

\begin{definition}
\label{signed_matr_assoc}
A {\it signed adjacency  matrix} for a tagged triangulation $\h T$ of the associated orbifold $\hat\O$ is the signed adjacency matrix for the initial
 triangulation $T$ of the orbifold $\O^w$.

\end{definition}
According to Lemma~\ref{graph on O graph o}, the signed adjacency matrix is well-defined.

Clearly, the lambda lengths on the associated orbifold satisfy the Ptolemy relations shown in~\cite{FT}, Lemma~\ref{Ptolemy-prime} and
Lemma~\ref{Ptolemy-local}. In addition, the same reasoning as in the proofs of the lemmas cited above (together with the results of the lemmas) shows the following Ptolemy  relations.

\begin{lemma}
\label{Ptolemy-gen}
In the notation of Fig.~\ref{ptolemy-gen} the following Ptolemy relations hold:
\begin{itemize}
\item[(a)] $\lambda(\nu)\lambda(\zeta)=\lambda(\eta)+\lambda(\mu)$;
\item[(b)] $\lambda(\mu)\lambda(\rho)=\lambda(\eta)^2+\lambda(\nu)^4$;
\item[(c)] $\lambda(\eta)\lambda(\chi)=\lambda(\mu)\lambda(\psi)+\lambda(\nu)^2\lambda(\phi)$.

\end{itemize}

\end{lemma}

\begin{figure}[!h]
\begin{center}
\psfrag{alpha}{\scriptsize $\alpha$}
\psfrag{beta}{\scriptsize $\beta$}
\psfrag{gamma}{\scriptsize $\gamma$}
\psfrag{delta}{\scriptsize $\delta$}
\psfrag{sigma}{\scriptsize $\sigma$}
\psfrag{theta}{\scriptsize $\theta$}
\psfrag{xi}{\scriptsize $\xi$}
\psfrag{mu}{\scriptsize $\mu$}
\psfrag{nu}{\scriptsize $\nu$}
\psfrag{zeta}{\scriptsize $\zeta$}
\psfrag{eta}{\scriptsize $\eta$}
\psfrag{phi}{\scriptsize $\phi$}
\psfrag{psi}{\scriptsize $\psi$}
\psfrag{hi}{\scriptsize $\chi$}
\psfrag{rho}{\scriptsize $\rho$}
\psfrag{c}{\small (b)}
\psfrag{d}{\small (c)}
\raisebox{0mm}{\small (a)}\epsfig{file=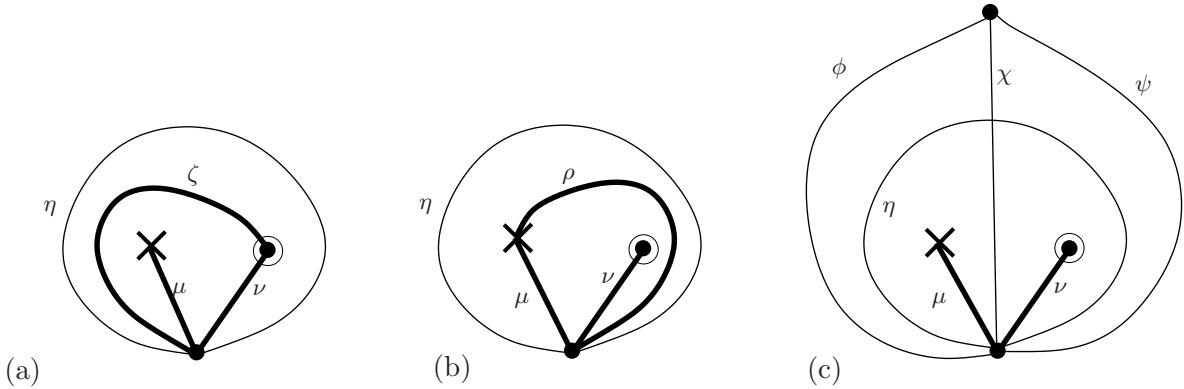,width=0.95\linewidth}
\caption{Notation for Ptolemy relations (Lemma~\ref{Ptolemy-gen})}
\label{ptolemy-gen}
\end{center}
\end{figure}

The same Ptolemy relations hold for tagged arcs as well (with the usual definition of lambda length of tagged
arc).

Now we are able to formulate the main result of the section.

Given a tagged triangulation $\h T$ of $\hat\O$,
choose an initial seed as follows:\\
\begin{itemize}
\item $B=B(\h T)$ is a signed adjacency matrix of $\h T$;
\item ${\bf x}={\bf x}(\h T)=\{\lambda(\gamma): \gamma\in \h T, \gamma\not\subset\partial\hat\O\}$;
\item ${\bf p}={\bf p}(\h T)=\{\lambda(\gamma): \gamma\in \h T, \gamma\subset\partial\hat\O\}.$

\end{itemize}

Denote by $\A(B(\hat T),{\bf x}(\hat T),{\bf p}(\hat T))$  the cluster algebra constructed by the initial seed $(B(\hat T),{\bf x}(\hat T),{\bf p}(\hat T))$. By the same arguments as in the previous sections, we get the following theorem.

\begin{theorem}
\label{thm_gen}
Let $B$ be an s-decomposable skew-symmetrizable matrix. Let $(\O^w,T_{\O^w})$ be a weighted orbifold and its triangulation constructed by an s-decomposition of $B$. Suppose also that weights of all orbifold points of $\O^w$ are equal to $2$.
Let $\hat \O$ be the associated orbifold and let $\hat T$ be the corresponding triangulation of $\hat \O$.
Then the  cluster algebra  $\A= \A(B(\h T),{\bf x}(\h T),{\bf p}(\h T))$
satisfies the following conditions:
\begin{itemize}
\item the exchange matrices of $\A$ are signed adjacency matrices of tagged triangulations of $\hat O$;
\item the cluster variables of $\A$ are lambda lengths of tagged arcs contained in triangulations of $\hat \O$;
\item the clusters consist of all lambda lengths of tagged arcs contained in the same triangulation of $\hat \O$;
\item the coefficients are lambda lengths of the boundary segments of $\hat\O$
(the coefficient semifield $\P$ is the tropical semifield generated by the lambda lengths of boundary segments);
\item the exchange graph of $\A$ coincides with the exchange graph of tagged triangulations of $\hat\O$.

\end{itemize}
\end{theorem}

\section{Laminations on the associated orbifold}
\label{sec-lam}

In this section, we define laminations on associated orbifold (cf.~\cite{FG}). Then we follow~\cite{FT} to construct geometric realization for cluster algebras with general coefficients.

\begin{definition}[{{\it Lamination}}]
\label{def-lam}
Let $\hat \O=\hat \O(M,N,Q)$ be an associated orbifold with the set of marked points $M$, the set of special marked points $N$ and the set
of orbifold points $Q$.
An {\it integral unbounded measured lamination} - in this paper just a {\it lamination} - on  an associated orbifold $\hat \O$
is a finite collection of non-self-intersecting and pairwise non-intersecting curves on $\hat \O$ modulo isotopy relative to the set
$M\cup N\cup Q$, subject to the restrictions below. Each curve must be one of the following:

\begin{itemize}
\item a closed curve;
\item a non-closed curve each of whose ends in one of the following:
\begin{itemize}
\item an unmarked point of the boundary of $\hat \O$;
\item a spiral around a puncture contained in $M$ (either clockwise or counterclockwise);
\item an orbifold point $q\in Q$;

\end{itemize}
\end{itemize}

Also, the following is not allowed:
\begin{itemize}
\item a curve that bounds an unpunctured disk or a disk containing a unique point of $M\cup N\cup Q$;
\item a curve with two endpoints on the boundary of $\hat \O$ isotopic to a piece of boundary containing no marked points
or a single marked point;
\item two curves starting at the same orbifold point (or two ends of the same curve starting at the same orbifold point);
\item curve spiraling in or starting at any special marked point $n\in N$.
\end{itemize}

\end{definition}

Our next aim is to introduce coordinates on laminations on associated orbifolds using W.~Thurstons's notion of shear coordinates extended
in~\cite{FT} to the case of tagged triangulations of surfaces. We refer to~\cite{FT} for all the details and present here only
the basic idea of shear coordinates on surfaces.

Let $S$ be a marked surface with a triangulation $T$ (containing no self-folded triangles),
let $L$ be a lamination on $S$. For each arc $\gamma$ of $T$ the corresponding
{\it shear coordinate} of $L$ with respect to the triangulation $T$, denoted by $b_\gamma (T,L)$, is defined as a sum of
contributions from all intersections of curves in $L$ with the arc $\gamma$. Such an intersection contributes $+1$ (resp, -1)
to  $b_\gamma (T,L)$ if the corresponding segment of the curve in $L$ cuts through the quadrilateral surrounding $\gamma$ as
shown in Fig.~\ref{lam} on the left (resp, on the right).

\begin{figure}[!h]
\begin{center}
\psfrag{g}{\scriptsize $\gamma$}
\psfrag{1}{\small $+1$}
\psfrag{-1}{\small $-1$}
\epsfig{file=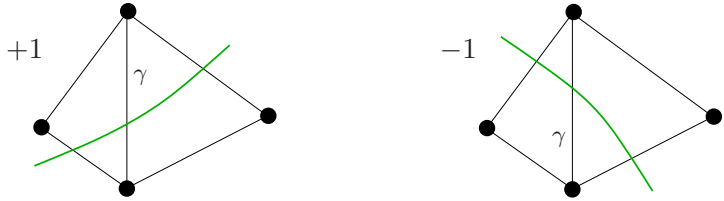,width=0.6\linewidth}
\caption{Defining the shear coordinate  $b_\gamma (T,L)$ on surfaces.}
\label{lam}
\end{center}
\end{figure}

To adjust this definition to the case of associated orbifolds, we will use the following construction.

\begin{definition}[{{\it Shear coordinates on associated orbifold}}]
\label{def-shear}
Let $\hat \O=\hat \O(M,N,Q)$ be an associated orbifold with the set of marked points $M$, the set of special marked points $N$, and the set
of orbifold points $Q$.
Let $\h T$ be a tagged triangulation of $\hat \O$.
Construct a surface $\t\O$ with a tagged triangulation $\t T$ in the following way: substitute each digon or monogon in $\h T$ containing an orbifold point or a special marked point
by a digon or monogon containing an ordinary marked point as in  Fig.~\ref{ne-unf}.
For each ordinary (i.e. non-pending and non-double) arc $\gamma\in \h T$ there exists a unique corresponding arc
$\gamma'\in \t T$, while each pending and each double arc $\gamma\in \h T$ corresponds to two conjugate arcs
 $\gamma'$ and  $\gamma''$ in  $\t T$.
Let $L$ be a lamination on $\hat \O$ and let $\t L$ be the image of this lamination on  $\t \O$ (for a curve ending at orbifold point we define its image as spiraling into the corresponding marked point counterclockwise).
Let $b_{\gamma'}(\t T,\t L)$ be shear coordinates of the lamination $\t L$ on the surface $\t \O$
with triangulation $\t T$.

Then the {\it shear coordinates} of the lamination $L$ on $\hat \O$ are defined as follows:
\begin{itemize}
\item for an ordinary arc $\gamma\in \h T$ define $b_\gamma(\h T,L)= b_{\gamma'}(\t T,\t L)$;
\item for a pending arc $\gamma\in \h T$ define
$b_\gamma(\h T,L)= b_{\gamma'}(\t T,\t L)+b_{\gamma''}(\t T,\t L)$;
\item for a double arc  $\gamma\in \h T$ define
$b_\gamma(\h T,L)=\frac{1}{2} [b_{\gamma'}(\t T,\t L)+b_{\gamma''}(\t T,\t L)]$.

\end{itemize}
\end{definition}

\begin{figure}[!h]
\begin{center}
\psfrag{g}{\scriptsize $$}
\psfrag{g'}{\tiny $$}
\psfrag{g''}{\tiny $$}
\psfrag{b}{\scriptsize $$}
\psfrag{b'}{\tiny $$}
\psfrag{b''}{\tiny $$}
\epsfig{file=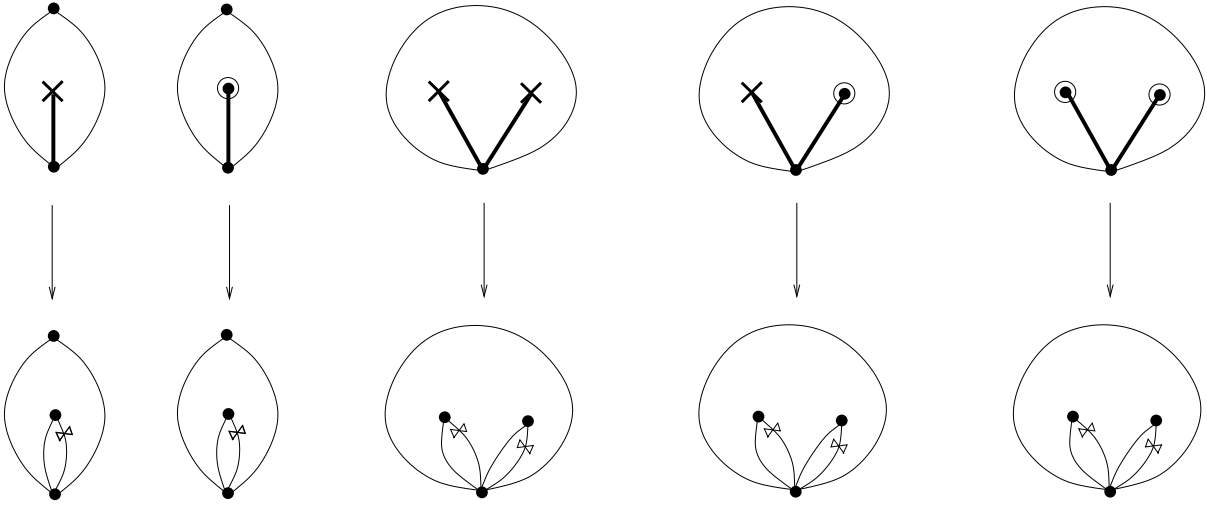,width=\linewidth}
\caption{Defining shear coordinate  $b_\gamma (\h T,L)$ for laminations on associated orbifolds: construction of the surface $\t \O$
from the orbifold $\hat \O$.}
\label{ne-unf}
\end{center}
\end{figure}

\begin{remark}
The definition of $b_\gamma(\h T,L)$ for double arcs contains a division by $2$, however,
it is easy to see that $b_\gamma(\h T,L)\in \Z$. Indeed, by Definition~\ref{def-lam} no curve of $L$ is spiraling into a special
marked point, which implies that $b_{\gamma'}(\t T,\t L)=b_{\gamma''}(\t T,\t L)$,
see Fig.~\ref{elem-lam} for the values of shear coordinates on the pairs of conjugate arcs.

\end{remark}

\begin{figure}[!b]
\begin{center}
\psfrag{gamma}{\scriptsize $\gamma$}
\psfrag{1}{\scriptsize $(-1,-1)$}
\psfrag{2}{\scriptsize $(1,1)$}
\psfrag{3}{\scriptsize $(0,-1)$}
\psfrag{4}{\scriptsize $(0,1)$}
\psfrag{5}{\scriptsize $(-1,0)$}
\psfrag{6}{\scriptsize $(1,0)$}
\epsfig{file=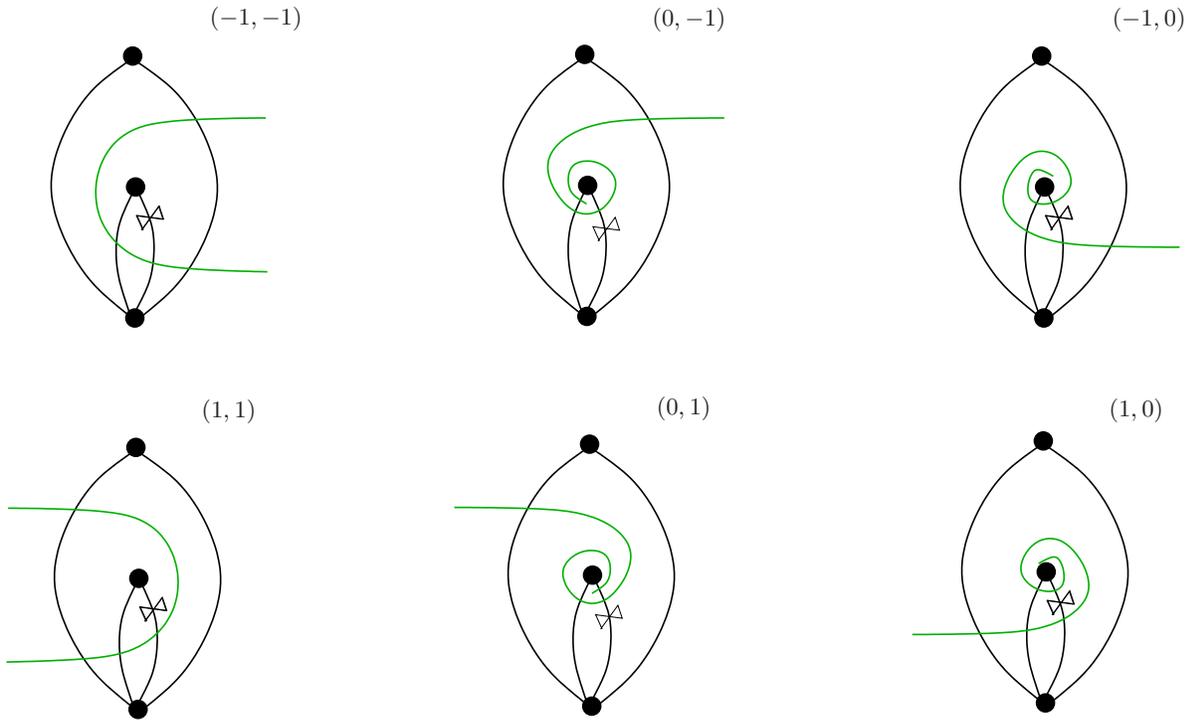,width=0.98\linewidth}
\caption{``Elementary'' laminations of a once-punctured digon: values of shear coordinates on conjugate arcs
(cf.~\cite[Fig.~36]{FT}; first we write the value for the arc tagged plain, then for the arc tagged notched).
The values coincide unless a curve of the lamination spirals into the puncture. Note that the sum of coordinates does not depend on the direction of spiraling.}
\label{elem-lam}
\end{center}
\end{figure}

In the same way as in~\cite{FT} we define {\it multi-laminations} and associated {\it extended signed adjacency matrix}.

\begin{definition}
\label{multi,extended}
A {\it multi-lamination}
is a finite set of laminations $\mathbf L=(L_{n+1},\dots,L_m)$. For a tagged triangulation $\h T$ of $\hat \O$ define an $m\times n$
{\it extended signed adjacency matrix}
$\tilde B=\tilde B(\h T,\mathbf L)= (b_{ij})$ as follows: the top $n\times n$ part of $\tilde B$ is a signed adjacency matrix
$B(\h T)=(b_{ij})_{1\le i,j\le n}$, the bottom $m-n$ rows are formed by shear coordinates of the laminations $L_i$ with respect to the triangulation $\h T$: $b_{ij}=b_j(\h T,L_i), n<i\le m$.

\end{definition}

A straightforward verification shows that, under the flips of $\h T$  the matrix $\tilde B(\h T,\mathbf L)$ transforms
according to the mutation rules:

\begin{lemma}
Let $\mathbf L$ be a multi-lamination on $\hat \O$. If tagged triangulations $\h T$ and $\h T_1$ of $\hat \O$ are related by
a flip in tagged arc $k$, then the corresponding matrices $\tilde B(\h T,\mathbf L)$ and   $\tilde B(\h T_1,\mathbf L)$
are related by a mutation in direction $k$.

\end{lemma}

Now, we will show that, given a tagged triangulation $\h T$ on $\hat \O$,
for each (ordered) $n$-tuple of numbers $(k_1,\dots,k_n)\in\Z^n$
there exists a unique lamination $L$ on $\hat \O$ such that $b_{\gamma_i}(\h T,L)=k_i$, $i=1,\dots, n$.
In Lemma~\ref{ex-uniq-prime} we show this property for the case of associated orbifolds containing no special marked points
(so that $\hat \O=\O$ is a usual orbifold). The general case will be derived from this one in Theorem~\ref{ex-uniq}.

\begin{lemma}
\label{ex-uniq-prime}
Let $\hat \O$ be an associated orbifold containing no special marked points.
For a fixed tagged triangulation $\h T$ of $\hat \O$, the map
$$ L \to (b_\gamma(\h T,L))_{\gamma \in \h T}$$
is a bijection between laminations on $\hat \O$ and $\Z^n$.

\end{lemma}

\begin{proof}
The proof is by induction on the number $q$ of orbifold points on $\hat \O$.
If $q=0$, then the statement coincides with~\cite[Theorem~13.6]{FT}.

Suppose that $q>0$.
First, assume that $\h T$ contains no monogons with orbifold points.
Let $\gamma$ be a pending arc in $\h T$, let $t\in T$ be a digon containing $\gamma$.
Denote by $\hat \O'$ an orbifold obtained from $\hat \O$ by removing $t$ and attaching a one-punctured digon $t'$ instead (as
shown in Fig.~\ref{ne-unf}).
Denote by $\h T'$ the obtained tagged triangulation of $\hat \O'$.
Let $\gamma'$ and $\gamma''$ be the conjugated tagged arcs in $t'$. For each arc $\beta\ne \gamma$ in $\h T$
we denote by $\beta'$ the corresponding arc in $\h T'$. Let $x\in t$ be the orbifold point and $x'\in t'$ be
the corresponding puncture in $\h T'$.

Note that there is a natural correspondence between laminations on $\hat\O$ not ending in $x$ and laminations on $\hat\O'$ not spiraling into $x'$. In the sequel, we will identify these laminations.

Choose an indexing $\gamma_1,\dots,\gamma_n$ of arcs of $\h T$ so that $\gamma=\gamma_n$,
and choose an $n$-tuple  $(k_1,\dots,k_n)\in \Z^n$.
Our aim is to find a lamination $L$ on $\hat \O$  such that  $b_{\gamma_i}(\hat T,L)=k_i$, $i=1,\dots,n$.
Suppose that $k_n=2\bar k_n$ (resp, $k_n=2\bar k_n+1$).
By inductive assumption, there exists a unique lamination $L'$ on $\hat O'$ such that
$$
\begin{array}{ccl}
b_{\gamma_i'}(\hat T',L') &=& k_i, \ \ \  \text{for \ $i=1,\dots,n-1$};\\
b_{\gamma_n'}(\hat T',L')& = &\bar k_n \\
b_{\gamma_n''}(\hat T',L')& =& k_n-\bar k_n.\\
\end{array}
$$
Consider the elementary laminations on a once punctured digon $t'$ (see Fig.~\ref{elem-lam}): their shear coordinates on the conjugate arcs
$\gamma$ and $\gamma'$
are distinct if and only if the curve is spiraling into a puncture. Moreover, if there exists a curve spiraling into a puncture in
a clockwise direction, then the same lamination can not contain a curve spiraling counter-clockwise into the same puncture.
Since $|b_{\gamma_n''}(\h T',L')-b_{\gamma_n'}(\h T',L')|\le 1$, the lamination $L'$ either contains a unique curve spiraling into
$x'$ or contains no such a curve. In the latter  case define $L=L'$. In the former case we do the same, substituting in addition
the curve spiraling into $x'$ by a curve ending in $x$. It is easy to see that the obtained set of curves $L$ is a lamination on $\hat \O$
and that  $b_{\gamma_i}(\h T,L)=k_i$, $i=1,\dots,n$.

Therefore, we have shown the existence of a lamination on $\hat \O$ with given shear coordinates. It is left to prove the uniqueness.
Suppose that $L_1$ and $L_2$ are two laminations with the same shear coordinates with respect to $\h T$.
For each lamination $L_i$, $i=1,2$ on $\hat \O$ we build (a unique) lamination $L'_i$, $i=1,2$ on $\hat \O'$ in the following way:
if $L$ contains no curve ending at $x$, then we take $L'_i=L_i$; otherwise, we substitute the curve of $L_i$ ending in $x$ by a curve spiraling in $x'$ in a clockwise direction.
Clearly, if $L_1\ne L_2$, then $L_1'\ne L_2'$. On the other hand,
it follows immediately from Definition~\ref{def-shear} that the shear coordinates of $L'_1$ and $L_2'$ coincide.
This contradicts to the inductive assumption (namely, its  ``uniqueness'' part).

So, the lemma is proved for the case of triangulations without monogons with orbifold points.
The case of triangulation $\h T$ containing monogons is treated in the same way.

\end{proof}

Now we obtain the following theorem.

\begin{theorem}
\label{ex-uniq}
Let $\hat \O$ be an associated orbifold.
For a fixed tagged triangulation $\h T$ of $\hat \O$, the map
$$ L \to (b_\gamma(\h T,L))_{\gamma \in \h T}$$
is a bijection between laminations on $\hat \O$ and $\Z^n$.

\end{theorem}

\begin{proof}
If $\hat \O$ contains no special marked points, then the theorem follows from Lemma~\ref{ex-uniq-prime}.
To prove the theorem in the general case, substitute all the special marked points of $\hat\O$ by ordinary punctures, and all double arcs by pairs of conjugate arcs (as in Definition~\ref{def-shear}). Denote the new triangulation by $\tilde T$ and new orbifold by $\tilde \O$.

For every lamination $L$ on $\hat \O$ we construct in a natural way a unique lamination $\tilde L$ on $\tilde \O$.
Since no lamination is spiraling into a special marked point, one can note that if $\gamma\in\h T$ is a double arc and $(\tilde\gamma',\tilde\gamma'')$ is a pair of corresponding conjugate arcs in $\tilde T$, then $b_{\gamma}(\h T,L)=b_{\tilde\gamma'}(\tilde T,\tilde L)=b_{\tilde\gamma''}(\tilde T,\tilde L)$ (see Fig.~\ref{elem-lam}). On the other hand,
if $\tilde L_0$ is a lamination on $\tilde \O$ having the same shear coordinates for pair of conjugate arcs $(\tilde\gamma',\tilde\gamma'')$ then $\tilde L_0$ contains no curves spiraling into the endpoint of conjugate arcs with different tags, which implies that $\tilde L_0$ can be obtained from some lamination $L_0$ on $\hat \O$ by the procedure above. 
\end{proof}

\section{Opened associated orbifolds}
\label{sec-opened}
Following~\cite{FT}, our next step to the realization of cluster algebras with general coefficients is
``opening'' of interior marked points of the associated orbifold (resp. of the surface in case of skew-symmetric case).
{ 
The aim is to prepare the ground for introducing new geometric quantities which will serve as cluster variables satisfying new Ptolemy relations taking into account general coefficients.
}
At this step the construction for the orbifold case coincides with the one for the surfaces.
We briefly reproduce all necessary definitions and refer to~\cite{FT} for the details and examples.

\begin{definition}[{{\it Opening of an associated orbifold}}]
\label{opened}
Let $\hat \O=\hat\O(M,N,Q)$ be an associated orbifold, where $M$ is the set of marked points, $N$ is the set of special marked
points and $Q$ is the set of orbifold points. Let $\overline M=M\setminus \partial \O$ be the set of punctures , i.e. marked points in the interior of $\hat \O$. For a subset $P\subset \overline M$, the corresponding {\it opened associated orbifold} $\hat \O_P$
is obtained from $\hat \O$ by removing a small open disk around each point in $P$. For $p\in P$ let $C_p$ be the boundary
component of $\hat \O_P$ created in this way. For each $C_p$ one introduces a new marked point $z_p\in C_p$ and set
$M_P=(M\setminus P)\cup \{z_p\}_{p\in P}$ creating a new associated orbifold $\hat \O_P=\hat \O_P(M_P,N,Q)$ with
$|\overline M|-|P|$ punctures (here $|X|$ is the cardinality of the set $X$).

\end{definition}

It is easy to see that $\hat \O_P$ can be constructed as an orbifold associated to weighted orbifold $\O_P^w$ obtained from $\O^w$ by opening of the set $P$ of punctures.

We denote by $\hat \O_{\overline M}$ the opened orbifold $\hat \O_P$ in case $P=\overline M$.

\begin{definition}[{{\it Lifts of tagged arcs}}]
\label{lifts}
One can define a natural map $\hat \O_P\to\hat\O$ by collapsing every $C_p$ to a point. Following~\cite{FT}, we call by a {\it lift}
of a curve $\alpha$ on $\hat \O$ any curve $\bar\alpha$ on $\hat \O_P$ that projects to $\alpha$ under the map above.

A lift may not be unique: if we denote by $\psi_p$ a twist around $C_p$ on $\hat \O_P$, and $\alpha$ ends in $p\in M$, then for given lift $\bar\alpha$ every curve $\psi_p^n\bar\alpha$ will also be a lift. Moreover, if the second end of $\alpha$ does not belong to $P$, then the set
$\{\psi_p^n\bar\alpha\}_{n\in\Z}$ will contain all the lifts of $\alpha$. Similarly, if $\alpha$ ends in two points of $P$, say $p$ and $q$, then all the lifts of $\alpha$ can be written as $\{\psi_p^n\psi_q^m\bar\alpha\}_{n,m\in\Z}$ for some particular lift $\bar\alpha$ of $\alpha$.

The {\it lifts of tagged arcs} have the same tags as their preimages.
\end{definition}

\begin{definition}
\label{partial teichm}
For a subset $P\subset \overline M$, a {\it partial Teichm\"uller space} $\T(\hat \O_P)$,
where $\hat \O_P=\hat \O_P(M_P,N,Q)$,
is the space of all finite-volume complete hyperbolic metrics on $\hat \O_P\setminus (Q\cup N\cup M\setminus P)$
with geodesic boundary and with cone points of angle $\pi$ in $Q$, modulo isotopy. A point of {\it decorated partial Teichm\"uller space}
 $\t \T(\hat \O_P)$ is the space of metrics from $\T(\hat \O_P)$ modulo isotopy relative to ${C_p},\,p\in P$,
with a choice of horocycle around each point in $M\setminus P$ and a self-conjugated horocycle around each point of $N$.

\end{definition}

The hyperbolic structure in  $\t\T(\hat \O_P)$ has a cusp at every point of $N\cup(M\setminus P)$, and an orbifold point with angle $\pi$
at each point of $Q$. The boundary components coming from $\hat \O$ are of infinite length, while the new components $C_p$
($p\in P$) are of finite length: there are no cusps in the points $z_p$ introduced for the new boundary components.

Given $P\subset \overline M$, an orientation on $C_p$ for each $p\in P$ and a decorated hyperbolic structure $\sigma\in \t\T(\hat \O_P)$,
one can build for each arc $\gamma$ on $\hat \O$ a unique non-self-intersecting geodesic $\gamma_\sigma$ on $\hat \O_P$:
if an endpoint of $\gamma$ belongs to $P$, then the corresponding end of $\gamma_\sigma$ spirals around $C_p$ in the direction of orientation of $C_p$, otherwise the end of $\gamma_\sigma$ does the same  as the corresponding end of $\gamma$ (i.e. runs into a cusp or ends in an
orbifold point).

The {\it tagged } arcs are represented in the following way: the ends tagged notched are represented by geodesics spiraling against the
chosen direction, while the ends tagged plain are spiraling in the chosen direction.

For each $p\in P$, there is a {\it perpendicular horocyclic segment} $h_p$ near $C_p$: this is a (short) segment of the horocycle from
$z_p\in C_p$ perpendicular to $C_p$ and to all geodesics $\gamma_\sigma$ spiraling into $C_p$.
For the tagged arcs, one also introduces the {\it conjugate} perpendicular horocyclic segment $\bar h_p$ which satisfies the same
requirements as $h_p$ with respect to the geodesics spiraling in the opposite direction.

\begin{definition}[{{\it Lambda length on an opened associated orbifold}}]
\label{lambda opened}
For a tagged arc $\gamma_\sigma$ on $\hat \O_P$ the lambda length is $\lambda(\gamma_\sigma)=e^{l(\gamma_\sigma/2)}$,
where $l(\gamma_\sigma)$ is a distance between appropriate intersections of the geodesic $\gamma_\sigma$ with the horocycles at its two
ends (or, in case of a geodesic ending in an orbifold point, a doubled distance between the orbifold point and the appropriate
intersection with the horocycle at another end). In case of an end spiraling around one of the openings $C_p$ there are
infinitely many intersections of $\gamma_\sigma$ with the perpendicular horocyclic segment $h_p$ (or a conjugate perpendicular horocyclic
segment $\bar h_p$ in case of the notched tagging), so one needs to choose the intersection.

To choose the correct intersection (in the case when $\gamma_\sigma$ twists
sufficiently far around the openings),
consider an auxiliary curve $\hat \gamma$ obtained from $\gamma_\sigma$ by deleting the spiraling ends (from a given
 point of the intersection with the horosphere) and attaching the segments of the horosphere instead of it.
The points of intersections then should be chosen in a way that $\hat \gamma$ is homotopic to the given lift  $\bar \gamma\subset \hat \O_P$.

To extend the definition to all tagged arcs, not only ones that twist sufficiently many times, one uses formula
$$ l(\psi_p\bar \gamma_\sigma)=l(p)+l(\bar \gamma_\sigma),$$
where $\gamma_\sigma$ is an arc spiraling around $C_p$, $p\in P$,
$\psi_p$ is a clockwise twist around the component $C_p$, and
$$l(p)=
\left\{
\begin{array}{ll}
-\text{length of $C_p$} & p\in P, \text{ if $C_p$ is oriented counterclockwise};\\
0 & \text{if $p\notin P$};\\
\text{length of $C_p$} & p\in P, \text{ if $C_p$  is oriented clockwise}.   \\

\end{array}
\right.
$$

\end{definition}

The correctness of this definition can be checked directly.

\begin{definition}
\label{complete teichm}
The {\it complete decorated Teichm\"uller space}  $\overline \T(\hat \O)=\overline \T(\hat \O(M,N,Q))$ is
a disjoint union over all subsets $P\subset \overline M$ of $2^{|P|}$ copies of  $\t \T(\hat \O_P)$, one for each choice of
orientation on each boundary circle $C_p$ (so that $\overline \T(\hat \O)$ consists of $3^{|\overline M|}$ strata of type
$\t \T(\hat \O_P)$).
The topology on  $\overline \T(\hat \O)$ is the weakest in which all the lambda lengths $\lambda(\gamma)$
are continuous. Here $\lambda(\gamma)$ is defined for all $\gamma\in \hat \O_{\overline M}$ at each point of every stratum $\t \T(\hat \O_P)$ in the following way: we project $\gamma$ to $\hat \O_P$ by contracting remaining circular components $C_q$, $q\notin P$, and compute the lambda length of the obtained arc by Definition~\ref{lambda opened}.

\end{definition}

\begin{prop}
\label{prop 9.5}
Let $\h T$ be a tagged triangulation of $\hat \O$. For each $\gamma\in \h T$, fix an arc $\bar \gamma\in \hat \O_{\overline M}$
that projects to $\gamma$. Then the map
$$\Phi=\left(\prod_{p\in \overline M}\lambda(p) \right)\times
\left(\prod_{\beta\subset \partial \hat \O}\lambda(\beta) \right) \times
\left(\prod_{\gamma\in T}\lambda(\bar \gamma) \right):
\overline \T(\hat \O)\to \R^{n+c+|M|}_{>0}
$$
is a homeomorphism.

\end{prop}

For surfaces the proposition is proved in~\cite[Proposition~10.10]{FT} (more precisely, it is proved for ideal
triangulations and then extended to the case of tagged triangulations). The proposition extends to the case of associated orbifolds
without any changes.

Denote by $\P=\P(\hat \O)$ the free abelian multiplicative group generated by the set
\begin{equation}
\label{coefficients}
\{ \lambda(p): p\in \o M\}\cup \{\lambda(\beta): \beta\subset \partial \hat \O  \}
\end{equation}
of lambda lengths of boundary components $\beta$ and circular components $C_p$.
In view of Proposition~\ref{prop 9.5}, these lambda lengths can be viewed either as functions on $\overline \T(\hat \O)$ or as formal
variables.

\begin{definition}
The {\it rescaling factors} are defined for the marked points $a\in M_{\overline M}\cup N\cup Q$ on the opened associated orbifold
$\hat \O_{\overline M}$ by
$$ \nu(a)=\left\{
\begin{array}{ll}
\sqrt{1-\lambda_+(p)^{-2}} & \text{if $a=z_p$, with $p\in\overline M$;}\\
1& \text{otherwise,}

\end{array}
\right.
$$
where $\lambda_+(p)=\exp(\text{length of $C_p$})/2$.

\end{definition}

For a tagged arc $\bar \gamma\in \hat \O_{\overline M}$ with endpoints $a,b\in  M_{\overline M}\cup N\cup Q$
define the {\it rescaled lambda lengths} as
\begin{equation}
\label{rescaled lambda length}
x(\bar \gamma)=
\left\{
\begin{array}{ll}
\lambda(\bar \gamma)\nu(a)^2\nu(b)^2& \text{if either $a\in Q$ or $b\in Q$;}\\
\lambda(\bar \gamma)\nu(a)\nu(b)& \text{otherwise.}\\

\end{array}
\right.
\end{equation}

\begin{remark}
The definition of $x(\bar \gamma)$ is the only place in the current section where we need to introduce some changes for the orbifold
settings. On the other hand, one may interpret pending arcs as ones ``coming to the orbifold point and then going back''.
Then none of the endpoints of the pending arc is an orbifold point, and both ends $a$ and $b$ coincide,
so that the the formula $x(\bar \gamma)=\lambda(\bar \gamma)\nu(a)\nu(b)$ holds for this case as well.

\end{remark}

It is shown in~\cite{FT} that in the surface settings
the rescaled functions $x(\bar \gamma)$ satisfy the same Ptolemy relations as lambda lengths do.
It is a straightforward computation that the same also holds for rescaled functions in the settings of associated orbifolds
(one needs to check the relations listed in Lemmas~\ref{Ptolemy-prime},~\ref{Ptolemy-local} and~\ref{Ptolemy-gen}).

For each tagged arc $\gamma\in \hat \O$ we fix an arbitrary lift $\bar \gamma\in \hat \O_{\overline M}$ (see Definition~\ref{lifts}),
and set $x(\gamma)=x(\bar \gamma)$, where $x(\bar \gamma$) is  defined by~(\ref{rescaled lambda length}).
Then for each tagged triangulation $\h T$ of the associated orbifold $\hat \O$ define
\begin{equation}
\label{rescaled var}
\begin{array}{l}
\mathbf x(\h T)=
\{x(\bar \gamma):\gamma\in \h T\}.
\end{array}
\end{equation}

Proposition~\ref{prop 9.5} shows that the rescaled lambda lengths in $\mathbf x(\hat T)$ can be treated as formal variables
algebraically independent over the field of fractions of $\P(\hat \O)$.

\begin{theorem}
\label{thm-opened}
For an arbitrary choice of lifts $\bar \gamma$ of tagged arcs $\gamma\in \hat \O$ there exists a (unique) cluster algebra
$\A$ with the following properties:
\begin{itemize}
\item the coefficient group is $\P=\P(\hat \O)$ (see~(\ref{coefficients}));
\item the cluster variables are the rescaled lambda lengths $x(\bar \gamma)$ defined by~(\ref{rescaled lambda length});
\item the cluster $\mathbf x(\hat T)$ is given by~(\ref{rescaled var});
\item the ambient field is generated over $\P$ by some (equivalently, any) cluster $\mathbf x(\hat T)$;
\item the exchange matrices are the signed adjacency matrices $B(\hat T)$;
\item the exchange relations out of each seed are relations associated with the corresponding tagged flips, properly rescaled using {  Definition~\ref{lambda opened}} to reflect the choices of lifts.


\end{itemize}
\end{theorem}

The proof of the theorem for surfaces is given in~\cite{FT} (see Theorem 11.1), and does not require any changes in the orbifold settings.

\section{Tropical lambda lengths and laminated Teichm\"uller spaces}
\label{sec-teichm}
In this section, we reproduce for reader's convenience the definitions from~\cite{FT} almost with no changes (except for
Definition~\ref{intersections} where we need to generalize the definition of number of intersections to the case of pending and double
arcs on the orbifold).
{ The laminated lambda lengths on opened associated orbifold introduced in this section will serve as cluster variables in geometric realization of cluster algebras with s-decomposable exchange matrices with arbitrary coefficients.}

\begin{definition}[{{\it Lifts of laminations}}]
\label{lift of lam}
In order to lift a lamination $L$ on $\hat \O$ to a (non-unique) lamination $\bar L$ on an opened orbifold $\hat \O_P$,
for each $p\in P$ we do the following: we replace each curve in $L$ that spirals into $p$ by a new curve in $\bar L$
that run into a point on $C_p$ so that different curves do not intersect and their ends lying on $C_p$ are pairwise distinct and
different from $z_p$. The (ends of) curves in $L$ that do not spiral into opened punctures are lifted ``tautologically''.

A {\it lifted multi-lamination} $\overline {\mathbf L}$ consists of (uncoordinated) lifts of the individual laminations contained in ${\mathbf L}$.

\end{definition}

\begin{definition}[{{\it Intersection numbers on opened orbifolds }}]
\label{intersections}
Denote by $|\bar L\cap \gamma|$ the (geometric, i.e. non-negative) number of intersections of a tagged arc
$\gamma\subset \hat \O_{\overline M}$ with the curves of the lifted lamination $\bar L\subset \hat \O_{\overline M}$
(to find the ``number of intersections''
we need to choose the curves in the corresponding homotopy classes that minimize this number, for example, geodesics for
some hyperbolic structure on $\hat O_{\overline M}$).

If $\gamma$ is a pending arc, then each intersection with the inner part of $\gamma$ counts with multiplicity 2,
and an intersection at the orbifold point counts with multiplicity 1.
All intersections with double arcs count with multiplicity 1.

\end{definition}

\begin{definition}[{{\it Transverse measures}}]
\label{transverse}
For a tagged arc $\gamma\subset \hat \O_{\overline M}\cup \partial \hat \O$  and a lift $\bar L\subset \hat \O_{\overline M}$ of
a lamination $L$, the {\it transverse measure} of $\gamma$ with respect to $\bar L$ is an integer $l_{\bar L}(\gamma)$
defined as follows:
\begin{itemize}
\item $l_{\bar L}(C_p)$ is the number of ends of curves in $\bar L$ that lie on $C_p$;
\item if $\gamma$ does not have ends at boundary components $C_p$ ($p\in \overline M$), then
$l_{\bar L}(\gamma)=|\bar L\cap \gamma|$;
\item if $\gamma$ has one or two such ends, and twists sufficiently many times around the opening(s) in the direction consistent with their
orientation (respectively, in the opposite directions if the end is notched), then again,
$l_{\bar L}(\gamma)=|\bar L\cap \gamma|$;
\item otherwise, $l_{\bar L}(\gamma)$ is defined using the cases above and the formula
$$l_{\bar L}(\psi_p\gamma)=(-1)^{t}l_{\bar L}(p)+l_{\bar L}(\gamma),$$
where
\begin{itemize}
\item[$-\ $] $\psi_p$ is a clockwise twist around $C_p$,
\item[$-\ $] $t=
\left\{
\begin{array}{ll}
0 &\text{if $\gamma$ tagged plain at $z_p$;}  \\
1 &\text{if $\gamma$ tagged notched at $z_p$;}
\end{array}
\right.
$
\item[$-\ $]
$
l_{\bar L}(p)=
\left\{
\begin{array}{ll}
-l_{\bar L}(C_p) &\text{if $p\in \overline M$ and $C_p$ is oriented counterclockwise;}  \\
0 & \text{if $p\notin \overline M$};\\
l_{\bar L}(C_p) &\text{if $p\in \overline M$ and $C_p$ is oriented clockwise.}  \\
\end{array}
\right.
$

\end{itemize}
\end{itemize}
\end{definition}

\begin{definition}[{{\it Tropical semifield associated with a multi-lamination}}]
Let $\mathbf L=(L_i)_{i\in I}$ be a multi-lamination on $\hat \O$, here $I$ is a finite indexing set.
Let $q_i$ be a formal variable for each lamination $L_i$, and let
$$\P_{\mathbf L}={\mathrm{Trop}}(q_i: i\in I)$$
be the multiplicative group of Laurent monomials in variables $\{q_i: i\in I\}$.
Addition $\oplus$ is defined by
$$\prod\limits_i q_i^{a_i}\oplus\prod\limits_i q_i^{b_i}=\prod\limits_i q_i^{\min(a_i,b_i)}$$
$\P_{\mathbf L}$ is called {\it tropical semifield} associated with multi-lamination ${\mathbf L}$.

\end{definition}

\begin{definition}[{{\it Tropical lambda lengths}}]
\label{tropical lambda}
Let $\overline{\mathbf L}=(\bar L_i)_{i\in I}$ be a lift of a multi-lamination $\mathbf L$.
The {\it tropical lambda length} of a tagged arc $\gamma\subset \hat O_{\overline M}\cup \partial \hat \O$ with respect to
$\overline L$ is
$$
c_{\overline {\mathbf L}}(\gamma)=\prod\limits_{i\in I} q^{-l_{\overline L_i}(\gamma)/2}\in \P_{\mathbf L}.
$$

\end{definition}

\noindent
Tropical lambda lengths satisfy the equality
$$
c_{\overline {\mathbf L}}(\psi_p\gamma)=c_{\overline {\mathbf L}}(p)^{t}c_{\overline {\mathbf L}}(\gamma),
$$
where $t=0$ if $\gamma$ tagged plain and $t=1$ if $\gamma$ tagged notched at $z_p$,  and
$$
c_{\overline {\mathbf L}}(p)=c_{\mathbf L}(p)=\prod\limits_{i\in I} q^{-l_{\overline L_i}(p)/2}\in \P_{\mathbf L}.
$$

Tropical lambda lengths of boundary segments, holes, or arcs that are not incident to punctures do not depend on the choice of
a lift $\overline {\mathbf L}$. So, we can use the notation $c_{\mathbf L}(\beta)=c_{\overline {\mathbf L}}(\beta)$ for
$\beta\subset \partial \hat \O$, or  $c_{\mathbf L}(p)=c_{\overline {\mathbf L}}(p)$ for $p\in \overline M$.

Similarly to ordinary lambda lengths, tropical lambda length of a given arc does not depend on a tagged triangulation containing the arc.

The main property of tropical lambda lengths is that they satisfy the tropical version of Ptolemy relations:
to obtain tropical version of an expression containing operations ``$\cdot$'' and ``$+$''
(multiplication and addition), one substitutes multiplication $c\cdot b$ and addition $a+b$ by addition $a\oplus b$ and
minimum $\min(a,b)$ respectively.
For example, the relation $e\cdot f=a\cdot c+b\cdot d$ turns into $e\oplus f=max(a\oplus c, b\oplus d)$.
It is shown in~\cite{FT} that tropical lambda lengths satisfy tropical versions of Ptolemy relations in the surface settings.
To adjust the statement to the orbifold settings one needs to check the relations listed in Lemmas~\ref{Ptolemy-prime},~\ref{Ptolemy-local} and~\ref{Ptolemy-gen}.
This is a straightforward verification based on the computing of the intersection numbers.


\begin{definition}[{{\it Laminated Teichm\"uller space}}]
\label{laminated teichm}
For a multi-lamination $\mathbf L=(L_i)_{i\in I}$ on $\hat\O$, the {\it laminated Teichm\"uller space}
$\overline \T(\hat \O, \mathbf L)$ is defined as follows.
A point $(\sigma,q)\in \overline \T(\hat \O, \mathbf L)$  is a decorated hyperbolic structure $\sigma\in \overline \T(\hat \O)$
together with a collection of positive real weights $q=(q_i)_{i\in I}$ which are chosen so that the following boundary conditions hold:

\begin{itemize}
\item for each boundary segment $\beta\subset \partial \hat \O$, we have $\lambda_\sigma(\beta)=c_{\mathbf L}(\beta)$;
\item for each hole $C_p$ with $p\in \overline M$, we  have $\lambda_\sigma(p)=c_{\mathbf L}(p)$.

\end{itemize}

\end{definition}

\begin{prop}
\label{14.2}
Let $\mathbf L=(L_i)_{i\in I}$ be a multi-lamination in $\hat \O$ and let $\h T$ be a tagged triangulation of $\hat \O$.
Choose a lift $\overline {\mathbf L}=(\overline L_i)$ of $\mathbf L$, and choose a lift of each of the $n$ arcs $\gamma \in \h T$
to an arc on $\hat \O_{\overline M}$. Then the map
$$
\Psi: \overline \T(\hat \O, \mathbf L)\to \R^{n+|I|}_{>0}
$$
defined by
$$
\Psi(\sigma,q)=(\lambda_\sigma(\bar \gamma))_{\gamma\in \h T}\times q
$$
is a homeomorphism. The same is true with the lambda lengths $\lambda(\bar \gamma)$ replaced by their rescaled versions $c(\bar \gamma)$
defined by~(\ref{rescaled lambda length}).

\end{prop}

The proposition follows from Proposition~\ref{prop 9.5} and Definition~\ref{laminated teichm}.

\begin{definition}[{{\it Laminated lambda lengths}}]
\label{laminated lambda}
Fix a lift $\overline {\mathbf L}$ of a multi-lamination $\mathbf L$. For a tagged arc $\gamma$ on $\hat O$, the {\it laminated
lambda length} $x_{\overline {\mathbf L}}(\gamma)$ is a function on the laminated Teichm\"uller space $\overline \T(\hat \O, \mathbf L)$
defined by
$$
x_{\overline {\mathbf L}}(\gamma)=x(\bar \gamma)/c_{\overline {\mathbf L}}(\bar \gamma),
$$
where $\bar\gamma$ is an arbitrary lift of $\gamma$, $x(\bar \gamma)$ is a rescaled lambda length defined
by~(\ref{rescaled lambda length}), and $c_{\overline {\mathbf L}}(\bar \gamma)$ is the tropical lambda length as in
Definition~\ref{tropical lambda}.

\end{definition}

The value of $x_{\overline {\mathbf L}}(\gamma)$ does not depend on the choice of the lift $\bar \gamma$ since
$x(\bar \gamma)$ and $c_{\overline {\mathbf L}}(\bar \gamma)$ are rescaled by the same factor $c_{\mathbf L}(p)=\lambda_\sigma(p)$
as $\bar \gamma$ twists around the opening $C_p$.

\begin{cor}
For a tagged triangulation $\h T$ of $\hat \O$, the map
$$
\begin{array}{rll}
\overline \T(\hat \O, \mathbf L) &\to & \R^{n+|I|}_{>0}\\
(\sigma,q) &\mapsto & (x_{\overline {\mathbf L}}(\gamma))_{\gamma\in\h T}\times q\\
\end{array}
$$
is a homeomorphism. (As before, $\overline {\mathbf L}$ is a lift of a multi-lamination $\mathbf L$).

\end{cor}

\section{Cluster algebras with arbitrary coefficients associated with orbifolds}
\label{main}
Now, we are ready to present the main construction: a geometric realization of any cluster algebra with $s$-decomposable exchange matrix.

Let us recall the main idea of~\cite{FT}. For
 a suitable choice of coefficients the corresponding cluster variables can be identified with lambda lengths (see also Section~\ref{sec-geom}). To introduce principal (and in view of~\cite{FZ4} arbitrary) coefficients special laminations on the surface are used. However, a price needs to be paid for introducing principal coefficients since standard lambda lengths do not satisfy cluster relations with coefficients. In order to compensate disturbance caused by coefficients one needs to replace lambda lengths by \emph{laminated lambda lengths}. Geometric meaning of laminated lambda length is a ratio of lambda length and its tropical limit (see~\cite{FT} and Section~\ref{sec-teichm}). The whole construction can be transferred with almost no changes to orbifolds.

In short, the cluster variables in the construction are interpreted as laminated lambda lengths
$x_{\overline {\mathbf L}}(\gamma)$ of tagged arcs on the associated orbifold,
while coefficients $q_i$ are functions on the laminated Teichm\"uller space $\overline \T(\hat \O, \mathbf L)$.

Before stating the theorem, we shortly recall the whole construction. The stars label the steps where we need to adjust construction from \cite{FT} to orbifold case (in contrast to the other steps, where the definitions are copied straightforwardly from the surface case).

\begin{itemize}
\item[$1^*.$] Define s-decomposable skew-symmetrizable matrix via s-decomposable diagrams (see Definitions~\ref{s-dec diagr},~\ref{s-dec matr}).
\item[$2^*.$] Given an s-decomposable skew-symmetrizable matrix $B$, construct weighted diagram $\D(B)$ and weighted orbifold $\O^w$ (see Definitions~\ref{s-dec diagr},~\ref{s-dec matr}).
\item[$3^*.$] Given a weighted orbifold $\O^w$ build an associated orbifold $\hat \O$  as in Definition~\ref{def-ass-gen}.
\item[$4^*.$] Define lambda length $\lambda(\gamma)$  of an arc $\gamma$ on $\hat \O$ (see  Definition~\ref{lambda}) and a decorated Teichm\"uller
  space $\widetilde \T(\hat \O)$ (Definition~\ref{dec teichm}).
\item[$5^*.$] Define laminations on the associated orbifold (Definition~\ref{def-lam}) and shear coordinates
(Definition~\ref{def-shear}).
\item[$6$.] Define an opened associated orbifold $\hat \O_P$ (Definition~\ref{opened}), and lifts $\bar \gamma$ of an arc
$\gamma\in \hat \O$ and  $\overline L$ of a lamination $L$ from $\hat \O$ to $\hat \O_{P}$
(see Definition~\ref{lifts} and~\ref{lift of lam} respectively). Define partial Teichm\"uller space $\widetilde \T(\hat \O_P)$
(Definition~\ref{partial teichm}) and complete Teichm\"uller space $\overline T(\hat O)$ (Definition~\ref{complete teichm}).
\item[$7$.] Define   lambda length $\lambda(\bar \gamma)$ on the opened orbifold $\hat\O_P$
(Definition~\ref{lambda opened}), as well as rescaled lambda length $x(\bar \gamma)$ (see~(\ref{rescaled lambda length})).
\item[$8$.] Define transverse measure $l_{\bar L}(\gamma)$ (Definition~\ref{transverse}) and tropical lambda length
$c_{\overline {\mathbf L} }(\gamma)$ (Definition~\ref{tropical lambda}).
\item[$9$.]
Define laminated Teichm\"uller space $\overline\T(\hat \O,\mathbf L )$ (Definition~\ref{laminated teichm})
and laminated lambda length $x_{\overline {\mathbf L}}(\gamma)$ (Definition~\ref{laminated lambda}).

\end{itemize}

Now, we are able to state the main theorems.

\begin{theorem}
\label{14.6}
For a given associated orbifold $\hat \O$ with a given
multi-lamination $\mathbf L=(L_i)_{i\in I}$ on $\hat \O$, there exists a unique cluster algebra $\A$ of geometric type
with the following properties:
\begin{itemize}
\item the coefficient semifield is the tropical semifield $\P_{\mathbf L}=Trop(q_i:i\in I)$;
\item the cluster variables $x_{\mathbf L}(\gamma)$ are labeled by the tagged arcs $\gamma$ on $\hat \O$;
\item the extended exchange matrix  $\tilde B(\h T,\mathbf L)$ is the extended signed adjacency matrix described in Definition~\ref{multi,extended};
\item the seeds $(\mathbf x(\hat T),\tilde B(\h T,\mathbf L))$ are labeled by tagged triangulations $\hat T$ of $\hat \O$;
\item the exchange graph is the graph of flips of tagged triangulations of $\hat\O$.

\end{itemize}

This cluster algebra has a realization by functions on the laminated Teichm\"uller space $\overline \T(\hat \O,\mathbf L)$.
To obtain this realization, choose a lift $\overline {\mathbf L}$ of the multi-lamination $\mathbf L$; then represent each cluster
variable $x_{\mathbf L}(\gamma)$ by the corresponding laminated lambda length $x_{\overline {\mathbf L}}(\gamma)$, and each coefficient
variable $q_i$ by the corresponding function on  $\overline \T(\hat \O,\mathbf L)$.

\end{theorem}

\begin{cor}
\label{5.1}
Let $T_0$ be a tagged triangulation consisting of $n$ tagged arcs in $\hat\O$.
Let $\Sigma_0=({\bf x}(T_0),{\bf p}(T_0),B(T_0))$ be a triple such that
\begin{itemize}
\item ${\bf x}(T_0)$ is an $n$-tuple of formal variables labeled by the arcs in $T_0$;
\item ${\bf p}(T_0)$ is an $2n$-tuple of elements $(p^\pm(e))$, $e$ is an edge of $T_0$ of a tropical semifield ${\mathbb P}$ satisfying normalization condition $p^+(e)\oplus p^-(e)=1$ ;
    \item $B(T_0)$ is the signed adjacency matrix of $T_0$
\end{itemize}

\noindent Then there is a unique normalized cluster algebra $\A$ with initial seed $\Sigma_0$ whose exchange graph coincides with the exchange graph of tagged triangulations of $\hat\O$.

Moreover,
\begin{itemize}
\item the seeds are labeled by tagged triangulations of $\hat\O$;
\item the cluster variables are labeled by tagged arcs and each cluster variable $x_\gamma=x_\gamma(T)$ does not depend on tagged triangulation $T$.
\end{itemize}
\end{cor}

\begin{cor}
\label{5.2}
Let $\A$ be a cluster algebra whose exchange matrix is s-decomposable. Then
\begin{itemize}
\item each seed in $\A$ is uniquely determined by its cluster;
\item the cluster complex and the exchange graph $\mathbf E$ of $\A$ do not depend on the choice of coefficients in $\A$;
\item the seeds containing a given cluster variable form a connected subgraph in $\mathbf E$;
\item several cluster variables appear together in the same cluster if and only if every pair among them does;
\item the cluster complex is the complex of the tagged arcs on the corresponding orbifold.

\end{itemize}
\end{cor}

Proofs of Theorems~\ref{14.6} and Corollaries~\ref{5.1},~\ref{5.2} repeat word-by-word the proofs of~\cite[Theorem~15.6]{FT},~\cite[Theorem~6.1]{FT} and~\cite[Corollary~6.2]{FT}.\\

\section{Growth rate of geometric cluster algebras}
\label{sec-growth}
In this section, we compute growth of cluster algebras arising from orbifolds. Combining with classification of mutation-finite cluster algebras~\cite{FST2}, this determines growth of all but small finite number of exceptional cluster algebras. The work will be completed in the upcoming paper.

A cluster algebra (or the corresponding exchange graph) has {\it polynomial growth} if the number of distinct seeds which can be
obtained from a fixed initial seed by at most $n$ mutations is bounded by a polynomial function of $n$.
A cluster algebra has {\it exponential growth} if the number of such seeds is bounded from below by an
exponentially growing function of $n$.

\begin{remark}
\label{rem_growth}
According to Corollary~\ref{5.2}, exchange graph of cluster algebra with s-decom\-posable exchange matrix depends on the exchange matrix only. Mutations of a matrix, in their turn, are completely determined by mutations of the diagram of the matrix. Thus, while investigating growth of a cluster algebra with exchange matrix $B$ in some seed, we can look instead at cluster algebra with any exchange matrix $B'$ such that $\D(B')=\D(B)$.
\end{remark}

According to remark~\ref{rem_growth}, to determine the growth rate of cluster algebra $\A$  with s-decomposable exchange matrix $B$, we
may assume without loss of generality that the weighted orbifold $\O^w=\O^w(B)$ contains orbifold points of weight 2 only.
For such an orbifold, we can build the associated surface $\S=\S(\O^w)$ as in Definition~\ref{ass surface}
(in other words, the associated orbifold $\hat \O^w$ is a surface in this case).
We will compare the exchange graph of tagged triangulations of $\O^w$ with the exchange graph of tagged triangulations of
$\S(\O^w)$, whose growth is known due to~\cite[Proposition~11.1]{FST}.


%
%
%

\begin{theorem}
\label{quasiisom}
Let $\O^w$ be a weighted orbifold with orbifold points of weight 2 only. Let $\S(\O^w)$ be the associated surface.
Then the exchange graph of tagged triangulations of $\O^w$ is quasi-isometric to the exchange graph of tagged triangulations
of $\S(\O^w)$.


\end{theorem}

To prove Theorem~\ref{quasiisom} it is sufficient to show the following lemma.

\begin{lemma}
\label{conj pairs}
Let $S$ be a bordered surface with marked points, and let $T$ be a triangulation of $S$. Let $\M$ be any proper subset of the set of punctures. Then there exists a positive integer $N=N(S,|M|)$ depending on the surface $S$ and the cardinality $|M|$ of the set $M$ only, such that $T$ can be turned via at most $N$ flips into a triangulation $T'$ having a conjugate pair in each vertex $x\in \M$.

\end{lemma}

Indeed, exchange graph of triangulations of $\O^w$ is isomorphic to exchange graph of triangulations of $\S$ with conjugate pairs in all special marked points, where conjugate pairs should be flipped simultaneously. Lemma~\ref{conj pairs} implies that the latter graph is quasi-isometric to the exchange graph of triangulations of $\S$, and the theorem follows.

We define the notion of {\it valence} in tagged triangulations consistent with one for untagged triangulations.

\begin{definition}
\label{val}
The {\it valence} of a marked point $x\in S$ in triangulation $T$ is the number of ends of arcs incident to $x$, with one exception. If $(\gamma, \gamma')$ is a conjugate pair in $x$, and $y$ is the other end of $\gamma$ (and $\gamma'$), then valence of $x$ is $1$, and the contribution of $\gamma$ and $\gamma'$ to the valence of $y$ is equal to $3$.

\end{definition}

To prove Lemma~\ref{conj pairs}, we will need the three technical lemmas. First, we show that we can decrease the valence under some assumptions.

\begin{lemma}
\label{valence}
Let $T$ be a triangulation containing an arc $e$ with ends $x$ and $y$, $x\ne y$.
If there is no conjugate pair in $x$ then there exists an arc $e'\ne e$ incident to $x$ such that the flip in $e'$ decreases the valence
of $x$.

\end{lemma}

\begin{proof}
Consider the arcs $e_1$ and $e_2$ incident to $x$ and neighboring to $e$ (say, $e_2$ follows $e$ and $e$ follows $e_1$
in the clockwise order,  see Fig.~\ref{fig valence}.a ($e_1$ may coincide with $e_2$, see~Fig.~\ref{fig valence}.b).
If $e_1$ coincides with $e_2$ then we have two triangles containing arcs
$e$ and $e_1=e_2$, as in Fig.~\ref{fig valence}.b, so, after a flip in $e_1=e_2$ the valence of $x$ decreases from 2 to 1,
and we obtain a conjugate pair in $x$.

Suppose that $e_1\ne e_2$. Consider  a flip $f_1$ in $e_1$. If both ends of $f_1(e_1)$ are distinct from $x$ then the flip $f_1$
decreases the valence of $x$ at least by 1. Suppose that one end of $f_1(e_1)$ coincides with $x$. Then the arc $e_3$ incident to
$x$ and following $e_1$ in a clockwise direction (say $e_3$, see Fig.~\ref{fig valence}.c) has both ends in $x$.
Therefore, if $e_1$ has two distinct ends, then flip in $e_3$ decreases the valence of $x$, otherwise the flip in $e_1$ takes one end of $e_1$ to $y$, so it decreases the valence of $x$ as well (as $x\ne y$).

\end{proof}

\begin{figure}[!h]
\begin{center}
\psfrag{x}{\scriptsize $x$}
\psfrag{y}{\scriptsize $y$}
\psfrag{e}{\scriptsize $e$}
\psfrag{e1}{\scriptsize $e_1$}
\psfrag{e12}{\scriptsize $e_1=e_2$}
\psfrag{e2}{\scriptsize $e_2$}
\psfrag{3e}{\scriptsize $e_3$}
\psfrag{a}{\small (a)}
\psfrag{b}{\small (b)}
\psfrag{c}{\small (c)}
\epsfig{file=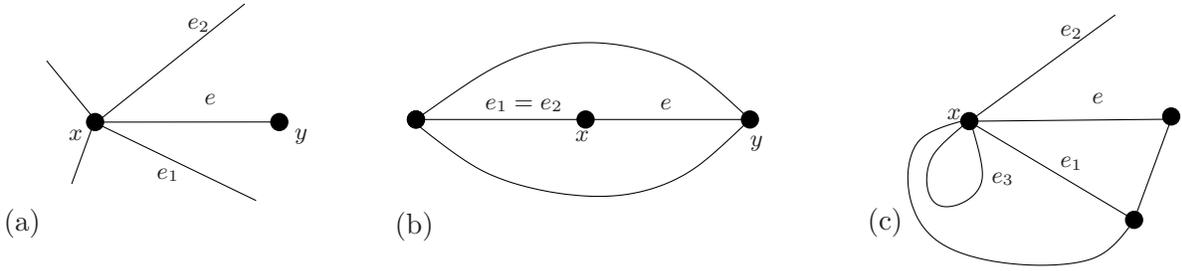,width=0.98\linewidth}
\caption{To the proof of Lemma~\ref{valence}}
\label{fig valence}
\end{center}
\end{figure}

We will say that a vertex $x\in \M$ is {\it filled} if there is a conjugate pair in $x$
and the arc incident to $x$ is not incident to other vertices $y\in \M$.
The second technical lemma concerns the structure on unfilled vertices.

\begin{lemma}
\label{chain}
Let $T$ be a triangulation with a vertex $x\in\M$ that is not filled. Then there exist vertices $z\in\M$ and $y\notin\M$ joined by an arc of $T$ such that $z$ is not filled.

\end{lemma}

\begin{proof}
Since $\M$ is a proper subset of the set of vertices of $T$, there exists a vertex $y_0\notin \M$. Consider the shortest path $\alpha=\{e_1,\dots,e_l\}$ of arcs of $T$ connecting $x$ to $y_0$ (i.e., $l$ is minimal possible).
Note that no vertex of the path $\alpha$ can be a filled vertex from $\M$ (the filled vertices are of valence $1$).
Since $x\in \M$ and $y_0\notin \M$,  there   exists an arc $e_i\in \alpha$  whose one endpoint belongs to $\M$ and the other does not.
Denote by $z$ and $y$ the ends of $e_i$ contained in $\M$ and the other respectively.
Then $z$ and $y$ satisfy the conditions of the lemma.

\end{proof}

The third lemma shows that every unfilled vertex $x\in\M$ can be filled in a uniformly bounded number of flips.

\begin{lemma}
\label{next filled}
 Let $\M=\{x_1,\dots,x_m\}$, and let $T$ be a triangulation where $x_1,\dots,x_k$ are filled and the other vertices are not filled.
Let $E$ be the number of arcs in $T$.
Then it is possible to find a vertex $x_j\in \{x_{k+1},\dots,x_m \}$
such that after at most $E$ flips of arcs incident to $x_j$ the vertex $x_j$ becomes filled.

\end{lemma}

\begin{proof}
By Lemma~\ref{chain}, there exists a vertex $x_j\in \{x_{k+1},\dots,x_m \}$ which is not filled and is joined with some vertex $y\notin \M$.
By Lemma~\ref{valence}, it is possible to make a flip in some arc incident to $x_j$ (and distinct from $x_jy$) such that valence of $x_j$
decreases. Applying   Lemma~\ref{valence} to $x_j$ several times, we decrease the valence of $x_j$ to 1 in less than $E$ flips.
So, in less than $E$ flips of arcs incident to $x_j$ we obtain a conjugate pair in $x_j$ with the other end of conjugate pair in $y\notin \M$, so  $x_j$ is filled.

\end{proof}

Now, we are able to prove Lemma~\ref{conj pairs}.

\begin{proof}[Proof of Lemma~\ref{conj pairs}]
The proof is by induction on the number of filled vertices. Let $\M=\{x_1,\dots,x_m\}$ ($m=|\M|$) and let $E$ be the number of arcs in
the triangulation $T$. Suppose that the vertices $x_1,\dots,x_k$
are filled. By Lemma~\ref{next filled}, we can make one of the vertices $x_{j}\in\{x_{k+1},\dots,x_m \}$ filled
via at most $E$ flips applied only to arcs incident to $x_j$.
Notice that by definition of filled vertex none of the vertices $x_i\in\{ x_1,\dots,x_k\}$ was joined with $x_j$.
This implies that while treating $x_{j}$ we preserve  the vertices $x_1,\dots,x_k$ filled, so that after the procedure we
have a triangulation with $k+1$ filled vertices.
Applying this procedure $m$ times we will get $m$ filled vertices in at most $N(S,m=|\M|)=Em$ flips.

\end{proof}

This completes the proof of Theorem~\ref{quasiisom}.

In view of Remark~\ref{rem_growth}, Theorem~\ref{quasiisom} together with~\cite[Proposition~11.1]{FST}
lead to the following result.

\begin{theorem}
\label{grows result}
Let $\A$ be a cluster algebra with an s-decomposable exchange matrix $B$. Then $\A$ has a polynomial growth if and
only if it corresponds to a diagram $\D(B)$ in the following list:
\begin{itemize}
\item finite type $A_n$, $B_n$, $C_n$ or $D_n$ (finite);
\item affine type $\widetilde A_n$, $\widetilde B_n$, $\widetilde C_n$ or $\widetilde D_n$ (linear growth);
\item diagram $\Gamma(n_1,n_2)$, $n_1,n_2\in \Z_{>0}$, shown in Fig.~\ref{growth-diagr} (quadratic growth);
\item diagram $\Delta(n_1,n_2)$, $n_1,n_2\in \Z_{>0}$, shown in Fig.~\ref{growth-diagr} (quadratic growth);
\item diagram $\Gamma(n_1,n_2,n_3)$, $n_1,n_2,n_3\in \Z_{>0}$, shown in Fig.~\ref{growth-diagr} (cubic growth).

\end{itemize}
Otherwise $\A$ has exponential growth.

\end{theorem}

\begin{figure}[!h]
\begin{center}
\begin{tabular}{rl}
 $\Gamma(n_1,n_2)$&
\psfrag{a1}{\scriptsize $a_1$}
\psfrag{a2}{\scriptsize $a_2$}
\psfrag{an1-}{\scriptsize $a_{n_1-1}$}
\psfrag{an1}{\scriptsize $a_{n_1}$}
\psfrag{b1}{\scriptsize $b_1$}
\psfrag{b2}{\scriptsize $b_2$}
\psfrag{bn2+}{\scriptsize $b_{n_2+1}$}
\psfrag{bn2}{\scriptsize $b_{n_2}$}
\psfrag{b0}{\scriptsize $b_0$}
\psfrag{b0'}{\scriptsize $b_0'$}
\psfrag{c1}{\scriptsize $c_1$}
\psfrag{cn3-}{\scriptsize $c_{n_3-1}$}
\psfrag{cn3}{\scriptsize $c_{n_3}$}
\psfrag{dots}{\scriptsize $\dots$}
\raisebox{-18pt}{
\epsfig{file=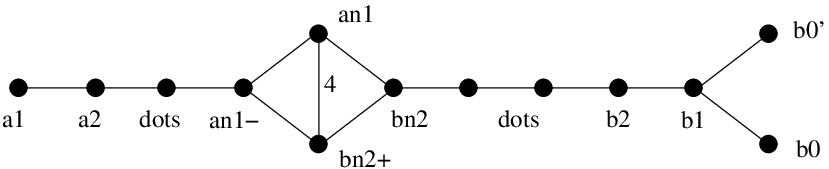,width=0.716\linewidth}
}
\\
 $\Gamma(n_1,n_2,n_3)$&
\psfrag{a1}{\scriptsize $a_1$}
\psfrag{a2}{\scriptsize $a_2$}
\psfrag{an1-}{\scriptsize $a_{n_1-1}$}
\psfrag{an1}{\scriptsize $a_{n_1}$}
\psfrag{b1}{\scriptsize $b_1$}
\psfrag{b2}{\scriptsize $b_2$}
\psfrag{bn2+}{\scriptsize $b_{n_2+1}$}
\psfrag{bn2++}{\scriptsize $b_{n_2+2}$}
\psfrag{bn2}{\scriptsize $b_{n_2}$}
\psfrag{b0}{\scriptsize $b_0$}
\psfrag{b0'}{\scriptsize $b_0'$}
\psfrag{c1}{\scriptsize $c_1$}
\psfrag{cn3-}{\scriptsize $c_{n_3-1}$}
\psfrag{cn3}{\scriptsize $c_{n_3}$}
\psfrag{dots}{\scriptsize $\dots$}
\raisebox{-18pt}{
\epsfig{file=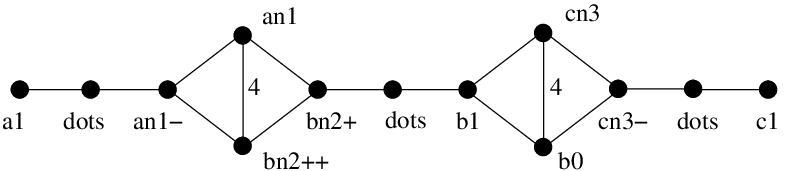,width=0.68\linewidth}
}\\
 $\Delta(n_1,n_2)$&
\psfrag{a1}{\scriptsize $a_1$}
\psfrag{a2}{\scriptsize $a_2$}
\psfrag{an1-}{\scriptsize $a_{n_1-1}$}
\psfrag{an1}{\scriptsize $a_{n_1}$}
\psfrag{b1}{\scriptsize $b_1$}
\psfrag{b2}{\scriptsize $b_2$}
\psfrag{bn2+}{\scriptsize $b_{n_2+1}$}
\psfrag{bn2}{\scriptsize $b_{n_2}$}
\psfrag{b0}{\scriptsize $b_0$}
\psfrag{b0'}{\scriptsize $b_0'$}
\psfrag{c1}{\scriptsize $c_1$}
\psfrag{cn3-}{\scriptsize $c_{n_3-1}$}
\psfrag{cn3}{\scriptsize $c_{n_3}$}
\psfrag{dots}{\scriptsize $\dots$}
\psfrag{2}{$2$}
\raisebox{-18pt}{
\epsfig{file=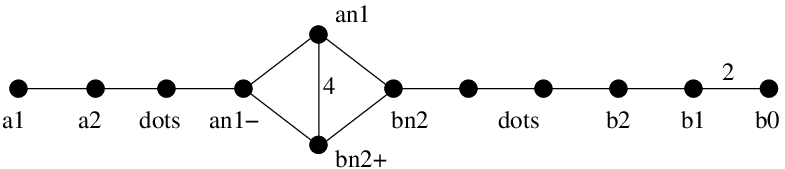,width=0.68\linewidth}
}\end{tabular}
\caption{Diagrams for the cluster algebras of quadratic and cubic growth. All triangles are oriented. Orientations of the remaining
edges are of no importance. Diagrams $\Gamma(n_1,n_2)$ and $\Gamma(n_1,n_2,n_3)$ are obtained in~\cite{FST} for skew-symmetric case.
}
\label{growth-diagr}
\end{center}
\end{figure}

\section{Unfoldings of matrices and diagrams}
\label{unfolding-s}

In this section, we recall basic definitions of unfoldings of matrices and diagrams defined in~\cite{FST2}, and reformulate some constructions of~\cite{FST2} in terms of orbifolds.

\subsection{Definitions}
\label{unf-def}
Let $B$ be an indecomposable $n\times n$ skew-symmetrizable integer matrix, and let $BD$ be a skew-symmetric matrix, where $D=(d_{i})$ is a diagonal integer matrix with positive diagonal entries. Notice that for any matrix $\mu_i(B)$ the matrix $\mu_i(B)D$ will be skew-symmetric.

We use the following definition of unfolding of a skew-symmetrizable matrix (communicated to us by A.~Zelevinsky).

Suppose that we have chosen disjoint index sets $E_1,\dots, E_n$ with $|E_i| =d_i$. Denote $m=\sum\limits_{i=1}^n d_i$.
Suppose also that we choose a skew-symmetric integer matrix $C$ of size $m\times m$ with rows and columns indexed by the union of all $E_i$, such that

(1) the sum of entries in each column of each $E_i \times E_j$ block of $C$ equals $b_{ij}$;

(2) if $b_{ij} \geq 0$ then the $E_i \times E_j$ block of $C$ has all entries non-negative.

Define a {\it composite mutation} $\h\mu_i = \prod_{\hat\imath \in E_i} \mu_{\hat\imath}$ on $C$. This mutation is well-defined, since all the mutations  $\mu_{\hat\imath}$, $\hat\imath\in E_i$, for given $i$ commute.

\begin{definition}
Skew-symmetric matrix $C$ is an {\it unfolding} of skew-symmetrizable matrix $B$ if $C$ satisfies assertions $(1)$ and $(2)$ above, and for any sequence of iterated mutations $\mu_{k_1}\dots\mu_{k_m}(B)$ the matrix $C'=\h\mu_{k_1}\dots\h\mu_{k_m}(C)$ satisfies assertions $(1)$ and $(2)$ with respect to $B'=\mu_{k_1}\dots\mu_{k_m}(B)$.
\end{definition}

If $C$ is an unfolding of a skew-symmetrizable integer matrix $B$, it is natural to define an {\it unfolding of a diagram} of $B$ as a diagram of $C$. In general, we say that a diagram $\h\D$ is an unfolding of a diagram $\D$ if there exist matrices $B$ and $C$ with diagrams $\D$ and $\h\D$ respectively, and $C$ is an unfolding of $B$. This definition is equivalent to the following one.

\begin{definition}
Let $\D$ be a diagram with vertices $x_1,\dots,x_n$, and let $d_1,\dots,d_n$ be positive integers. Let $\h\D$ be a connected skew-symmetric diagram with vertices $x_{\hat\imath}$ indexed by sets $E_i$ of order $d_i$ (i.e., ${\hat\imath}\in E_i$, $E_i=\{\hat\imath_1,\dots,\hat\imath_{d_i}\}$), such that for each $i,j\in\[1\dots n\]$ the following holds:

(A) there are no edges joining vertices inside $E_i$ and no edges joining vertices inside $E_j$;

(B) for all $\hat\imath\in E_i$ the sum of weights of all edges joining $x_{\hat\imath}$ with $E_j$ is the same, and all the arrows are oriented simultaneously either from $E_i$ to $E_j$ or from $E_j$ to $E_i$;

(C) for all $\hat\imath\in E_i$ and $\hat\jmath\in E_j$ the product of total weight of edges joining $x_{\hat\imath}$ with $E_j$ and total weight of edges joining $x_{\hat\jmath}$ with $E_i$ equals the weight of $x_ix_j$.

Define a {\it composite mutation} $\h\mu_i = \prod_{\hat\imath \in E_i} \mu_{\hat\imath}$ on $\h S$. As in the case of matrices, the mutation is well-defined. We say that $\h\D$ is an {\it unfolding} of $\D$ if for any sequence of iterated mutations $\mu_{i_1}\dots\mu_{i_k}$ a pair of diagrams $(\mu_{i_1}\dots\mu_{i_k}\! \D,\,\h\mu_{i_1}\dots\h\mu_{i_k}\!\h\D\,)$ satisfies the same conditions as the pair $(\D,\h\D)$ does, i.e. for each $i,j\le n$ the assumptions (A), (B) and (C) hold.

\end{definition}

One can note that unfolding of a diagram may not be unique: if the diagram corresponds to more than one matrix, then these matrices may have distinct unfoldings with distinct diagrams.

\subsection{Local unfoldings}
\label{unf-local}

In~\cite{FST2}, we build an unfolding for every s-decomposable diagram $\D$ in the following way: every s-block of $\D$ (see column~2 of Table~\ref{table-loc}) is substituted by a skew-symmetric block located in the same row of Table~\ref{table-loc}. The procedure is well-defined since there is a one-to-one correspondence between s-blocks and their unfoldings. According to~\cite[Section~6]{FST2}, the procedure above results in an unfolding of $\D$.

\begin{definition}
The unfolding constructed above is called {\it local} unfolding.

\end{definition}

In terms of s-decomposable matrices, this procedure provides an unfolding of a skew-symmetrizable matrix $B$ with diagram $\D$ such that weights of all the outlets of weighted diagram $\D^w$ are equal to one.

Local unfolding can be reformulated in terms of orbifolds as well. Consider skew-symmetrizable matrix $B$ as above, and construct a weighted orbifold $\O^w$. Assumptions on $B$ imply that all orbifold points of $\O^w$ have weight $2$. Substituting every orbifold point by a puncture, and every arc by a conjugate pair, we obtain the associated surface (see Definition~\ref{ass surface}), and then take the signed adjacency matrix of the obtained triangulation. In other words, we substitute all elementary orbifolds (Section~\ref{orbifolds-s}) by elementary surfaces depicted in the same row of Table~\ref{table-loc} (cf. Remark~\ref{ex gr ass surf}).

\begin{table}[!h]
\begin{center}
\caption{Triangulations corresponding to local unfoldings of non-exceptional s-blocks. Weights $d_i$ for all the outlets are equal to one. The vertices with indices belonging to the same index sets $E_i$ in the unfolding diagram are denoted by the same letter (i.e. $(v_1,v_2)$, $(w_1,w_2)$, $(p_1,p_2)$), cf.~Table~\ref{newblocks} }
\label{table-loc}
\begin{tabular}{cp{0.7cm}cp{0.7cm}c}
Diagram&&Unfolding&&Triangulation\\
\hline
&&&&\\
\psfrag{u}{\tiny $u$}
\psfrag{v}{\tiny $v$}
\psfrag{2-}{\tiny $2$}
\raisebox{8mm}{\epsfig{file=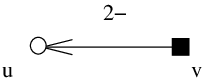,width=0.105\linewidth}}&&
\psfrag{u}{\tiny $u$}
\psfrag{v1}{\tiny $v_1$}
\psfrag{v2}{\tiny $v_2$}
\raisebox{4mm}{\epsfig{file=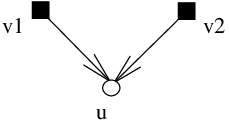,width=0.12\linewidth}}&&
\psfrag{u}{\tiny $u$}
\psfrag{v1}{\tiny $v_1$}
\psfrag{v2}{\tiny $v_2$}
\raisebox{0mm}{\epsfig{file=diagrams_pic/s-block-I.eps,width=0.12\linewidth}}\\
\psfrag{u}{\tiny $u$}
\psfrag{v}{\tiny $v$}
\psfrag{2-}{\tiny $2$}
\raisebox{8mm}{\epsfig{file=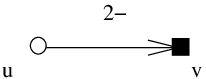,width=0.105\linewidth}}&&
\psfrag{u}{\tiny $u$}
\psfrag{v1}{\tiny $v_1$}
\psfrag{v2}{\tiny $v_2$}
\raisebox{4mm}{\epsfig{file=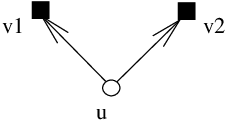,width=0.12\linewidth}}&&
\psfrag{u}{\tiny $u$}
\psfrag{v1}{\tiny $v_1$}
\psfrag{v2}{\tiny $v_2$}
\raisebox{0mm}{\epsfig{file=diagrams_pic/s-block-II.eps,width=0.12\linewidth}}\\
\psfrag{u}{\tiny $u$}
\psfrag{v}{\tiny $v$}
\psfrag{w}{\tiny $w$}
\psfrag{2-}{\tiny $2$}
\raisebox{6mm}{\epsfig{file=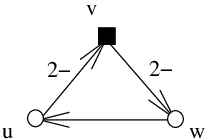,width=0.105\linewidth}}&&
\psfrag{u}{\tiny $u$}
\psfrag{w}{\tiny $w$}
\psfrag{v1}{\tiny $v_1$}
\psfrag{v2}{\tiny $v_2$}
\raisebox{0mm}{\epsfig{file=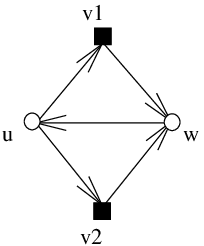,width=0.12\linewidth}}&&
\psfrag{u}{\tiny $u$}
\psfrag{w}{\tiny $w$}
\psfrag{v1}{\tiny $v_1$}
\psfrag{v2}{\tiny $v_2$}
\raisebox{0mm}{\epsfig{file=diagrams_pic/s-block-III.eps,width=0.09\linewidth}}\\
&&&&\\
\psfrag{u}{\tiny $u$}
\psfrag{w}{\tiny $w$}
\psfrag{p}{\tiny $p$}
\psfrag{q}{\tiny $q$}
\psfrag{2-}{\tiny $2$}
\raisebox{0mm}{\epsfig{file=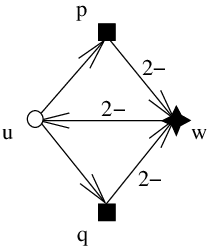,width=0.12\linewidth}}&&
\psfrag{u}{\tiny $u$}
\psfrag{r}{\tiny $q$}
\psfrag{w}{\tiny $w_2$}
\psfrag{p}{\tiny $p$}
\psfrag{q}{\tiny $w_1$}
\raisebox{4mm}{\epsfig{file=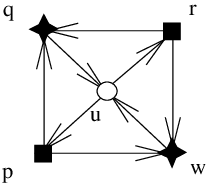,width=0.12\linewidth}}&&
\psfrag{u}{\tiny $u$}
\psfrag{w1}{\tiny $p$}
\psfrag{w2}{\tiny $q$}
\psfrag{p1}{\tiny $w_1$}
\psfrag{p2}{\tiny $w_2$}
\raisebox{0mm}{\epsfig{file=diagrams_pic/s-block-IV.eps,width=0.15\linewidth}}\\
&&&&\\
\psfrag{u}{\tiny $u$}
\psfrag{w}{\tiny $w$}
\psfrag{p}{\tiny $p$}
\psfrag{q}{\tiny $q$}
\psfrag{2-}{\tiny $2$}
\raisebox{0mm}{\epsfig{file=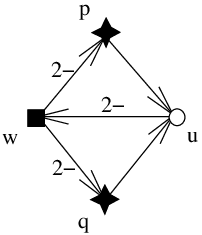,width=0.12\linewidth}}&&
\psfrag{u}{\tiny $u$}
\psfrag{r}{\tiny $w_2$}
\psfrag{w}{\tiny $q$}
\psfrag{p}{\tiny $w_1$}
\psfrag{q}{\tiny $p$}
\raisebox{4mm}{\epsfig{file=diagrams_pic/block5l.eps,width=0.12\linewidth}}&&
\psfrag{u}{\tiny $u$}
\psfrag{w1}{\tiny $w_1$}
\psfrag{w2}{\tiny $w_2$}
\psfrag{p1}{\tiny $p$}
\psfrag{p2}{\tiny $q$}
\raisebox{0mm}{\epsfig{file=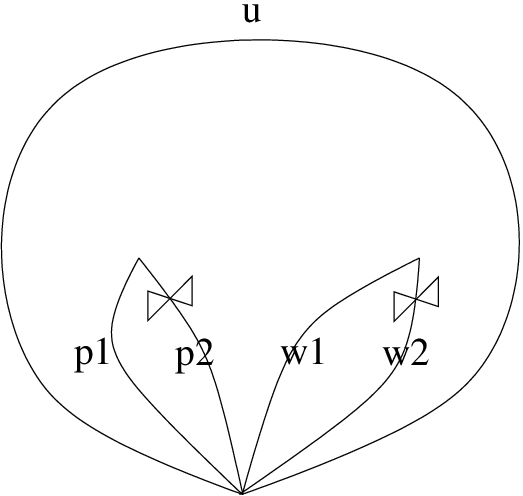,width=0.15\linewidth}}\\
&&&&\\
\psfrag{u}{\tiny $u$}
\psfrag{w}{\tiny $p$}
\psfrag{p}{\tiny $w$}
\psfrag{2-}{\tiny $2$}\psfrag{4}{\tiny $4$}
\raisebox{5mm}{\epsfig{file=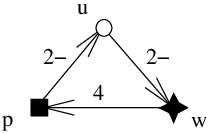,width=0.12\linewidth}}&&
\psfrag{u}{\tiny $u$}
\psfrag{r}{\tiny $w_2$}
\psfrag{w}{\tiny $p_2$}
\psfrag{p}{\tiny $w_1$}
\psfrag{q}{\tiny $p_1$}
\raisebox{3mm}{\epsfig{file=diagrams_pic/block5l.eps,width=0.12\linewidth}}&&
\psfrag{u}{\tiny $u$}
\psfrag{w1}{\tiny $w_1$}
\psfrag{w2}{\tiny $w_2$}
\psfrag{p1}{\tiny $p_1$}
\psfrag{p2}{\tiny $p_2$}
\raisebox{0mm}{\epsfig{file=diagrams_pic/s-block-VI.eps,width=0.15\linewidth}}\\
&&&&\\

\psfrag{u}{\tiny $u$}
\psfrag{r}{\tiny $r$}
\psfrag{w}{\tiny $w$}
\psfrag{p}{\tiny $p$}
\psfrag{q}{\tiny $q$}
\psfrag{2-}{\tiny $2$}
\psfrag{u1}{\tiny $u_1$}
\psfrag{u2}{\tiny $u_2$}
\psfrag{r}{\tiny $r$}
\psfrag{w}{\tiny $w$}
\psfrag{p}{\tiny $p$}
\psfrag{q}{\tiny $q$}
\psfrag{w1}{\tiny $p$}
\psfrag{w2}{\tiny $r$}
\psfrag{u}{\tiny $u_1$}
\psfrag{v}{\tiny $u_2$}
\psfrag{p}{\tiny $w$}
\psfrag{q}{\tiny $q$}

\end{tabular}
\end{center}
\end{table}

\section{Unfoldings of mutation-finite matrices}
\label{unfolding-fin}

In this section, we provide unfoldings for s-decomposable skew-symmetrizable matrices (with exception of one family, see Remark~\ref{unf-ex}). Unfoldings of mutation-finite non-decomposable matrices are constructed in~\cite{FST2}. In other words, we associate with (almost) every triangulated weighted orbifold a triangulated surface, such that flips on the orbifold agree with composite flips on the surface.

Matrix $B$ not admitting local unfolding is characterized by the following property: weights of all outlets in its weighted diagram $\D^w$ are equal to $2$. Further considerations split in two cases depending on the weights of orbifold points of the orbifold $\O^w$: either all the orbifold points are of weight $1/2$ (in this case matrix admits {\it prime unfolding} constructed in Section~\ref{sec-prime}, see also Remark~\ref{unf-ex}), or there are orbifold points of both weights $2$ and $1/2$ (Section~\ref{sec-composite}).

\begin{remark}
\label{unf-ex}
The only series of s-decomposable matrices for which we are not able to construct an unfolding can be described as follows: the corresponding weighted orbifold is of genus zero without boundary, and has exactly one orbifold point of weight $1/2$. The corresponding diagrams are shown in Fig~\ref{fig-ex}. We will exclude these matrices from all considerations in Section~\ref{sec-prime}, see Section~\ref{one}.

\end{remark}

\begin{figure}[!h]
\begin{center}
\psfrag{1}{\scriptsize $c_1$}
\psfrag{2}{\scriptsize $c_2$}
\psfrag{3}{\scriptsize $c_3$}
\psfrag{4}{\scriptsize $c_4$}
\psfrag{5}{\scriptsize $c_5$}
\psfrag{6}{\scriptsize $c_6$}
\psfrag{7}{\scriptsize $c_7$}
\psfrag{8}{\scriptsize $c_8$}
\psfrag{3n3}{\scriptsize $c_{3n+3}$}
\psfrag{3n4}{\scriptsize $c_{3n+4}$}
\psfrag{3n5}{\scriptsize $c_{3n+5}$}
\psfrag{dd}{\scriptsize $\dots$}
\psfrag{2_}{\scriptsize $2$}
\epsfig{file=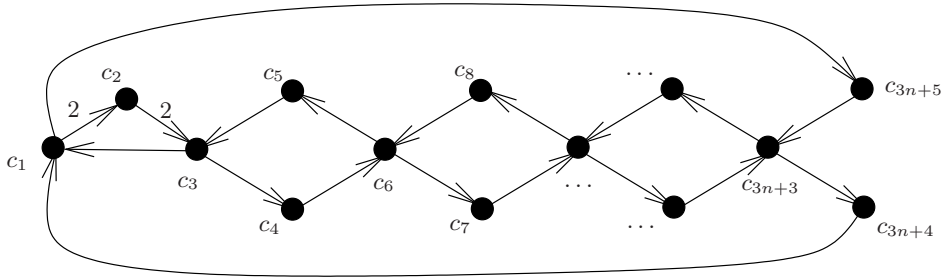,width=0.78\linewidth}
\caption{Series of diagrams corresponding to closed sphere with one orbifold point and $n+3$ marked points, $n\ge 0$.}
\label{fig-ex}
\end{center}
\end{figure}

For the both cases described above we need the following construction.

\begin{definition}
A {\it partial unfolding} of a skew-symmetrizable matrix $B$ is a skew-symmetrizable matrix $C$ satisfying all properties of unfolding except skew-symmetry. A {\it partial unfolding} of a diagram is defined accordingly.

\end{definition}

The following lemma follows immediately from the definition of unfolding.

\begin{lemma}
\label{composition}
A composition of two partial unfoldings is a partial unfolding.

\end{lemma}

\begin{remark}[{{\it Local partial unfolding}}]
\label{partial local}
In the construction of local unfolding we treat all pending arcs simultaneously. We also can do the same construction for any partial
subset of the set of all pending arcs: we can unfold only one (or several) s-blocks, while the other s-blocks remain unchanged.
Then we obtain a partial unfolding of a diagram (or a matrix), which corresponds to an orbifold with fewer orbifold points than the
initial orbifold has.
In this case we say that we perform a {\it local partial unfolding} in the given set of blocks (or in the given set of corresponding
orbifold points of weight $2$).

\end{remark}

\subsection{Prime unfoldings}
\label{sec-prime}

In this section we build a matrix unfolding for the matrices with weighted orbifolds having all orbifold points of weight $1/2$.
Let us fix s-decomposable matrix $B$, its weighted diagram $\D^w$, and weighted orbifold $\O^w$.

Given weighted orbifold $\O^w$, we will choose a triangulation $T$ of $\O^w$ of special type and build a ramified Galois covering of $\O^w$ by a surface $\hat S$ with branching points exactly in all orbifold points of $\O^w$. In case of even number of orbifold points the covering will be of degree 2, otherwise it will be of degree 4. Then we show that the composite flips of triangulation $\hat T$ of $\hat S$ (obtained as a lift of $T$)
agree with the Galois group of the covering, which allows to derive (Theorem~\ref{cover-unf}) that the properties of unfolding are satisfied.
The unfolding will be constructed for signed adjacency matrix of $T$, which is mutation-equivalent to $B$. By definition of unfolding, this is equivalent to a construction of an unfolding for $B$.

\begin{remark}
\label{ex}
The structure of a prime unfolding differs from the structure of a local unfolding. In case of local unfolding there is a natural bijection between elements of the mutation class of matrix $B$ and elements of the mutation class of its unfolding $\hat B$. This does not hold for prime unfolding: here we have a multivalued map from the mutation class of $B$ to the mutation class of $\hat B$. In other words, the graph of mutations of $\hat B$ covers the graph of mutations of $B$.

\end{remark}

\subsubsection*{Construction of Galois covering}
Denote by $z_1,\dots,z_m$ the orbifold points of $\O^w$.
The construction of the covering depends on the parity of the number of orbifold points on $\O^w$. First, we assume that $m$ is even, i.e. $m=2k$.

By Lemma~\ref{bubbles-even}, there exists a triangulation $T$ of $\O^w$ such that each triangle containing a pending arc
actually contains two pending arcs (all small orbifold pieces are disks with two orbifold points, and all s-blocks in the
corresponding diagram $\D^w$ are blocks of type $\widetilde V_{12}$).

Now, consider the monogons $\Delta_1,\dots,\Delta_k \in T$ containing two pending arcs each.
For each of these monogons we join the two orbifold points by a (non-self-intersecting) segment $s_i$ contained in $\Delta_i$.
We cut $\O^w$ along all $s_i$'s, take two copies $\O^w_1$ and $\O^w_2$ of $\O^w$, and glue the left component of $s_i\in \O^w_1$ to the right component of $s_i\in  \O^w_2$, and converse. As a result, we obtain a surface $\hat S$ without orbifold points,  which is a ramified covering of $\O^w$ (with branching in all $2k$ orbifold points) of degree $2$.

The surface $\hat S$ is endowed with a triangulation $\hat T$ covering the triangulation $T$. We will show (Theorem~\ref{cover-unf}) that signed adjacency matrix of $\hat T$ provides an unfolding of signed adjacency matrix of $T$.

\begin{remark}
\label{g1}
The covering $\pi: \hat S \to \O$ is a quotient by an involution on $\hat S$ acting as a central
symmetry with respect to the preimage of any orbifold point.

\end{remark}

Now assume that the number of orbifold points is odd, i.e. $m=2k+1$, $k>0$ (see~Section~\ref{one} for the case $k=0$).

This time we choose a triangulation $T$ of $\O^w$ of the type described in Lemma~\ref{bubbles-odd},
namely, $T$ satisfies the following conditions:
\begin{itemize}
\item $T$ contains $k$ monogons $\Delta_1,\dots,\Delta_k$
with two pending arcs each and one digon $\Delta_0$ with one pending arc (all other triangles have no pending arcs);

\item $\Delta_0$ has a common side with $\Delta_1$.
\end{itemize}

Denote by $z_i,z_{i+1}$ the orbifold points contained in $\Delta_i$, $i>0$, and by $z_0$ the orbifold point contained in $\Delta_0$.

We construct the covering in two steps.

For each of the monogons $\Delta_1,\dots,\Delta_k$  we join the orbifold points $z_i$ and  $z_{i+1}$  by a (non-self-intersecting)
segment $s_i$ contained in $\Delta_i$. Then, similarly to the case of even number of orbifold points we take two copies of
$\O^w$, cut it along all $s_i$'s and construct a degree two ramified covering $\t\O$
with branching points in all $2k$ orbifold points.
Note that $\t\O$ is not a surface: it still has two orbifold points $x_0$ and $x_0'$
(which are projected to the same orbifold point $z_0$ of $\O^w$).
Now we can apply the algorithm for even number of orbifold points.

\subsubsection*{Verification of unfolding properties}
As before, $\O^w$ is an orbifold with orbifold points $z_1,\dots,z_m$ of weight $1/2$.
Denote by $\pi:\hat S \to \O$ the ramified Galois covering constructed above, and let $d$ be its degree (as we have seen, $d=2$ or $4$).
Let $T$ be any triangulation of $\O^w$ and denote by $\hat T$ its lift on $\hat S$.
For each flip $f_i$ in an arc $e_i\in T$ we define a composite flip $\hat f_i$ of $\hat T$
as a composition of flips $f_{ik}$ in all arcs $\hat e_{ik}$ of $\hat T$ projecting to $e_i$.

\begin{lemma}
\label{correctness}
Composite flip is well defined, i.e. for any two arcs $\hat e_{ik}, \hat e_{il}\in \hat T$, such that
$\pi(\hat e_{ik})=\pi(\hat e_{il})=e_i$ one has $f_{ik}f_{il}=f_{il}f_{ik}$.
Furthermore, $\hat f_i(\hat T)=\widehat{(f_i(T))}$.

\end{lemma}

\begin{proof}
We prove the lemma separately for $e_i$ being a pending arc or not.

Suppose that $e_i$ is not a pending arc. Then arcs $\hat e_{i1},\dots,\hat e_{id}\in \hat T$ project to $e_i$ (recall, $d$ is the degree of the covering). Note that for $j\ne k$ two arcs $\hat e_{ij}$ and $\hat e_{ik}$ are not sides of the same triangle
(this follows from the property of the covering that there are no branching points other than orbifold points of $\O^w$).
Thus, the corresponding flips commute.

Now, consider two triangles attached along the common edge $e_i$, these two triangles compose a quadrilateral $q$ with diagonal $e_i$. The quadrilateral $q$ is covered by $d$ copies of $q$. The composite flip $\hat f_i$ on $\hat T$ consists of
$d$ commuting flips (of the diagonals of each of these $d$ quadrilaterals), which implies that the triangulation $\hat f_i(\hat T)$
covers the triangulation $f_i(T)$.

Suppose now that $e_i$ is a pending arc. Consider $e_i$ as a ``round trip'' from a marked point to itself via an orbifold point. Then there are exactly $d/2$ copies of this arc in $\hat T$, no pair of them are sides of one triangle (again, since there is no ramification in marked points).
Every lift $\hat e_{ij}$ of $e_i$ is a diagonal of quadrilateral which consists of two copies of the digon containing $e_i$ (the consideration for monogons is similar). All these $d/2$ quadrilaterals project to the digon containing $e_i$ (see Fig.~\ref{commute}), so we see that the lemma holds for pending arcs as well.

\end{proof}

\begin{figure}[!h]
\begin{center}
\psfrag{e}{\scriptsize $e_i$}
\psfrag{p}{\scriptsize $\pi$}
\psfrag{ej}{\scriptsize $\hat e_{ij}$}
\psfrag{1}{\raisebox{-2mm}{\scriptsize $1$}}
\psfrag{2}{\scriptsize $2$}
\psfrag{1'}{\raisebox{-2mm}{\scriptsize $1'$}}
\psfrag{2'}{\raisebox{-0.2mm}{\scriptsize $2'$}}
\psfrag{f1}{\scriptsize $\hat f_i$}
\psfrag{f2}{\scriptsize $f_i$}
\epsfig{file=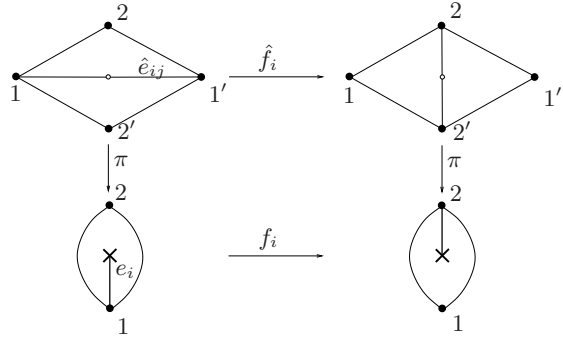,width=0.45\linewidth}
\caption{Composite flips and projection (one of $d/2$ congruent quadrilaterals
of $\hat T$ projecting to the digon in $T$)}
\label{commute}
\end{center}
\end{figure}

Let  $B$ and $C$  be the signed adjacency matrices of the triangulations $T$ and $\hat T$ respectively.
In the following theorem we prove that $C$ is an unfolding of $B$.

\begin{theorem}
\label{cover-unf}
Let $\O^w$ be an orbifold with all orbifold points of weight $1/2$. Let $\pi: \hat S \to \O^w$ be a Galois covering
whose set of branching points coincides with the set of orbifold points of $\O^w$, and suppose that all ramifications are simple.
Let $T$ be a tagged triangulation of $\O^w$  and let $\hat T $ be a tagged triangulation of $\hat S$ covering the triangulation $T$.
Let  $B$ and $C$  be the signed adjacency matrices of the triangulations $T$ and $\hat T$ respectively.

Then $C$ is an unfolding of $B$.

\end{theorem}

\begin{proof}
Denote by $|T|$ the number of arcs of the triangulation $T$, and denote by $d$ the degree of $\pi$.
Define the natural index sets $E_1,\dots,E_{|T|}$ for $C$:  $E_i$ is the set of indices of all arcs of $\hat T$ projecting to the given arc of $T$.
Then $|E_i|=d$ for non-pending arcs $e_i\in T$ and $|E_i|=d/2$ for pending arcs $e_i$.
This exactly corresponds to the fact that all orbifold points are of weight $1/2$.

To show that $C$ is an unfolding of $B$ we need to show that for any sequence of iterated mutations
$\mu_{k_1}\dots\mu_{k_m}$ the matrices $B'=\mu_{k_1}\dots\mu_{k_m}B$ and $C'=\hat \mu_{k_1}\dots\hat \mu_{k_m}(C)$
satisfy assertions (1) and (2) of the definition of matrix unfolding.

Note that the notion of composite mutation in the definition of the unfolding coincides at the first step with the notion
of composite flip for the covering (i.e. the same indices are involved).
It follows from Lemma~\ref{correctness} that the same property holds after one composite flip
(i.e. the elements of the index set $E_i$ after one mutation still correspond to the set of arcs projected to the same arc $f_j(e_i)$ of
$f_j(T)$). This implies that the notions coincide after any number of flips and that it is sufficient to check the assertions
(1) and (2) with respect to the matrices $B$ and $C$ only.

For matrices $B$ and $C$ the properties (1) and (2) follow from the fact that $\pi$ has no ramification in the vertices of the triangulation $\hat T$. More precisely, for each angle $\alpha$  in $T$ (it may be formed by one or two arcs)
there are exactly $d$ copies of this angle in $\hat T$, and the Galois group of the covering  acts on these copies transitively by orientation preserving transformations.
Therefore, the contributions of these angles to the matrix $C$ are of the same sign, which show both (1) and (2).

\end{proof}

\subsection{General case}
\label{sec-composite}

Suppose now that weighted orbifold contains orbifold points of both weights $2$ and $1/2$. To get the unfolding in this case we do local partial unfolding in all the orbifold points of weight $2$, and then construct prime unfolding of obtained orbifold via Galois covering with branching points in all the remaining orbifold points (they all are of weight $1/2$).

By Lemma~\ref{composition}, the composition of local partial unfolding and prime unfolding is an unfolding. Therefore, as a result of the procedure
we obtain an unfolding of signed adjacency matrix of any triangulation of the initial unfolding. The construction works for each orbifold with at least 2 orbifold points.

\subsection{Orbifolds with a unique orbifold point}
\label{one}
The construction of Section~\ref{sec-composite} works to build an unfolding for each matrix corresponding to an orbifold with at least
two orbifold points. In this section we investigate the case of one orbifold point. We can assume that the weight of the orbifold point is $1/2$, otherwise we construct local unfolding.

First, we give a construction for the case of orbifolds with boundary. Next, we build the unfolding for the case of closed orbifold
topologically different from a sphere. The case of closed sphere with exactly one orbifold point of weight $1/2$ remains open.

\subsubsection*{Orbifolds with boundary}
Let $\O^w$ be a weighted orbifold with one orbifold point $z$ and at least one boundary component. Let $x\in\partial \O^w$ be a marked point.
Take a triangulation $T$ of $\O^w$ such that the pending arc of $T$ is incident to $x$, and the triangle containing the pending arc contains also an edge lying on $\partial\O^w$.

Then we can do the following: cut $\O^w$ along the pending arc, take two copies of obtained surface,
and attach them so that we obtain a 2-sheet covering $\hat S $ branching exactly in $z$ (see Fig~\ref{boundary}).
If we denote by $\hat T$ the triangulation obtained on $\hat S$, then Theorem~\ref{cover-unf} implies that the signed adjacency matrix of $\hat T$ is an unfolding of the signed adjacency matrix of $T$.

\begin{figure}[!h]
\begin{center}
\psfrag{2}{\scriptsize $2$}
\psfrag{R}{\scriptsize $R$}
\psfrag{R'}{\scriptsize $R'$}
\psfrag{R''}{\scriptsize $R''$}
\epsfig{file=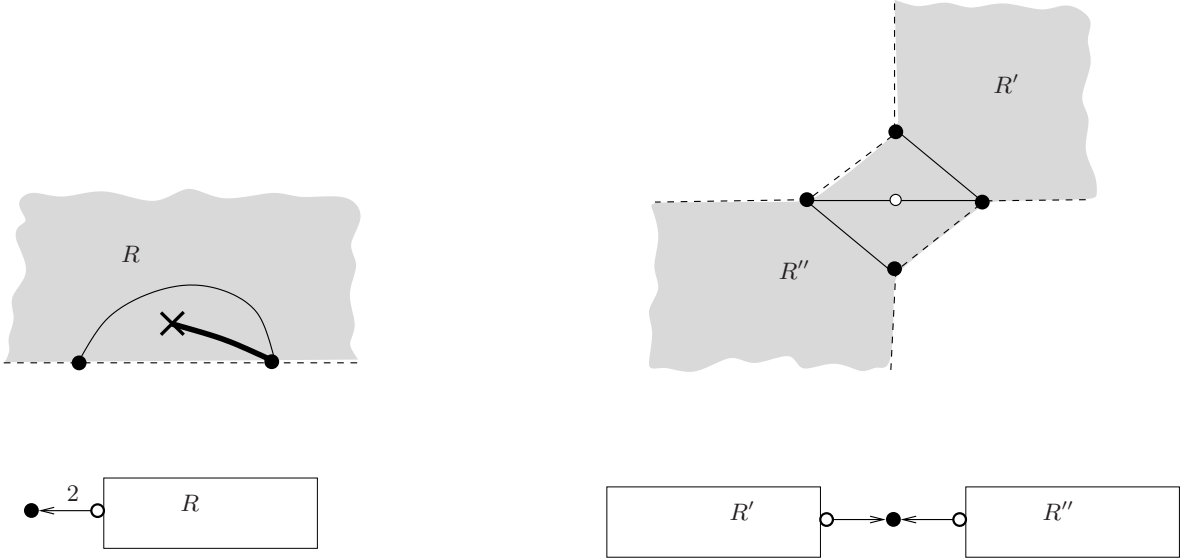,width=0.98\linewidth}
\caption{Construction of unfoldings for orbifolds with boundary.}
\label{boundary}
\end{center}
\end{figure}

\subsubsection*{Closed orbifolds of positive genus}
Let $\O^w$ be a closed weighted orbifold of genus $g>0$ with a unique orbifold point $z$ of weight $1/2$.
Then $\O^w$ contains a closed curve $\gamma$ which does not cut $\O^w$ into two connected components.
We cut $\O^w$ along $\gamma$, take two copies of the obtained surface and attach the sheets so that we obtain a non-ramified 2-sheet covering $\widetilde \O^w$ of $\O^w$. Clearly, $\widetilde \O^w$ has two orbifold points, so we can construct prime unfolding. Combining it with partial unfolding provided by non-ramified covering, we obtain an unfolding of signed adjacency matrix of a triangulation of $\O^w$.

The discussion above leads to the following theorem.

\begin{theorem}
\label{thm-unfolding}
Let $B$ be an s-decomposable skew-symmetrizable matrix.
Let $\O^w$ be a corresponding weighted orbifold. If $\O^w$ is not a closed sphere with a unique orbifold point of weight $1/2$, then $B$ admits an unfolding to a signed adjacency matrix of a triangulated surface.

\end{theorem}

\begin{remark}
\label{unf-all}
In~\cite{FST2} unfoldings of all non-decomposable mutation-finite matrices were constructed. In view of Theorem~\ref{thm-unfolding}, signed adjacency matrices of closed sphere with a unique orbifold point of weight $1/2$ are the only mutation-finite matrices for which unfolding is not constructed yet. Diagrams of these matrices are shown in Fig.~\ref{fig-ex}.

\end{remark}

\subsection{Constructing the diagram of unfolding}
\label{construction}

In this section, given a skew-sym\-met\-riz\-able s-decomposable matrix (or its weighted diagram) admitting an unfolding, we explicitly construct the diagram of the unfolding.

As it was described in previous sections, every unfolding can be understood as a superposition of partial local unfolding and prime unfolding. The way to construct partial local unfolding was explained in Section~\ref{unf-local}. Now we concentrate on prime unfoldings.

Let $\D^w$ be a weighted diagram, $\O^w$ weighted orbifold with all orbifold points of weight $1/2$, and let $T$ be the corresponding triangulation of $\O^w$.
Performing some flips (and mutations), we may assume that each triangle of $T$ (except at most one) containing a pending arc
actually contains two pending arcs (i.e. all s-blocks in the corresponding diagram $\D^w$ are blocks of type $\widetilde V_{12}$).

Given weighted s-decomposable diagram $\D^w$ (or exchange matrix $B$), we call by {\it irregular part} of $\D^w$ (respectively, $B$) the union of s-blocks corresponding to orbifold points of weight $1/2$ (these s-blocks are also called {\it irregular}), and by {\it regular part} the union of all the other blocks. While considering prime unfolding, regular part is obligatory skew-symmetric.

The diagram $\D^w$ consists of its regular part $R$ and $k$ irregular s-blocks $B_0,B_1,\dots,B_k$ of type $\widetilde V_{12}$ with equal weights $w/2=1$ of the dead ends. If the number $k+1$ of irregular blocks is even, we can proceed in the following way.

 Denote by $x_i\in D^w$ the outlet of $B_i$.
To construct $\hat D$ we take two copies $R'$ and $R''$ of the regular part $R$ of $\D$.
 Then we take for each s-block $B_i$ its unfolding $\hat B_i$ shown in the bottom of Fig.~\ref{V12}
and attach the two outlets of $\hat B_i$ to two copies of the vertex $x_i$ (one in $R'$ and another in $R''$), see Fig.~\ref{prime-even}.
In view of the symmetry of $\hat B_i$ there is a unique way to attach $\hat B_i$ to these two vertices. The resulting diagram is an unfolding of $\D^w$.

\begin{figure}[!t]
\begin{center}
\psfrag{x}{\scriptsize $x$}
\psfrag{a}{\small a)\quad}
\psfrag{b}{\small b)}
\psfrag{2}{\scriptsize $2$}
\psfrag{4}{\scriptsize $4$}
\psfrag{B1}{\scriptsize $B_1$}
\psfrag{Bk}{\scriptsize $B_k$}
\epsfig{file=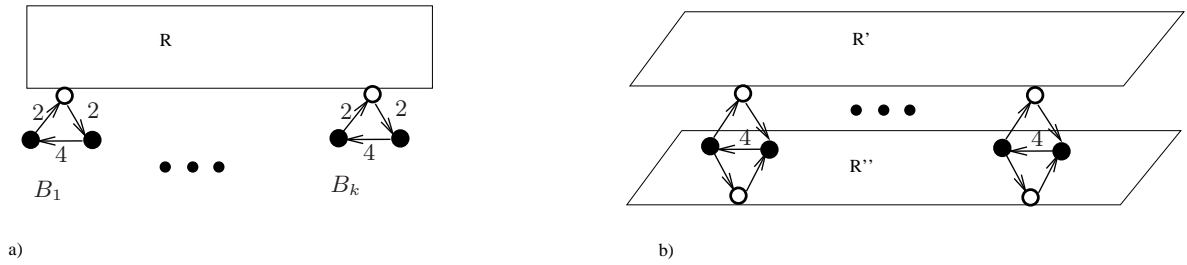,width=0.98\linewidth}
\caption{Prime unfolding for even number of orbifold points: a) diagram $\D$; b) diagram $\hat D$}
\label{prime-even}
\end{center}
\end{figure}

In terms of surfaces, this is equivalent to taking two copies of $\O^w\setminus(\cup_{i=1}^k \Delta_i)$, where $\Delta_i$ are monogons corresponding to irregular blocks of $\D^w$, and connecting every hole obtained from $\Delta_i$ by a cylinder obtained from two copies of $\Delta_i$ (see Fig.~\ref{V12}).

\begin{figure}[!hb]
\begin{center}
\psfrag{x}{\scriptsize $x$}
\psfrag{a}{\small (a)}
\psfrag{b}{\small (b)}
\psfrag{2}{\scriptsize $2$}
\psfrag{c1}{\raisebox{0.5mm}{\scriptsize $c_1$}}
\psfrag{c2}{\raisebox{0.5mm}{\scriptsize $c_2$}}
\psfrag{c1'}{\scriptsize $\hat c_1$}
\psfrag{c2'}{\scriptsize $\hat c_2$}
\psfrag{u}{}
\psfrag{u'}{}
\psfrag{u'}{}
\psfrag{x}{\scriptsize $x$}
\psfrag{x1}{\scriptsize $\hat x'$}
\psfrag{x2}{\scriptsize $\hat x''$}
\epsfig{file=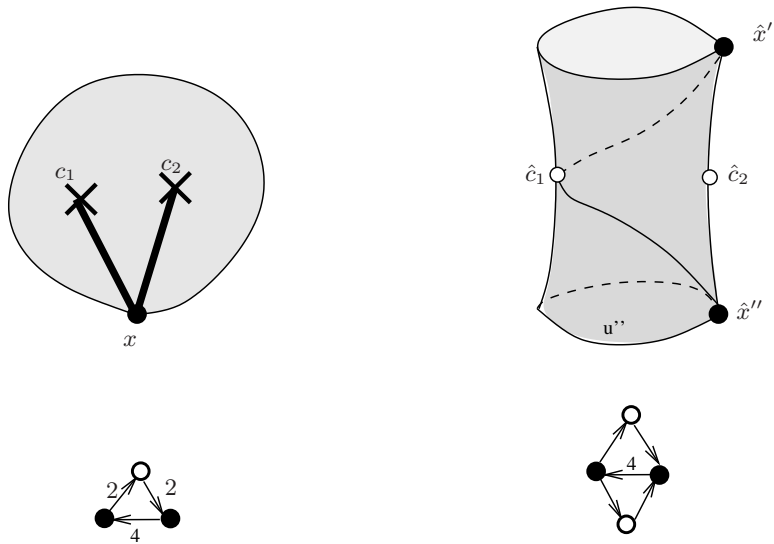,width=0.64\linewidth}
\caption{s-block $\widetilde V_{12}$ (on the left)  and its unfolding (on the right)  with corresponding orbifold and surface}
\label{V12}
\end{center}
\end{figure}

If the number of irregular s-blocks is odd, we assume that $T$ is of the type described in Lemma~\ref{bubbles-odd}, namely, all but one irregular blocks of $\D^w$ correspond to monogons $\Delta_i$, and the last one corresponds to a digon $\Delta_0$ adjacent to a monogon $\Delta_1$.
The s-blocks $B_i$ ($i>0$) are of type $\widetilde V_{12}$, s-block $B_0$ is of type $\widetilde {IV}$. Now we construct the unfolding in two steps.

First, we take two copies $R'$ and $R''$ of the union of the regular part $R$ of $\D$ and $B_0$, and connect $R'$ with $R''$ by unfoldings of blocks $B_1,\dots,B_k$, see Fig.~\ref{first-cover}. In this way we obtain a diagram with two irregular s-blocks. Now we can apply the procedure for even number of s-blocks.

\begin{figure}[!h]
\begin{center}
\psfrag{R}{\small $R$}
\psfrag{R'}{\small $R'$}
\psfrag{R''}{\small $R''$}
\psfrag{c1}{\raisebox{0.5mm}{\scriptsize $c_1$}}
\psfrag{c2}{\raisebox{0.5mm}{\scriptsize $c_2$}}
\psfrag{u}{}
\psfrag{x}{\scriptsize $x_0$}
\psfrag{x'}{\scriptsize $x_0'$}
\psfrag{x1}{\scriptsize $x_1$}
\psfrag{x2}{\scriptsize $x_2$}
\psfrag{a}{\small a)}
\psfrag{b}{\small b)}
\psfrag{2}{\scriptsize $2$}
\psfrag{4}{\scriptsize $4$}
\psfrag{B1}{\scriptsize $B_1$}
\psfrag{Bk}{\scriptsize $B_k$}
\epsfig{file=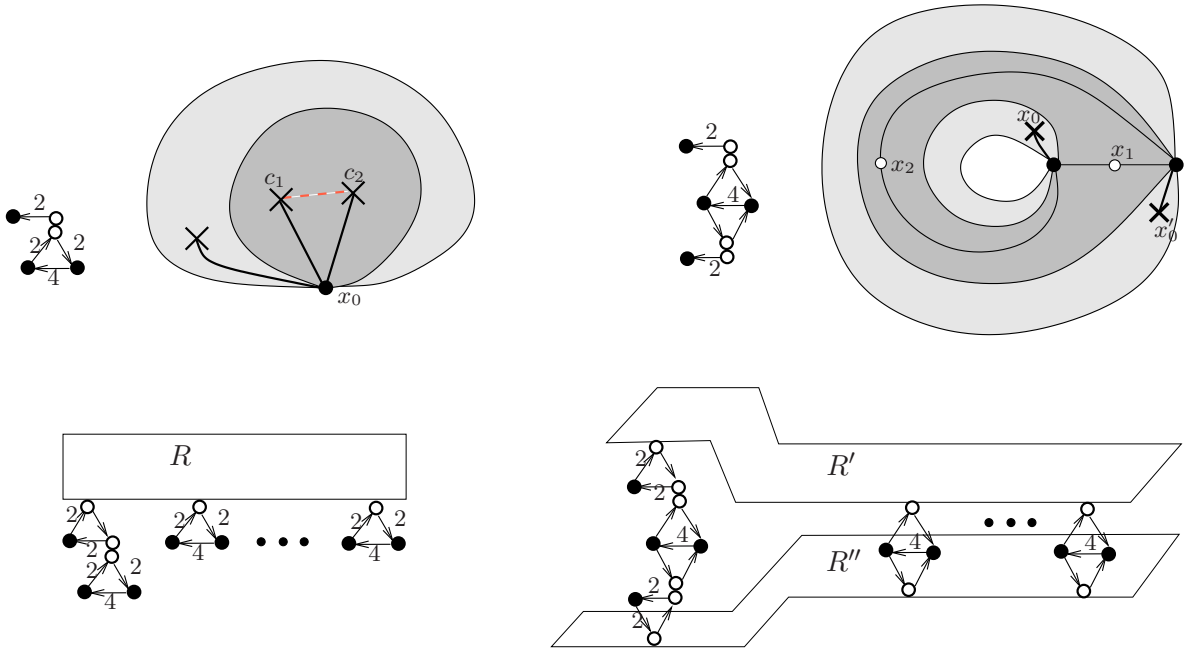,width=0.98\linewidth}
\caption{First step of the construction of prime unfolding for odd number of orbifold points: orbifold and diagram.
A piece of the orbifold with its diagram and the diagram $\D^w$ (all the dead ends of s-blocks are of weight 1) to the left, covering orbifold and $\hat D$ to the right.}
\label{first-cover}
\end{center}
\end{figure}

\section{Positivity}
\label{sec-pos}


In~\cite{FZ1} Fomin and Zelevinsky proved that, given an initial seed of cluster algebra $\A$, any cluster variable of $\A$ can be expressed as a Laurent polynomial of the variables of the initial seed (``Laurent Phenomenon''). The famous positivity conjecture~\cite{FZ1} states that the coefficients of the Laurent polynomials are non-negative integer linear combinations of elements of the coefficient group $\P$.

In~\cite{MSW}, Musiker, Schiffler and Williams show that the positivity conjecture holds for cluster algebras of geometric type originating from surfaces. We will use the unfolding construction to extend this result to algebras with s-decomposable skew-symmetrizable exchange matrices admitting unfoldings.

\begin{theorem}
\label{thm-pos}
Let $\O^w$ be a weighted orbifold (distinct from a closed sphere with marked points and a unique orbifold point of weight $1/2$), let $\A$ be a corresponding cluster algebra of geometric type, and let ${\mathbf x}$ be an initial cluster. Then the coefficients of the Laurent expansion of every cluster variable of $\A$ are non-negative.

\end{theorem}

To prove the theorem, we note that, by construction of unfolding (Section~\ref{unfolding-fin}), the collections of lambda lengths of arcs of triangulation $T$ of $\O^w$ and covering triangulation $\t T$ of the resulting covering surface $S$ coincide. Moreover, considering initial seed of $\A$ as a collection of lambda lengths of arcs of some triangulation $T$, we see that the collections of initial cluster variables of $\A$ and of its unfolding coincide as well.

Now we apply the positivity result from~\cite{MSW} to the unfolding and use the fact that all cluster variables of $\A$ are cluster variables of the unfolding.

\section{Sign-coherence}
\label{signs}

In~\cite{FZ4}, Fomin and Zelevinsky conjectured that, when starting with a seed with principal coefficients, the coefficient vectors in every other seed have all coordinates either nonpositive or nonnegative. The conjecture is proved by Derksen, Weyman and Zelevinsky~\cite{DWZ} for skew-symmetric case, and by Demonet~\cite{D} for a large class of skew-symmetrizable algebras. In this section, we prove the conjecture for all algebras from orbifolds. 

We reformulate the conjecture in terms of shear coordinates of laminations.     

Let $\h T$ be a tagged triangulation on associated orbifold $\hat\O$, denote the arcs of $\h T$ by $\gamma_i$, $1\le i\le n$. Consider a multi-lamination ${\mathbf L}_{\h T}=(L_{n+1},\dots,L_{2n})$ constructed in the following way: every $L_i$ is a single curve with $i$-th shear coordinate $b_{\gamma_i}(\h T,L_i)$ equal $1$ and all the others being zero (such a curve is easy to draw: it is spiraling into non-special marked ends of an edge, or ends in orbifold point of a pending edge, see Fig.~\ref{elements}). 

\begin{figure}
\begin{center}
\begin{tabular}{cp{4mm}cccp{4mm}ccc}
\psfrag{a}{\small (a)}
\psfrag{b}{\small (b)}
\psfrag{c}{\small (c)}
\epsfig{file=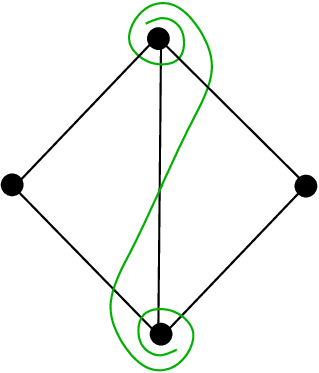,width=0.15\linewidth}&&
\epsfig{file=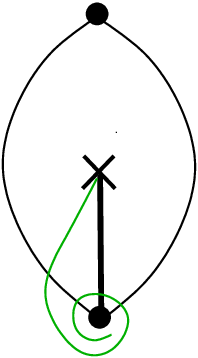,width=0.098\linewidth}&
\epsfig{file=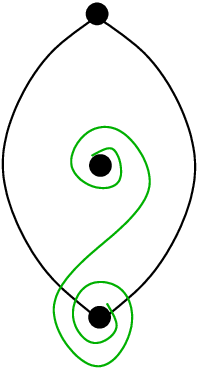,width=0.095\linewidth}&
\epsfig{file=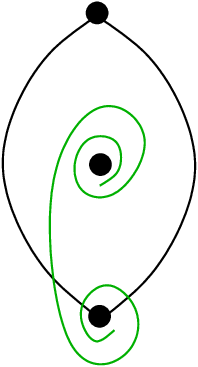,width=0.095\linewidth}&&
\raisebox{-2mm}{\epsfig{file=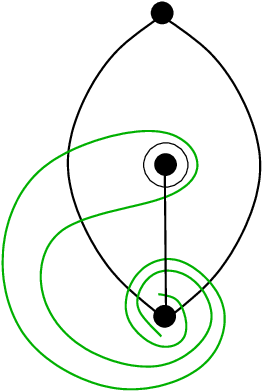,width=0.128\linewidth}}&
\epsfig{file=pic/elem-p.eps,width=0.095\linewidth}&
\epsfig{file=pic/elem-pp.eps,width=0.095\linewidth}\\
\multicolumn{1}{c}{(a)}&&
\multicolumn{3}{c}{(b)}&&
\multicolumn{3}{c}{(c)}
\end{tabular}
\caption{Construction of laminations ${\mathbf L}_{\h T}$ and $\t{\mathbf L}_{\t T}$: (a) $L_i$ for ordinary arc $\gamma_i\in\h T$, (b) $L_j$, $\t L_j'$ and $\t L_j''$ for pending arc $\gamma_j\in\h T$, (c) $L_k$, $\t L_k'$ and $\t L_k''$ for double arc $\gamma_k\in\h T$.} 
\label{elements}
\end{center}
\end{figure}

\begin{theorem}
\label{coherence}
Let $\h T^0$ be arbitrary tagged triangulation of $\hat\O$. For every arc $\gamma$ of $\h T^0$ the shear coordinates $b_\gamma(\h T^0,L_i)$, $n+1\le i\le 2n$, are either all nonpositive or all nonnegative.
\end{theorem}

\begin{proof}
Proceeding as in Definition~\ref{def-shear}, we construct a surface $\t\O$ by replacing every orbifold point by a puncture, and triangulations $\t T$ and $\t T^0$ by replacing every pending or double arc $\gamma_i$ in $\h T$ and $\h T^0$ by conjugate pair of arcs $\t\gamma_i'$ and $\t\gamma_i''$ (we also denote by $\t\gamma_i$ the image of $\gamma_i$ on $\t\O$). We also define new multi-lamination $\t {\mathbf L}=\t {\mathbf L}_{\t T}$ consisting of images $\t L_i$ of $L_i$ if $\gamma_i$ is an ordinary arc, or pairs $\t L_i',\t L_i''$ if $\gamma_i$ is pending or double arc. Here $b_{\t\gamma_i}(\t T,\t L_i)=1$ ($b_{\t\gamma_i'}(\t T,\t L_i')=1$ or $b_{\t\gamma_i''}(\t T,\t L_i'')=1$ respectively), and all the other shear coordinates are zeros). 

The triangulation $\t T^0$ can be obtained from $\t T$ by exactly the same sequence of flips required for obtaining $\h T^0$ from $\h T$ (applying composite flips in $\t\gamma_i'$ and  $\t\gamma_i''$ in case of $\gamma_i$ being double or pending). 
According to sign-coherence of $\cc$-vectors in skew-symmetric case, we conclude that for given $\t\gamma\in\t T^0$ the shear coordinates $b_{\t\gamma}(\t T^0,\t L_i)$ (or $b_{\t\gamma}(\t T^0,\t L_i')$ and $b_{\t\gamma}(\t T^0,\t L_i'')$ if $\gamma_i$ is pending or double arc) are either all nonpositive or all nonnegative. We will now compare the signs of shear coordinates of laminations ${\mathbf L}_{\h T}$ and $\t {\mathbf L}_{\t T}$ with respect to triangulations $\h T^0$ and $\t T^0$ respectively. 

First, assume $\gamma\in\h T^0$ is an ordinary arc. According to Definition~\ref{def-shear}, for every ordinary arc $\gamma_i\in\h T$ the coordinates $b_{\t\gamma}(\t T^0,\t L_i)$ and  $b_{\gamma}(\h T^0,L_i)$ are equal. In particular, all the entries $b_{\gamma}(\h T^0,L_i)$ for ordinary arcs $\gamma_i\in\h T$ are of the same sign (of course, some of them may vanish). 

Consider a pending arc $\gamma_j\in\h T$. One of the arcs $\t\gamma_j'$ and $\t\gamma_j''$ (say the former) is tagged plain. Then, by construction of multi-lamination $\t {\mathbf L}$, $b_{\gamma}(\h T^0,L_j)=b_{\t\gamma}(\t T^0,\t L_j')$. Therefore, $b_{\gamma}(\h T^0,L_j)$ has the same sign as all the other $b_{\gamma}(\h T^0,L_i)$ (or zero).

Now consider a double arc $\gamma_k\in\h T$. Then one can see that $$b_{\gamma}(\h T^0,L_k)=b_{\t\gamma}(\t T^0,\t L_k')+b_{\t\gamma}(\t T^0,\t L_k'').$$ Since the signs agree, the sum is of the same sign as well.

Assume now that $\gamma\in\h T^0$ is a pending or double arc. Then $$b_{\gamma}(\h T^0,L_i)=b_{\t\gamma'}(\t T^0,\t L_i)+b_{\t\gamma''}(\t T^0,\t L_i)$$ for ordinary arcs $\gamma_i$, and $$b_{\gamma}(\h T^0,L_j)=b_{\t\gamma'}(\t T^0,\t L_j')+b_{\t\gamma''}(\t T^0,\t L_j')$$ for pending arcs $\gamma_j$, where $\t\gamma'$ is tagged plain. For double arc $\gamma_k$ we have 
$$b_{\gamma}(\h T^0,L_k)=\frac{1}{2}(b_{\t\gamma'}(\t T^0,\t L_k')+b_{\t\gamma'}(\t T^0,\t L_k'')+b_{\t\gamma''}(\t T^0,\t L_k')+b_{\t\gamma''}(\t T^0,\t L_k'')).$$

Denote by $\t L$ one of the laminations $L_i$, $\t L_j'$, $\t L_k'$ or $\t L_k''$. One can note that the difference $|b_{\t\gamma'}(\t T^0,\t L)-b_{\t\gamma''}(\t T^0,\t L)|$ does not exceed one, cf. proof of Lemma~\ref{ex-uniq-prime}. Moreover, this difference is not zero only if the lamination $\t L$ is spiraling into the end of $\t\gamma'$ and $\t\gamma''$ where their tags are distinct. Thus, this end should originate from a special marked point on $\hat \O$ or from an orbifold point. By Definition~\ref{def-lam}, there are no laminations spiraling into special marked point, and by construction of laminations $L_j$, for every orbifold point exactly one of $L_j$ ends in this point. This implies that the columns of coordinates for $\t\gamma'$ and $\t\gamma''$ differ in two places only: for $\t L_j'$ and for $\t L_j''$, and the difference in the values in two columns is equal to one. Hence, both columns have the same sign, so the sum of them has the same sign as well.

\end{proof}

\end{document}